\documentclass[12pt]{article}

\usepackage{times}

\usepackage{enumitem}
\usepackage[utf8]{inputenc}
\usepackage{commath}
\usepackage{amsthm, amscd}
\usepackage{amsmath}
\usepackage{amssymb}
\usepackage{cite}
\usepackage{setspace}
\usepackage{ stmaryrd }
\usepackage{tikz-cd}

\usepackage{tocloft}
\usepackage{sectsty}
\usepackage{xparse}

\newcommand*\tasklabelformat[1]{#1)}

\usepackage{titlesec}

\usepackage{tasks}[newest]
\settasks{
  label = \alph* ,
  label-format = \tasklabelformat ,
  label-width  = 12pt
}

\usepackage{bigints}
\usepackage{comment}
\usepackage{mathtools}
\usepackage{mathrsfs}
\usepackage{fancyhdr}

\allowdisplaybreaks[4]

\numberwithin{equation}{section}
\usepackage{etoolbox}
\patchcmd{\thebibliography}
  {\settowidth}
  {\setlength{\itemsep}{0pt plus -10pt}\settowidth}
  {}{}
\apptocmd{\thebibliography}
  {
  }
  {}{}
\makeatletter
\newtheorem*{rep@theorem}{\rep@title}
\newcommand{\newreptheorem}[2]{%
\newenvironment{rep#1}[1]{%
 \def\rep@title{#2 \ref{##1}}%
 \begin{rep@theorem}}%
 {\end{rep@theorem}}}
\makeatother

\theoremstyle{theorem}

\newreptheorem{theorem}{Theorem}
\newtheorem{thm}{Theorem}[section]
\newtheorem*{thm*}{Theorem}
\theoremstyle{definition}
\newtheorem{prop}[thm]{Proposition}
\newtheorem*{prop*}{Proposition}
\newtheorem{defn}[thm]{Definition}
\newtheorem{lem}[thm]{Lemma}
\newtheorem{cor}[thm]{Corollary}
\newtheorem*{cor*}{Corollary}
\theoremstyle{remark}

\newtheorem{rem}[thm]{Remark}

\title{\vspace*{-1.5cm} Submultiplicative norms and filtrations 
on section rings}

\author
{Siarhei Finski
}

\date{}

\usepackage[%
    left=1in,%
    right=1in,%
    top=1.1in,%
    bottom=0.8in,%
    paperheight=11in,%
    paperwidth=8.5in%
]{geometry}

\newcommand{\imun} {\sqrt{-1}}

\newcommand{\res}{{\rm{Res}}}

\newcommand{\conv}{{\rm{conv}}}
\newcommand{\sym}{{\rm{Sym}}}

\newcommand{\comp}{\mathbb{C}}
\newcommand{\real}{\mathbb{R}}

\newcommand{\nat}{\mathbb{N}}
\newcommand{\integ}{\mathbb{Z}}

\newcommand{\enmr}[1]{\text{End}{(#1)}}
\newcommand{\Conv}[1]{{\rm{Conv}}{(#1)}}

\newcommand{\ccal}{\mathscr{C}}

\newcommand{\dbar}{ \overline{\partial} }

\makeatletter
\DeclareFontFamily{OMX}{MnSymbolE}{}
\DeclareSymbolFont{MnLargeSymbols}{OMX}{MnSymbolE}{m}{n}
\SetSymbolFont{MnLargeSymbols}{bold}{OMX}{MnSymbolE}{b}{n}
\DeclareFontShape{OMX}{MnSymbolE}{m}{n}{
    <-6>  MnSymbolE5
   <6-7>  MnSymbolE6
   <7-8>  MnSymbolE7
   <8-9>  MnSymbolE8
   <9-10> MnSymbolE9
  <10-12> MnSymbolE10
  <12->   MnSymbolE12
}{}
\DeclareFontShape{OMX}{MnSymbolE}{b}{n}{
    <-6>  MnSymbolE-Bold5
   <6-7>  MnSymbolE-Bold6
   <7-8>  MnSymbolE-Bold7
   <8-9>  MnSymbolE-Bold8
   <9-10> MnSymbolE-Bold9
  <10-12> MnSymbolE-Bold10
  <12->   MnSymbolE-Bold12
}{}

\let\llangle\@undefined
\let\rrangle\@undefined
\DeclareMathDelimiter{\llangle}{\mathopen}%
                     {MnLargeSymbols}{'164}{MnLargeSymbols}{'164}
\DeclareMathDelimiter{\rrangle}{\mathclose}%
                     {MnLargeSymbols}{'171}{MnLargeSymbols}{'171}
\makeatother

\newenvironment{sciabstract}{}

\cftsetindents{section}{0em}{1.7em}

\sectionfont{\large}

\titlespacing*{\section}
{0pt}{5pt}{5pt}

\begin{document}

\maketitle 

\vspace*{-0.7cm}

{\centering \small \textit{Dedicated to Xiaonan Ma on the occasion of his 50\textsuperscript{th} birthday}\par}

\vspace*{0.3cm}

\begin{sciabstract}
  \textbf{Abstract.}
 We show that submultiplicative norms on section rings of polarised projective manifolds are asymptotically equivalent to sup-norms associated with metrics on the polarisation.
 We then discuss some applications to the spectral theory of submultiplicative filtrations, the asymptotic study of the Narasimhan-Simha pseudonorms, and holomorphic extension theorem. 
 As an unexpected byproduct, we show that injective and projective tensor norms on symmetric algebras of finite dimensional complex normed vector spaces are asymptotically equivalent. 
\end{sciabstract}

\pagestyle{fancy}
\lhead{}
\chead{Submultiplicative norms and filtrations on section rings}
\rhead{\thepage}
\cfoot{}


\newcommand{\Addresses}{{
  \bigskip
  \footnotesize
  \noindent \textsc{Siarhei Finski, CNRS-CMLS, École Polytechnique F-91128 Palaiseau Cedex, France.}\par\nopagebreak
  \noindent  \textit{E-mail }: \texttt{finski.siarhei@gmail.com}.
}} 

\vspace*{0.25cm}

\par\noindent\rule{1.25em}{0.4pt} \textbf{Table of contents} \hrulefill

\vspace*{-1.5cm}

\tableofcontents

\vspace*{-0.2cm}

\noindent \hrulefill


\section{Introduction}\label{sect_intro}
	For a holomorphic line bundle $L$ over a compact complex manifold $X$, we define the \textit{section ring}
	\begin{equation}
		R(X, L) := \oplus_{k = 0}^{\infty} H^0(X, L^{\otimes k}).
	\end{equation}
	A graded norm $N = \sum N_k$, $N_k := \| \cdot \|_k$, over $R(X, L)$ is called \textit{submultiplicative} if for any $k, l \in \nat^*$, $f \in H^0(X, L^{\otimes k})$, $g \in H^0(X, L^{\otimes l})$, we have
	\begin{equation}\label{eq_subm_s_ring}
		\| f \cdot g \|_{k + l} \leq 
		\| f \|_k \cdot
		\| g \|_l.
	\end{equation}
	\par 
	\begin{sloppypar}
	As a basic example, any bounded metric $h^L$ on $L$ induces the sequence of sup-norms ${\rm{Ban}}^{\infty}_k(h^L) := \| \cdot \|_{L^{\infty}_{k}(X, h^L)}$ over $H^0(X, L^{\otimes k})$, defined for $f \in H^0(X, L^{\otimes k})$ as follows
	\begin{equation}
		\| f \|_{L^{\infty}_{k}(X, h^L)} = \sup_{x \in X} |f(x)|_{h^L}.
	\end{equation}
	The associated graded norm ${\rm{Ban}}^{\infty}(h^L) = \sum {\rm{Ban}}^{\infty}_k(h^L)$ is clearly submultiplicative (${\rm{Ban}}$ here stands for “Banach").
	The main goal of this article is to prove that under some mild assumptions on $L$ and $N$, asymptotically, these are the only possible examples.
	\end{sloppypar}
	\par 
	More precisely, we say that two graded norms $N = \sum N_k$, $N' = \sum N_k'$ over $R(X, L)$ are \textit{equivalent} ($N \sim N'$) if the multiplicative gap between the graded pieces, $N_k$ and $N_k'$, is subexponential. This means that for any $\epsilon > 0$, there is $k_0 \in \nat^*$, such that for any $k \geq k_0$, we have
	\begin{equation}\label{defn_equiv_rel}
		\exp(-\epsilon k) \cdot N_k
		\leq
		N_k'
		\leq
		\exp(\epsilon k) \cdot N_k.
	\end{equation}
	\par 
	We now assume that $L$ is ample. 
	Then for any $k \in \nat^*$, such that $L^{\otimes k}$ is very ample, any norm $N_k$ on $H^0(X, L^{\otimes k})$ induces the Fubini-Study metric $FS(N_k)$ on $L^{\otimes k}$ through the associated Kodaira embedding, see (\ref{eq_fs_defn}).
	By Fekete's lemma, for any submultiplicative norm $N$, the sequence of metrics $FS(N_k)^{\frac{1}{k}}$ converges, as $k \to \infty$, to a (possibly only bounded from above and even null) metric on $L$, which we denote by $FS(N)$, cf. Lemma \ref{lem_fs_dini}. 
	We can now state our first main result.
	\begin{thm}\label{thm_char}
		Assume that a graded norm $N = \sum N_k$ over the section ring $R(X, L)$ of an ample line bundle $L$ is submultiplicative and $FS(N)$ is continuous.
		Then 
		\begin{equation}\label{eq_char}
			N \sim {\rm{Ban}}^{\infty}(FS(N)).
		\end{equation}
	\end{thm}
	\begin{rem}
		a)
		The continuity of $FS(N)$ without submultiplicativity of $N$ do not determine the equivalence class of $N$, see Proposition \ref{prop_ex_dp_nnwork} or \cite[Proposition 4.16]{FinSecRing} for examples.
		\par 
		b)
		Following \cite[Definition 1.3]{FinSecRing}, one can formulate a purely algebraic criterion for a norm $N$ to be such that $FS(N)$ is continuous.
	\end{rem}
	\par
	In some of our applications, submultiplicative norms $N$ with non-continuous $FS(N)$ arise naturally, see Remark \ref{rem_ex_non_cont_ray} for an example arising from submultiplicative filtrations. 
	To study them, we define a weaker equivalence relation on the set of graded norms.
	Let $\lambda_1, \cdots, \lambda_r$ be the ordered logarithmic relative spectrum between two norms $N, N'$ on a finite dimensional complex vector space $V$, $\dim V = r$, see Section \ref{sect_append} for the definition.
	For $p \in [1, + \infty[$, we let
	\begin{equation}\label{eq_dp_defn_norms}
		d_p(N, N') := \sqrt[p]{\frac{\sum_{i = 1}^{r} |\lambda_i|^p}{r}},
		\qquad 
		d_{+ \infty}(N, N') := \max \big\{ |\lambda_1|, |\lambda_r| \big\},
	\end{equation}
	We say that graded norms $N = \sum N_k$ and $N' = \sum N_k'$ are $p$-\textit{equivalent} ($N \sim_p N'$) if 
 	\begin{equation}\label{defn_equiv_relp}
		\frac{1}{k} d_p(N_k, N_k') \to 0, \qquad \text{as } k \to \infty.
	\end{equation}
	We show in Section \ref{sect_append} that $\sim_p$, $p \in [1, +\infty]$, is an equivalence relation and $\sim$ equals $\sim_{+ \infty}$. 
	\par 
	A graded norm $N$ on $R(X, L)$ is called \textit{bounded} if $N \geq {\rm{Ban}}^{\infty}(h^L)$ for a certain smooth metric $h^L$ on $L$.
	For submultiplicative $N$, it is equivalent to the boundness of $FS(N)$, see (\ref{eq_norm_with_fs_compar}).
	We can now state our second main result.
	\begin{thm}\label{thm_char2}
		Assume that a graded norm $N = \sum N_k$ over the section ring $R(X, L)$ of an ample line bundle $L$ is submultiplicative and bounded.
		Then for any $p \in [1, + \infty[$, we have
		\begin{equation}\label{eq_char2}
			N \sim_p {\rm{Ban}}^{\infty}(FS(N)).
		\end{equation}
	\end{thm}
	\begin{rem}\label{rem_char2}
		a) For $p = + \infty$, the analogous statement fails, see Proposition \ref{prop_example}.
		In particular, the continuity assumption from Theorem \ref{thm_char} cannot be replaced by the boundness assumption.
		\par 
		b) 
		In non-Archimedean setting, where a submultiplicative norm is replaced by a submultiplicative filtration, a result analogous to Theorem \ref{thm_char} is called valuation theorem, and it was established by Rees, \cite{ReesValI}, \cite{ReesValII}, cf. also \cite[\S 4.1, \S 5.3]{ReesBook}, in the setting generalizing finitely generated filtrations on section rings. 
		Recent works of Boucksom-Jonsson \cite[Theorems D and 2.26]{BouckJohn21}, cf. also Reboulet \cite{Reboulet21}, further extended the valuation theorem in the setting of general bounded submultiplicative filtrations and in realms of Theorem \ref{thm_char2}. See Section \ref{sect_homog} for details.
	\end{rem}
	\par 
	It is natural to ask if $FS(N)$ is the unique metric which can be put on the right-hand side of (\ref{eq_char}) and (\ref{eq_char2}).
	While it is not the case if we are allowed to consider arbitrary metrics on the line bundle, see Theorem \ref{thm_bouck_erikss}, it becomes true if we restrict our attention to certain subclasses of them.
	To explain this in details, recall that a \textit{metric with plurisubharmonic (or psh) weight} is the (singular) metric $h^L$ on a holomorphic line bundle $L$ such that for any local holomorphic frame $\sigma$ of $L$, $- \log |\sigma|_{h^L}$ is psh.
	In what follows, for brevity, we call such metrics \textit{psh} metrics.
	A line bundle is called \textit{pseudoeffective} (or \textit{psef}) if it carries a psh metric.
	\begin{defn}\label{defn_regul_above}
		We say that a bounded psh metric $h^L$ is \textit{regularizable from above} if there is a decreasing sequence of continuous psh metrics $h^L_i$, $i \in \nat$, converging to $h^L$ almost everywhere.
	\end{defn}
	\begin{rem}\label{rem_regul_above}
		a) In \cite[Theorem 2]{BedfordRegulBelow} Bedford-Taylor described in a local setting regularizable from above psh metrics as those having pluripolar discontinuity set. 
		\par 
		b) According to Demailly's regularization theorem, see \cite{Dem82}, \cite{DemRegul}, on ample line bundle, any psh metric is regularizable from below, meaning that there is an increasing sequence of smooth positive metrics $h^L_i$, $i \in \nat$, converging pointwise to $h^L$, cf. \cite[Theorem 8.1]{GuedZeriGeomAnal}. 
	\end{rem}
	We will now fix a bounded submultiplicative norm $N$ on the section ring.
	As we explain in Section \ref{sect_append}, Theorem \ref{thm_char2} tells us that for any $p \in [1, +\infty[$, the lower semi-continuous regularization, $FS(N)_*$, of $FS(N)$ is the only regularizable from above psh metric for which the sup-norm lies in $\sim_p$-equivalence class of the fixed norm.
	Moreover, if $FS(N)$ is continuous, then, as we explain in Section \ref{sect_append}, Theorem \ref{thm_char} tells us that $FS(N)$ is the only continuous psh metric for which the sup-norm lies in $\sim$-equivalence class of the fixed norm.
	\par 
	We will now comment on the proofs of Theorems \ref{thm_char} and \ref{thm_char2}.
	The core of the argument is based on an interpretation of the submultiplicativity condition in terms of projective tensor norms, see (\ref{eq_reform_subm_cond}).
	We then use the techniques from \cite{FinSecRing} to reduce the proofs to the special case when $X$ is a projective space and $L$ is the hyperplane bundle.
	In this setting, the above statements are essentially equivalent to showing that injective and projective tensor norms on symmetric algebras of finite dimensional complex normed vector spaces are asymptotically equivalent, see Theorem \ref{thm_sym_equiv}.
	Surprisingly, our proof of this functional-analytic statement uses tools from complex geometry, as Ohsawa-Takegoshi extension theorem.
	Remark also that in full tensor algebras the projective and injective tensor norms are essentially never equivalent by a result of Pisier \cite{PisierDecart}, see Remark \ref{rem_sym_equiv}a).
	\begin{sloppypar}
	We now describe some applications of Theorems \ref{thm_char} and \ref{thm_char2}.
	We fix a compact complex manifold $X$ of dimension $n$ and denote by $K_X := \Lambda^{n} T^{(1, 0)*}X$ its canonical line bundle.
	Narasimhan-Simha in \cite{NarSimh} defined pseudonorms $\mathcal{NS}_k := \| \cdot \|^{\mathcal{NS}}_{k}$, $k \in \nat^*$, over the vector space of $k$-th \textit{pluricanonical sections}, $f \in H^0(X, K_X^k)$, as 
	\begin{equation}\label{eq_ns_defn}
		\| f \|^{\mathcal{NS}}_{k}
		:=
		\Big( \int_X  \big( (- \imun)^{k(n^2 + 2n)} \cdot f \wedge \overline{f} \big)^{\frac{1}{k}} \Big)^{k}.
	\end{equation}
	\end{sloppypar}
	\par 
	Remark that the sequence of pseudonorms $\mathcal{NS}_k$, $k \in \nat^*$, is defined without the use of any fixed metric on $K_X$.
	In particular, it depends only on the complex structure of $X$. 
	Even more, it is a \textit{birational invariant}, as birational equivalence between two complex manifolds $X$ and $Y$ induces the isometry with respect to $\mathcal{NS}_k$ between $H^0(X, K_X^k)$ and $H^0(Y, K_Y^k)$ for any $k \in \nat^*$, cf. \cite{NarSimh}.
	\par 
	In Section \ref{sec_klt}, we show that Narasimhan-Simha pseudonorms are submultiplicative and, as a consequence of this, we provide an application of Theorem \ref{thm_char} about their asymptotic structure.
	\par 
	In Section \ref{sect_hol_ext_pp}, we observe that the quotients of submultiplicative norms are submultiplicative.
	As a consequence of this, we then provide an application of Theorem \ref{thm_char2} to the holomorphic extension problem for non-regular metrics.
	\par 
	For a further application, let us recall the notion of submultiplicative filtrations.
	Recall that a \textit{decreasing $\real$-filtration} $\mathcal{F}$ of a vector space $V$ is a map from $\real$ to vector subspaces of $V$, $t \mapsto \mathcal{F}^t V$, verifying $\mathcal{F}^t V \subset \mathcal{F}^s V$ for $t > s$, and such that $\mathcal{F}^t V  = V$ for sufficiently small $t$ and $\mathcal{F}^t V = \{0\}$ for sufficiently big $t$.
	We say that $\mathcal{F}$ is \textit{graded} if it respects the grading of $V$.
	It is \textit{left-continuous} if for any $t \in \real$, there is $\epsilon_0 > 0$, such that $\mathcal{F}^t V = \mathcal{F}^{t - \epsilon} V $ for any $0 < \epsilon < \epsilon_0$.
	All filtrations in this article are assumed to be decreasing left-continuous and graded if applicable.
	\par 
	A filtration $\mathcal{F}$ on $R(X, L)$ is called \textit{submultiplicative} 
	if for any $t, s \in \real$, $k, l \in \nat$, we have 
	\begin{equation}
		\mathcal{F}^t H^0(X, L^{\otimes k}) \cdot \mathcal{F}^s H^0(X, L^{\otimes l}) \subset \mathcal{F}^{t + s} H^0(X, L^{\otimes (k + l)}).
	\end{equation}
	We say that $\mathcal{F}$ is \textit{bounded} if there is $C > 0$ such that for any $k \in \nat^*$, $\mathcal{F}^{ C k} H^0(X, L^{\otimes k}) = \{0\}$.
	Studying asymptotic properties of bounded submultiplicative filtrations is related to K-stability due to their relation with test configurations, see Section \ref{sect_filt}, cf. \cite{NystOkounTest}, \cite{SzekeTestConf}, \cite{BouckJohn21}.
	\par 
	We will show in Section \ref{sect_pf_filt} that any submultiplicative filtration induces a ray of submultiplicative norms.
	Using this, we then establish an application of Theorem \ref{thm_char2} about the relation between the spectral properties of bounded submultiplicative filtrations and associated geodesic rays, generalizing previous works of Witt Nystr{\"o}m \cite{NystOkounTest} and Hisamoto \cite{HisamSpecMeas}
	\par
	This paper is organized as follows. 
	In Section \ref{sect_prel}, we recall the preliminaries.
	In Section \ref{sect_class_sn}, we prove Theorems \ref{thm_char}, \ref{thm_char2} modulo a certain functional-analytic statement, to which Section \ref{sec_norm_poly_big} is dedicated.
	In Section \ref{sect_gr_jm}, we establish some applications.
	\par 
	\textbf{Notation}. 
	A sequence of numbers (resp. positive numbers) $a_k$, $k \in \nat$, is called \textit{subadditive} (resp. \textit{submultiplicative} or \textit{superadditive}) if $a_{k + l} \leq a_k + a_l$ (resp. $a_{k + l} \leq a_k a_l$ or $a_{k + l} \geq a_k + a_l$) for any $k, l \in \nat$.
	We extend these notions for sequences of functions and for metrics on powers of a line bundle.
	A sequence of positive real numbers $a_k$ is called \textit{subexponential} if for any $\epsilon > 0$, $\exp(- \epsilon k) \leq a_k \leq \exp( \epsilon k)$ for $k$ big enough.
	\par 
	Over $\comp^l$, $l \in \nat^*$, we denote by $\textit{l}_1 = \| \cdot \|_1$ and $\textit{l}_{\infty}= \| \cdot \|_{\infty}$ the norms, defined for $x = (x_1, \cdots, x_l)$ as follows $\| x \|_1 = \sum |x_i|$, $\| x \|_{\infty} = \max |x_i|$.
	By a \textit{seminorm} over a finite dimensional vector space $V$, we mean a non-negative absolutely homogeneous convex function over $V$. 
	By a \textit{pseudonorm} over a finite dimensional vector space $V$, we mean a non-negative absolutely homogeneous continuous function over $V$, which is equal to $0$ only at $0 \in V$.
	Clearly, any pseudonorm defines a dual pseudonorm on $V^*$ by the usual definition. 
	For any pseudonorm $N_V$ over a finite dimensional vector space $V$, one can associate the \textit{convex hull} norm $\Conv{N_V}$ on $V$ in such a way that the unit ball of $\Conv{N_V}$ is the convex hull of the unit ball of $N_V$.
	By a \textit{multiplicative gap} between the pseudonorms $N_1$, $N_2$ on a vector space $V$, we mean the minimal constant $C > 0$, such that both inequalities $N_1 \leq C N_2$ and $N_2 \leq C N_1$ are satisfied.
	\par 
	Recall that a norm $N_V = \| \cdot \|_V$ on a finite dimensional vector space $V$ naturally induces the norm $\| \cdot \|_Q := [N_V]$ on any quotient $Q$, $\pi : V \to Q$ of $V$ as follows
	\begin{equation}\label{eq_defn_quot_norm}
		\| f \|_Q
		:=
		\inf \Big \{
		 \| g \|_V
		 :
		 \quad
		 g \in V, 
		 \pi(g) = f
		\Big\},
		\qquad f \in Q.
	\end{equation}
	\par 
	We denote by $| \cdot |_{h^L}$ the induced pointwise norm on $L$ induced by a metric $h^L$.
	We sometimes denote ${\rm{Ban}}^{\infty}(h^L)$ by ${\rm{Ban}}^{\infty}_{X}(h^L)$ to underline the dependence on the ambient manifold, $X$.
	\par 
	Throughout the whole article $L$ is assumed to be ample and $X$ is a compact complex manifold.	
	\par 
	For $0 \geq a < b$, we will use the following notation $\mathbb{D}_{a, b} = \{ z \in \comp : a < |z| < b \}$, $\mathbb{D}_b = \{ z \in \comp : |z| < b \}$, $\mathbb{D} := \mathbb{D}_1$, $\comp^* := \comp \setminus \{0\}$.
	We also denote by $\pi$ the projection $\pi : X \times \mathbb{D} \to \mathbb{D}$ to the second factor and use the similar notations for all of the above spaces.
	\par 
	For a given function $f$ on a topological space, we denote by $f^*$ (resp. $f_*$) the upper (resp. lower) semi-continuous regularization of $f$. The same notations are used for metrics on line bundles. 
	\par 
	\par 
	\textbf{Acknowledgement}. 
	I would like to thank Hajime Tsuji and Mihai Păun for discussions related to the Narasimhan-Simha pseudonorms during Hayama Symposium 2022 and Oberwolfach workshop \# 2236 respectively, as well as the organizers of these meetings for making the discussions possible.
	I warmly thank Sébastien Boucksom for drawing my attention to many relevant works and for all the stimulating discussions on complex and non-Archimedean pluripotential theory, leading to a significant improvement of the first version of this article.
	I also thank Rémi Reboulet for several useful discussions and Tamás Darvas for his comments and for pointing out \cite{DarvXiaTest}.
	It is, finally, a pleasure to thank Xiaonan Ma for all of his help over the years.
	
	\section{Preliminaries}\label{sect_prel}
	This section is organized as follows.
	In Section \ref{sec_fs_bdem}, we recall the definition of the Fubini-Study metric.
	In Section \ref{sec_klt}, we give an application of Theorem \ref{thm_char} to the asymptotic study of the Narasimhan-Simha pseudonorms.
	In Section \ref{sect_append}, we prove some basic results about the rays of norms on finite dimensional vector spaces.
	In Section \ref{sect_pp_thr}, we recall the basics of pluripotential theory and related quantization results.
	
	\subsection{Fubini-Study metrics associated to pseudonorms on cohomology}\label{sec_fs_bdem}
	
	\begin{sloppypar}
	In this section we recall the definition of the Fubini-Study operator and its positivity properties.
	\par 
	We fix an ample line bundle $L$ over a compact complex manifold $X$.
	For $k \in \nat$ so that $L^{\otimes k}$ is very ample, Fubini-Study operator associates for any norm $N_k = \| \cdot \|_k$ on $H^0(X, L^{\otimes k})$, a continuous metric $FS(N_k)$ on $L$, constructed in the following way.
	Consider the Kodaira embedding 
	\begin{equation}\label{eq_kod}
		{\rm{Kod}}_k : X \hookrightarrow \mathbb{P}(H^0(X, L^{\otimes k})^*),
	\end{equation}
	which embeds $X$ in the space of hyperplanes in $H^0(X, L^{\otimes k})$.
	The evaluation maps provide the isomorphism $
		L^{\otimes (-k)} \to {\rm{Kod}}_k^* \mathscr{O}(-1),
	$
	where $\mathscr{O}(-1)$ is the tautological bundle over $\mathbb{P}(H^0(X, L^{\otimes k})^*)$.
	We endow $H^0(X, L^{\otimes k})^*$ with the dual norm $N_k^*$ and induce from it a metric $h^{FS}(N_k)$ on $\mathscr{O}(-1)$ over $\mathbb{P}(H^0(X, L^{\otimes k})^*)$. 
	We define the metric $FS(N_k)$ on $L^{\otimes k}$ as the only metric verifying under the dual of the above isomorphism the identity
	\begin{equation}\label{eq_fs_defn}
		FS(N_k) = {\rm{Kod}}_k^* ( h^{FS}(N_k)^* ).
	\end{equation}
	Sometimes, by abuse of notation, we denote by $FS(N_k)$ the metric $h^{FS}(N_k)^*$ on $\mathscr{O}(1)$ over $\mathbb{P}(H^0(X, L^{\otimes k})^*)$.
	A statement below can be seen as an alternative definition of $FS(N_k)$.
	\end{sloppypar}
	\begin{lem}\label{lem_fs_inf_d}
		For any $x \in X$, $l \in L^{\otimes k}_x$, the following identity takes place
		\begin{equation}\label{eq_fs_norm}
			|l|_{FS(N_k), x}
			=
			\inf_{\substack{s \in H^0(X, L^{\otimes k}) \\ s(x) = l}}
			\| s \|_k.
		\end{equation}
	\end{lem}
	\begin{proof}
		An easy verification, cf. Ma-Marinescu \cite[Theorem 5.1.3]{MaHol}.
	\end{proof}
	The above construction of the Fubini-Study metric works more generally for pseudonorms $N_k$.
	In this case, since the Fubini-Study operator uses the dual of the pseudonorm and double dual of a pseudonorm equals to its convex hull, we clearly have 
	\begin{equation}\label{eq_fs_conv_hh}
		FS(N_k) = FS(N_k^{**}) = FS(\Conv{N_k}).
	\end{equation}
	\par 
	When the norm $N_k$ comes from a Hermitian product on $H^0(X, L^{\otimes k})$, the Fubini-Study construction is standard and explicit evaluation shows that in this case $c_1(\mathscr{O}(-1) , h^{FS}(N_k))$ coincides up to a negative constant with the Kähler form of the Fubini-Study metric on $\mathbb{P}(H^0(X, L^{\otimes k})^*)$ induced by $N_k$.
	In particular, $c_1(\mathscr{O}(-1) , h^{FS}(N_k))$ is a negative $(1, 1)$-form.
	\par 
	Let us now discuss the positivity properties of the metric $FS(N_k)$ for general pseudonorms $N_k$.
	A pseudonorm $N_V := \| \cdot \|_V$ on a vector space $V$ defines a continuous function $F_V : V \to [0, +\infty[$, $v \mapsto \| v \|_V^2$.
	Following Kobayashi's terminology on Finsler metrics, \cite{KobFinNeg}, we say that $N_V$ is \textit{pseudoconvex} if we have $\imun \partial \dbar F_V \geq 0$ in the sense of currents.
	Clearly, if $N_V$ is a norm, the function $F_V$ is convex by triangle inequality and then $N_V$ is trivially pseudoconvex.
	\par 
	Pseudonorms $N_V$ on $V$ are in one-to-one correspondence with metrics $h^{N_V}$ on the tautological line bundle $\mathscr{O}(-1)$ over $\mathbb{P}(V)$.
	According to {\cite[Theorem 4.1]{KobFinNeg}}, pseudoconvexity of $N_V$ is equivalent to the negativity of the $(1, 1)$-current $c_1(\mathscr{O}(-1), h^{N_V})$.
	In particular, for any norm $N_V$ on $V$, we have $c_1(\mathscr{O}(-1), h^{N_V}) \leq 0$ in the sense of currents.
	Hence, from (\ref{eq_fs_defn}) and (\ref{eq_fs_conv_hh}), the (singular) metric $FS(N_k)$ is psh for any pseudonorm $N_k$ on $H^0(X, L^{\otimes k})$.

	\subsection{Asymptotic study of the Narasimhan-Simha pseudonorms}\label{sec_klt}
	The main goal of this section is describe an application of Theorem \ref{thm_char} to the asymptotic study of the Narasimhan-Simha pseudonorms.
	\par 
	Historically, Narasimhan-Simha pseudonorms have been introduced in \cite{NarSimh} to study moduli problems.
	It is known that for canonically polarised manifolds, the isomorphism type of pseudonormed vector space $(H^0(X, K_X^k), \mathcal{NS}_k)$ for sufficiently big and divisible $k$ determines $X$ up to an isomorphism, see Royden \cite[Theorem 1]{RoydenNS} and Chi \cite[Theorem 1.4]{ChiNSBir}.
	In family setting, the study of positivity of related (pseudo)norms is linked to Iitaka conjecture and invariance of plurigenera problem, see Kawamata \cite{KawamataNS}, Berndtsson-P{\u{a}}un \cite{BerndPaun}, P{\u{a}}un-Takayama \cite{PaunTakayama}.
	See also Amini-Nicolussi \cite{AminiNicol1} and Shivaprasad \cite{ShivaNS} for the study of the Narasimhan-Simha pseudonorms and related objects in singular family setting.
	The main goal of this section is, however, the approximate study of these pseudonorms in the semiclassical limit, i.e. when $k \to \infty$.
	\par 
	We first introduce some notations.
	We say that for a normal crossing divisor $\sum D_i$ on $X$, the $\mathbb{Q}$-divisor $\sum d_i D_i$, $d_i \in \mathbb{Q}$, is klt if for any index $i$, we have $d_i < 1$.
	More generally, a $\mathbb{Q}$-divisor $D$ is klt if for a resolution of singularities $\pi: \tilde{X} \to X$ of $|D|$ and the normal crossing $\mathbb{Q}$-divisor $\tilde{D}$, verifying 
	\begin{equation}
		K_{\tilde{X}} + \tilde{D} = \pi^*(K_X + D),
	\end{equation}
	the pair $(\tilde{X}, \tilde{D})$ is klt.
	This definition doesn't depend on the choice of the resolution $\pi$, cf. \cite[Lemma 3.10]{KollarPairs}.
	\par 
	The klt condition can be restated in the following more differential-geometric way.
	Let $r := r_D \in \nat^*$ be the minimal number such that the element $r D$ is a $\mathbb{Z}$-divisor.
	We denote by $h^{r D}$ the \textit{canonical (singular) metric} on the line bundle $\mathscr{O}_X(r D)$, defined as 
	\begin{equation}
		| s |_{h^{r D}}(x) = 1, \quad \text{for } x \notin |D|,
	\end{equation}
	where $s$ is the canonical (meromorphic) section of $\mathscr{O}_X(r D)$.
	According to \cite[Proposition 3.20]{KollarPairs}, the klt condition is equivalent to the fact that
	\begin{equation}\label{eq_klt_ref_int}
		\Big( \frac{h^{r D}}{h^{r D}_{sm}} \Big)^{\frac{1}{r}} \quad \text{is integrable over } X,
	\end{equation}
	where $h^{r D}_{sm}$ is some (hence, any) smooth metric on $\mathscr{O}_X(r D)$.
	\par 
	Let us now recall that psh metrics on log canonical line bundles of klt pairs give rise to positive integrable volume forms.
	More precisely, assume first that $h^K_0$ is a smooth metric on $K_X$.
	We define the positive volume form $dV_{h^K_0}$ by requiring that for any $x \in X$, we have
	\begin{equation}\label{eq_vol_f_defn_sm}
		dV_{h^K_0}(x) 
		=
		(-\imun)^{n^2 + 2n} dz_1 \wedge \cdots \wedge dz_n \wedge d\overline{z}_1 \wedge \cdots \wedge d\overline{z}_n,
	\end{equation}
	where $| dz_1 \wedge \cdots \wedge dz_n |_{h^K_0}(x) = 1$.
	This construction can be extended to psh metrics $h^K$ on $K_X$ by writing $h^K = e^{- \phi} \cdot h^K_0$ for $\phi \in L^1_{{\rm{loc}}}$ and defining $dV_{h^K} := e^{\phi} \cdot dV_{h^K_0}$.
	Clearly, the result doesn't depend on the choice of $h^K_0$.
	The volume form $dV_{h^K}$ is bounded since any quasi-psh function $\phi$ is bounded. 
	It might, nevertheless, vanish, as $\phi$ is allowed to take $- \infty$ values.
	\par 
	Now, more generally a psh metric $h^{K, D}$ on $K_X^r \otimes \mathscr{O}_X(rD)$ defines a singular metric $h^{K, D}_{sing} = \frac{h^{K, D}}{h^{rD}}$ on $K_X^{r}$. 
	Then $h^{K, D}$ defines a (singular) volume form $dV_{h^{K, D}}$ as $dV_{h^{K, D}} = e^{\phi} \cdot dV_{h^K_0}$, where $\phi$ is so that $h^{K, D}_{sing} = e^{- r \phi} \cdot (h^K_0)^{r}$.
	For $(X, D)$ klt, $dV_{h^{K, D}}$ is integrable by (\ref{eq_klt_ref_int}).
	\begin{sloppypar} 
	\par 
	We now fix a pair $(X, \Delta)$ of a manifold $X$ with a $\mathbb{Q}$-divisor $\Delta$ is klt.
	We denote the log canonical $\mathbb{Q}$-line bundle by $K_X(\Delta) := K_X \otimes \mathscr{O}_X(\Delta)$ and let $r := r_{\Delta} \in \nat^*$ be the minimal number such that $r \Delta$ is a $\mathbb{Z}$-divisor.
	Any section $f \in H^0(X, K_X(\Delta)^{kr})$, $k \in \nat^*$, can be then interpreted as a meromorphic section of $K_X^{kr}$.
	The klt condition implies, see (\ref{eq_klt_ref_int}) and after, that the integrand in (\ref{eq_ns_defn}) is finite.
	We denote by $\mathcal{NS}_{kr}^{\Delta} := \| \cdot \|^{\mathcal{NS}, \Delta}_{kr}$ the pseudonorm on $H^0(X, K_X(\Delta)^{kr})$, given by this integral.
	Over \textit{log canonical ring}, $R(X, K_X(\Delta)^r)$, we define the Narasimhan-Simha graded pseudonorm
	\begin{equation}
		\mathcal{NS}^{\Delta} := \sum_{k = 1}^{\infty} \mathcal{NS}_{kr}^{\Delta}.
	\end{equation}
	\par 
	Now, any psh metric $h^{K, \Delta}$ on $K_X(\Delta)^{r}$ induces a volume form (with singularities) on $X$, denoted by $dV_{h^{K, \Delta}}$, see (\ref{eq_vol_f_defn_sm}) for details. 
	If the pair $(X, \Delta)$ is klt, $dV_{h^{K, \Delta}}$ is of finite volume, see (\ref{eq_klt_ref_int}) and after.
	Recall that Tsuji in \cite{TsuCanKah} defined the supercanonical metric $h^{K, \Delta}_{{\rm{can}}}$ on $K_X(\Delta)^{r}$ over klt pairs $(X, \Delta)$ with psef $K_X(\Delta)^r$ through the following envelope construction: for $x \in X$, we let
	\begin{equation}\label{defn_tsu_m}
		h^{K, \Delta}_{{\rm{can}}}(x)
		=
		\inf \Big\{ 
			h^{K, \Delta}(x) : \quad h^{K, \Delta} \text{ is a psh metric on $K_X(\Delta)^{r}$, with } \int_X dV_{h^{K, \Delta}} \leq 1
		\Big\}.
	\end{equation}
	\par 
	Recall that a pair $(X, \Delta)$ is called \textit{log canonically polarised} if $K_X(\Delta)^{r}$ is ample.
	The main result of this section goes as follows.
	\begin{thm}\label{thm_ns_conv}
		For log canonically polarised klt pairs $(X, \Delta)$, the following equivalence of graded norms on the log canonical ring $R(X, K_X(\Delta)^{r})$ holds 
		\begin{equation}
			\Conv{\mathcal{NS}^{\Delta}} \sim {\rm{Ban}}^{\infty}(h^{K, \Delta}_{{\rm{can}}}).
		\end{equation}
	\end{thm}
	\begin{rem}
		Taking convex hull is necessary. In fact, for any $f \in H^0(X, K_X(\Delta)^{kr})$, $k \in \nat^*$, both the Narasimhan-Simha pseudonorms and sup-norms behave multiplicatively on the sequence $f^l$ for $l \in \nat^*$.
		Hence, if the statement would hold without taking the convex hull, it would imply that the Narasimhan-Simha pseudonorms coincide identically with the sup-norm, which is false.
	\end{rem}
	\par 
	Let us now recall a related result about the convergence of Fubini-Study metrics associated with the Narasimhan-Simha pseudonorms.
	Below we use the notations from Section \ref{sect_intro}.
	\end{sloppypar} 
	\begin{thm}[{Berman-Demailly \cite[Proposition 5.19 and Remark 5.23]{BerDem} }]\label{thm_berm_dem}
		For a log canonically polarised klt pair $(X, \Delta)$, the sequence of metrics $FS(\mathcal{NS}_{r k}^{\Delta})^{\frac{1}{k}}$ on $K_X(\Delta)^{r}$ converges uniformly, as $k \to \infty$, to the supercanonical metric, $h^{K, \Delta}_{{\rm{can}}}$.
	\end{thm}
	\begin{rem}
		In \cite{BerDem}, authors assume that $\Delta$ is effective, but this is never used in the proof.
	\end{rem}
	Clearly, from Theorem \ref{thm_bouck_erikss} and (\ref{eq_fs_conv_hh}), Theorem \ref{thm_ns_conv} refines Theorem \ref{thm_berm_dem}.
	Let us now prove Theorem \ref{thm_ns_conv}.
	For this, we first need the following basic fact.
	\begin{lem}\label{lem_conv_subm}
		For any graded submultiplicative pseudonorm $N$, the associate convex hull norm, $\Conv{N}$, is submultiplicative.
	\end{lem}
	\begin{proof}
		It follows directly from the fact that the convex hull norm $\| \cdot \|_V^{\conv}$ of a pseudonorm $\| \cdot \|_V$ can be described as $
			\| v \|_V^{\conv} = \inf \big\{ 
			\sum \| v_i \|_V : \sum v_i = v
			\big\}$.
	\end{proof}
	\begin{sloppypar}
	\begin{proof}[Proof of Theorem \ref{thm_ns_conv}.]
		By Hölder's inequality, $\mathcal{NS}^{\Delta}$ is submultiplicative.
		By the discussion before Theorem \ref{thm_ns_conv}, $h^{K, \Delta}_{{\rm{can}}}$ is continuous and non-null.
		Theorem \ref{thm_ns_conv} now follows from this, Theorems \ref{thm_char}, \ref{thm_berm_dem}, Lemma \ref{lem_conv_subm} and (\ref{eq_fs_conv_hh}).
	\end{proof}
	\end{sloppypar}

	\subsection{Rays of norms and logarithmic relative spectrum}\label{sect_append}
	The main goal of this section is to recall some basic results from the theory of finite dimensional normed vector spaces, emphasizing related metric structures.
	We also give two constructions of rays of norms on a finite dimensional vector space endowed with a filtration and compare them.
	\par 
	Let $N_i = \| \cdot \|_i$, $i = 1, 2$, be two norms on a finite dimensional vector space $V$.
	We define the \textit{logarithmic relative spectrum} of $N_1$ with respect to $N_2$ as a non-increasing sequence $\lambda_j := \lambda_j(N_1, N_2)$, $j = 1, \cdots, \dim V$, defined as follows
	\begin{equation}\label{eq_log_rel_spec}
		\lambda_j
		:=
		\sup_{\substack{W \subset V \\ \dim W = j}} 
		\inf_{w \in W \setminus \{0\}} \log \frac{\| w \|_2}{\| w \|_1}.
	\end{equation}
	Directly from (\ref{eq_log_rel_spec}), for any norms $N_1, N_2, N_3$, $j = 1, \cdots, \dim V$, we have
	\begin{equation}\label{eq_log_rel_spec02}
		\lambda_j(N_1, N_2) + \lambda_{\dim V}(N_2, N_3) \leq \lambda_j(N_1, N_3) \leq \lambda_j(N_1, N_2) + \lambda_1(N_2, N_3),
	\end{equation}
	and whenever $N_1 \leq N_2 \leq N_3$, we have
	\begin{equation}\label{eq_log_rel_spec2}
		\lambda_j(N_1, N_2) \leq \lambda_j(N_1, N_3).
	\end{equation}
	\par 
	When both $N_i$, $i = 1, 2$, are Hermitian norms associated with the scalar products $\langle \cdot, \cdot \rangle_i$, the logarithmic relative spectrum coincides with the logarithm of the spectrum of the transfer map $A \in \enmr{V}$ between $N_1$ and $N_2$, which is the Hermitian operator, verifying $\langle A \cdot, \cdot \rangle_1 = \langle \cdot, \cdot \rangle_2$.
	\par 
	By \cite[Theorem 3.1]{BouckErik21}, the functions $d_p$, $p \in [1, +\infty[$, defined in (\ref{eq_dp_defn_norms}), are such that
	\begin{equation}\label{eq_triangle}
		\text{$d_p$ satisfies the triangle inequality over the space of Hermitian norms.}
	\end{equation}
	Remark also that the John ellipsoid theorem, cf. \cite[p. 27]{PisierBook}, says that for any normed vector space $(V, N_V)$, there is a Hermitian norm $N_V^H$ on $V$, verifying 
	\begin{equation}\label{eq_john_ellips}
		N_V^H \leq N_V \leq \sqrt{\dim V} \cdot N_V^H.
	\end{equation}
	From (\ref{eq_log_rel_spec02}), (\ref{eq_log_rel_spec2}), (\ref{eq_triangle}), (\ref{eq_john_ellips}) and the fact that $\dim H^0(X, L^{\otimes k})$ grow polynomially (hence, subexponentially) in $k \in \nat$, we see that $\sim_p$ for $p \in [1, +\infty[$ is indeed an equivalence relation.
	\par 
	From (\ref{eq_log_rel_spec}), $\max \{ |\lambda_1|, |\lambda_{\dim V}| \} = d_{+ \infty}(N_1, N_2)$ is the minimal constant $C \geq 0$, verifying
	\begin{equation}
		\exp(- C) \cdot N_2 \leq N_1 \leq \exp(C) \cdot N_2.
	\end{equation}
	This clearly proves that $\sim_{+ \infty}$ coincides with the equivalence relation $\sim$.
	\par 
	Recall that in a metric space $(\mathcal{M}, d)$, a curve $\gamma_t \in \mathcal{M}$, $t \in [0, 1]$, is called a geodesic if there is $v \geq 0$, verifying $d(\gamma_a, \gamma_b) = v |a - b|$ for any $a, b \in [0, 1]$.
	\par 
	It is possible to prove that the space of Hermitian norms on $H^0(X, L^{\otimes k})$, endowed with the distance $d_2$, is isometric to the space $SL(\dim H^0(X, L^{\otimes k})) / SU(\dim H^0(X, L^{\otimes k}))$, endowed with the distance coming from the standard $SL(\dim H^0(X, L^{\otimes k}))$-invariant metric, cf. \cite[Theorem 1.1]{DarvLuRub}.
	The later space is known to be of non-positive sectional curvature, see \cite[Theorem XI.8.6]{KobNom2}, and contractible (by Cartan decomposition).
	In particular, by Cartan-Hadamard theorem, it is uniquely geodesic.
	This goes in a sharp contrast with the distance $d_1$, as we see below.
	\par 
	\begin{lem}\label{lem_d_1_ident_metr}
		Assume that Hermitian norms $H_1, H_2, H_3$ on a vector space $V$ are ordered as follows $H_1 \leq H_2 \leq H_3$.
		Then the following identity holds
		\begin{equation}\label{eq_d_1_ident_metr}
			d_1(H_1, H_2) + d_1(H_2, H_3) = d_1(H_1, H_3).
		\end{equation}
	\end{lem}
	\begin{rem}
		Hence, a concatenation of geodesics between $H_1, H_2$ and $H_2, H_3$, ordered as in Lemma \ref{lem_d_1_ident_metr}, is a geodesic.
	\end{rem}
	\begin{proof}
		The result follows from (\ref{eq_triangle}) and Lidskii’s inequality, cf. \cite[Theorem 5.1]{DarvLuRub}.
	\end{proof}
	\par 
	For any $p \in [1, +\infty]$ and graded norms $N = \sum N_k$, $N' = \sum N'_k$ on a section ring $R(X, L)$ of an ample line bundle, we define
	\begin{equation}\label{eq_dp_gr_n_def}
		d_p(N, N')
		:=
		\limsup_{k \to \infty} \frac{1}{k} d_p(N_k, N'_k).
	\end{equation}
	From (\ref{eq_log_rel_spec02}), (\ref{eq_log_rel_spec2}), (\ref{eq_triangle}), (\ref{eq_john_ellips}) and the fact that $\dim H^0(X, L^{\otimes k})$ grow polynomially (hence, subexponentially) in $k \in \nat$, $d_p$ is symmetric and satisfies the triangle inequality.
	\par 
	We say now that graded norms $N, N'$ are \textit{in the same asymptotic class} if $d_1(N, N') < +\infty$ and in (\ref{eq_dp_gr_n_def}), one can put $\lim$ instead of $\limsup$ for $p = 1$.
	\begin{lem}\label{lem_d_1_ident_metr_gr}
		Assume that graded norms $N_1, N_2, N_3$ on $R(X, L)$ are in the same asymptotic class and they are ordered as $N_1 \leq N_2 \leq N_3$.
		Then the following identity holds
		\begin{equation}\label{eq_d_1_ident_metr_gr}
			d_1(N_1, N_2) + d_1(N_2, N_3) = d_1(N_1, N_3).
		\end{equation}
	\end{lem}
	\begin{sloppypar}
	\begin{proof}
		We write $N_i = \sum N_{i, k}$ for $i = 1, 2, 3$, and denote by $N_{i, k}^H$ the Hermitian norm on $H^0(X, L^{\otimes k})$ associated to $N_{i, k}$ as in (\ref{eq_john_ellips}).
		Then for the graded Hermitian norms $N_1^H := \sum \frac{1}{\dim H^0(X, L^{\otimes k})} N_{1, k}^H$, $N_2^H := \sum N_{2, k}^H$, $N_3^H := \sum \dim H^0(X, L^{\otimes k}) \cdot N_{3, k}^H$, by (\ref{eq_log_rel_spec2}), (\ref{eq_john_ellips}) and the fact that $\dim H^0(X, L^{\otimes k})$ grow polynomially (hence, subexponentially) in $k \in \nat$, we have $N_i^H \sim N_i$ for any $i = 1, 2, 3$.
		By (\ref{eq_log_rel_spec02}), we deduce $| d_1(N_{i, k}, N_{j, k}) - d_1(N_{i, k}^H, N_{j, k}^H) | \leq 4 \log \dim H^0(X, L^{\otimes k})$ 
		So we have $d_1(N_i, N_j) = d_1(N_i^H, N_j^H)$ for any $i, j = 1, 2, 3$ and conclude that $N_1^H, N_2^H, N_3^H$ are in the same asymptotic class.
		The proof of Lemma \ref{eq_d_1_ident_metr_gr} follows from this, Lemma \ref{lem_d_1_ident_metr} and the fact that by (\ref{eq_john_ellips}), our norms are ordered as $N_1^H \leq N_2^H \leq N_3^H$.
	\end{proof}
	\end{sloppypar}
	\par 
	Let us now discuss two different constructions of rays of norms associated to a filtration.
	We fix a finitely-dimensional normed vector space $(V, N_V)$, $\| \cdot \|_V := N_V$ and a filtration $\mathcal{F}$ of $V$.
	We construct a ray of norms $N_{V, \mathcal{F}}^t := \| \cdot \|_{V, \mathcal{F}}^t$, $t \in [0, +\infty[$, emanating from $N_V$, as follows
	\begin{equation}\label{eq_ray_norm_defn0}
		\| f \|_{V, \mathcal{F}}^t :=
		\inf 
		\Big\{
			\sum
			e^{- t \mu_i} \cdot 
			\| f_i \|_V
			\,
			:
			\,
			f = \sum f_i, \, f_i \in \mathcal{F}^{\mu_i} V
		\Big\}.
	\end{equation}
	\begin{lem}\label{lem_log_spec_ray}
		For any $t \geq s \geq 0$ and $i = 1, \ldots, \dim V$, the following estimate holds
		\begin{equation}
			\Big| \lambda_i(N_{V, \mathcal{F}}^t, N_{V, \mathcal{F}}^s) - (t - s) \cdot e_{\mathcal{F}}(i) \Big|
			\leq
			\log \dim V,
		\end{equation}
		where $e_{\mathcal{F}}(i)$ are the jumping numbers of the filtration $\mathcal{F}$, defined as
		\begin{equation}\label{eq_defn_jump_numborig}
			e_{\mathcal{F}}(i) := \sup \Big\{ t \in \real : \dim \mathcal{F}^t H^0(X, L^{\otimes k}) \geq i \Big\}.
		\end{equation}
	\end{lem}
	\begin{proof}
		Let us first show that for any $t \geq s \geq 0$, $f \in V$ the following identity holds
		\begin{equation}\label{eq_ray_norm_defn2}
			\| f \|_{V, \mathcal{F}}^t =
			\inf 
			\Big\{
				\sum
				e^{- (t - s) \mu_i} \cdot 
				\| f_i \|_{V, \mathcal{F}}^s
				\,
				:
				\,
				f = \sum f_i, \, f_i \in \mathcal{F}^{\mu_i} V
			\Big\}.
		\end{equation}
		Clearly, once (\ref{eq_ray_norm_defn2}) is established, it would be enough to prove Lemma \ref{lem_log_spec_ray} for $s = 0$.
		\par 
		First of all, the inequality 
		\begin{equation}\label{eq_ray_norm_defn2aa1}
			\| f \|_{V, \mathcal{F}}^t \geq
			\inf 
			\Big\{
				\sum
				e^{- (t - s) \mu_i} \cdot 
				\| f_i \|_{V, \mathcal{F}}^s
				\,
				:
				\,
				f = \sum f_i, \, f_i \in \mathcal{F}^{\mu_i} V
			\Big\},
		\end{equation}
		follows directly from the bound $\| f_i \|_{V, \mathcal{F}}^s \leq e^{- s \mu_i} \cdot \| f_i \|_V$ and (\ref{eq_ray_norm_defn0}).
		\par 
		To establish the inverse inequality, we first remark that it is clear (by the use of projection operator) that if $f  \in \mathcal{F}^{\mu} V$, $\mu \in \real$, then in (\ref{eq_ray_norm_defn0}), it is sufficient to consider the decompositions with $f_i \in \mathcal{F}^{\mu_i} V$, $\mu_i \geq \mu$.
		For any $\epsilon > 0$, we now consider a decomposition $f = \sum f_i^{\epsilon}$, $f_i^{\epsilon} \in \mathcal{F}^{\mu_i} V$, verifying
		\begin{equation}
			\sum_{i}
				e^{- (t - s) \mu_i} \cdot 
				\| f_i^{\epsilon} \|_{V, \mathcal{F}}^s
			\leq
			\inf 
			\Big\{
				\sum_{i}
				e^{- (t - s) \mu_i} \cdot 
				\| f_i \|_{V, \mathcal{F}}^s
				\,
				:
				\,
				f = \sum f_i, \, f_i \in \mathcal{F}^{\mu_i} V
			\Big\}
			+
			\epsilon.
		\end{equation}
		We also consider decompositions $f_i^{\epsilon} = \sum f_{i, j}^{\epsilon}$, $f_{i, j}^{\epsilon} \in \mathcal{F}^{\mu_{i, j}} V$, $\mu_{i, j} \geq \mu_i$, verifying 
		\begin{equation}
			\sum_{j}
				e^{- s \mu_{i, j}} \cdot 
				\| f_{i, j}^{\epsilon} \|_V
			\leq
			\| f_i^{\epsilon} \|_{V, \mathcal{F}}^s
			+
			\epsilon.
		\end{equation}
		Clearly, since $\mu_{i, j} \geq \mu_i$, we then have
		\begin{equation}
			\sum_{i, j}
				e^{- (t - s) \mu_i} 
				\cdot
				e^{- s \mu_{i, j}} \cdot 
				\| f_{i, j}^{\epsilon} \|_V
			\geq
			\sum_{i, j}
				e^{- t  \mu_{i, j}} 
				\cdot 
				\| f_{i, j}^{\epsilon} \|_V
			\geq 
			\| f \|_{V, \mathcal{F}}^t,
		\end{equation}
		which, along with the above estimates, imply the inverse inequality to (\ref{eq_ray_norm_defn2aa1}), when $\epsilon \to 0$.
		\par 
		Let assume first that $N_V$ is Hermitian. We then show the following stronger statement
		\begin{equation}\label{eq_spec_herm}
			\lambda_i(N_{V, \mathcal{F}}^t, N_V) = t \cdot e_{\mathcal{F}}(i).
		\end{equation}
		For simplicity of the presentation, we assume that the numbers $e_{\mathcal{F}}(i)$ are different.
		Remark that if $x \in V$ is orthogonal (with respect to the scalar product associated with $N_V$) to $\mathcal{F}^{e_{\mathcal{F}}(i - 1)} V$, then $\| x \|_{V, \mathcal{F}}^t \geq \| x \|_V \cdot \exp(- t e_{\mathcal{F}}(i))$.
		In particular, since in any $i$-dimensional subspace of $V$, there is an element orthogonal to the $i - 1$-dimensional subspace $\mathcal{F}^{e_{\mathcal{F}}(i - 1)} V$, we have $\lambda_i(N_{V, \mathcal{F}}^t, N_V) \leq t  \cdot e_{\mathcal{F}}(i)$.
		However, by taking a subspace $W := \mathcal{F}^{e_{\mathcal{F}}(i)} V$ in (\ref{eq_log_rel_spec}), we obtain $\lambda_i(N_{V, \mathcal{F}}^t, N_V) \geq t \cdot e_{\mathcal{F}}(i)$.
		A combination of the two estimates imply (\ref{eq_spec_herm}).
		\par 
		Now, for general norms $N_V$, we consider the norm $N_V^H$ as in (\ref{eq_john_ellips}).
		Let us denote by $N_{V, \mathcal{F}}^{H, t}$ the ray of norms emanating from $N_V^H$ as in (\ref{eq_ray_norm_defn0}).
		Clearly, by (\ref{eq_john_ellips}), for any $t \in [0, +\infty[$, we have 
		\begin{equation}
			d_{+ \infty}(N_{V, \mathcal{F}}^{H, t}, N_{V, \mathcal{F}}^t) \leq \frac{1}{2} \log \dim V,
			\qquad
			d_{+ \infty}(N_V, N_V^H) \leq \frac{1}{2} \log \dim V.
		\end{equation}
		By (\ref{eq_log_rel_spec02}) and (\ref{eq_spec_herm}), we conclude the proof of Lemma \ref{lem_log_spec_ray}.
	\end{proof}
	We now fix a \textit{Hermitian} norm $N_H := \| \cdot \|_H$ on $V$. 
	Let us give an alternative construction of a ray of norms emanating from $N_H$.
	Consider an orthonormal basis $s_1, \ldots, s_r$, $r := \dim V$, of $V$, adapted to the filtration $\mathcal{F}$, i.e. verifying $s_i \in \mathcal{F}^{e_{\mathcal{F}}(i)} V$.
	We define the ray of Hermitian norms $N_{H, \mathcal{F}}^{\perp, t} := \| \cdot \|_{H, \mathcal{F}}^{\perp, t}$, $t \in [0, +\infty[$, on $V$ by declaring the basis 
	\begin{equation}\label{eq_bas_st}
		(s_1^t, \ldots, s_r^t) := \big( e^{t e_{\mathcal{F}}(1)} s_1, \ldots, e^{t e_{\mathcal{F}}(r)} s_r \big),
	\end{equation}
	to be orthonormal with respect to $N_{H, \mathcal{F}}^{\perp, t}$.
	The following result compares the two constructions.
	\begin{lem}\label{lem_two_norms_comp0}
		For any (resp. Hermitian) norm $N_V$ (resp. $N_H$) on $V$ and any $t \in [0, +\infty[$, we have
		\begin{equation}
			d_{+ \infty}(N_{H, \mathcal{F}}^{\perp, t}, N_{V, \mathcal{F}}^t)
			\leq
			d_{+ \infty}(N_H, N_V)
			+
			\log \dim V.
		\end{equation}
	\end{lem}
	\begin{proof}
		Let us denote by $N_{H, \mathcal{F}}^{t}$ the ray of norms emanating from $N_H$ by the construction from (\ref{eq_ray_norm_defn0}).
		Let us establish first that
		\begin{equation}\label{eq_cal_hn_comp2}
			\dim V \cdot N_{H, \mathcal{F}}^{t} \geq  N_{H, \mathcal{F}}^{\perp, t}.
		\end{equation}
		By the definition of $N_{H, \mathcal{F}}^{t}$, we conclude that for any $\lambda \in \real$, $f \in V$, we have
		\begin{equation}\label{eq_cal_hn_comp3}
			\| f \|_{H, \mathcal{F}}^{t}  \geq e^{-t \lambda} \| Q_{\lambda}(f) \|_H,
		\end{equation}
		where $Q_{\lambda}(f) := f - P_{\lambda}(f)$ and $P_{\lambda}(f)$ is the projection of $f$ to $\cup_{\epsilon > 0} \mathcal{F}^{\lambda + \epsilon} V$ with respect to the norm $N_H$.
		We take now the decomposition $f = \sum a_i s_1^t$, $a_i \in \comp$ of $f \in V$ in basis $(s_1^t, \ldots, s_r^t)$ from (\ref{eq_bas_st}).
		Then by the definition of $N_{H, \mathcal{F}}^{\perp, t}$, we have
		\begin{equation}\label{eq_cal_hn_comp001}
			\| f \|_{H, \mathcal{F}}^{\perp, t}  := \sqrt{ \sum |a_i|^2}.
		\end{equation}
		By taking sums of (\ref{eq_cal_hn_comp3}) over all jumping numbers, using (\ref{eq_cal_hn_comp001}) and the fact that for any $i = 1, \ldots, r$, we have
		$ \| Q_{e_{\mathcal{F}}(i)}(f) \|_H
		\geq
		e^{t e_{\mathcal{F}}(i)} \cdot |a_i|$, we deduce (\ref{eq_cal_hn_comp2}).
		\par 
		Now, directly by the definition of $N_{H, \mathcal{F}}^{t}$, we obtain
		$
			\| f \|_{H, \mathcal{F}}^{t} \leq  \sum |a_i|
		$.
		From this, (\ref{eq_cal_hn_comp001}) and mean value inequality, we establish
		\begin{equation}\label{eq_cal_hn_comp0}
			N_{H, \mathcal{F}}^{t} \leq \sqrt{\dim V} \cdot  N_{H, \mathcal{F}}^{\perp, t}.
		\end{equation}
		\par 
		From (\ref{eq_cal_hn_comp2}) and (\ref{eq_cal_hn_comp0}), we conclude that 
		$
			d_{+ \infty}(N_{H, \mathcal{F}}^{\perp, t}, N_{H, \mathcal{F}}^t)
			\leq
			\log \dim V
		$.
		To finish the proof, it is only left to use the following trivial bound
		$
			d_{+ \infty}(N_{V, \mathcal{F}}^t, N_{H, \mathcal{F}}^t)
			\leq
			d_{+ \infty}(N_H, N_V)
		$.
	\end{proof}
	\par 
	While Lemma \ref{lem_two_norms_comp0} says that the two rays are very close, the first construction of the ray is very often easier to work with.
	In particular, it is trivial that the first construction is monotonic with respect to the initial norm and the filtration (for the second construction, it is also true but it is a non-trivial statement, cf. \cite[Proposition 4.12]{FinTits}).
	More importantly, as we shall see in Section \ref{sect_gr_jm}, when both rays of norms are defined on a graded algebra (instead of a vector space), the first ray preserves submultiplicativity. 
	For the second ray, the analogous result is not quite clear; moreover, it is not so trivial to construct a Hermitian submultiplicative norm to initiate the ray.
	\par 
	Now, filtrations $\mathcal{F}$ on $V$ are in one-to-one correspondence with functions $\chi_{\mathcal{F}} : V \to [0, +\infty[$, defined as
	\begin{equation}\label{eq_filtr_norm}
		\chi_{\mathcal{F}}(s) := \exp(- w_{\mathcal{F}}(s)).
	\end{equation}
	where $w_{\mathcal{F}}(s)$ is the weight associated with the filtration, defined as
	$
		w_{\mathcal{F}}(s) := \sup \{ \lambda \in \real : s \in F^{\lambda} V \}
	$.
	An easy verification shows that $\chi_{\mathcal{F}}$ is a non-Archimedean norm on $V$ with respect to the trivial absolute value on $\comp$, i.e. it satisfies the following axioms
	\begin{enumerate}
		\item $\chi_{\mathcal{F}}(f) = 0$ if and only if $f = 0$,
		\item $\chi_{\mathcal{F}}(\lambda f) = \chi_{\mathcal{F}}(f)$, for any $\lambda \in \comp^*$, $k \in \nat^*$, $f \in V$,
		\item $\chi_{\mathcal{F}}(f + g) \leq \max \{ \chi_{\mathcal{F}}(f), \chi_{\mathcal{F}}(g) \}$, for any $k \in \nat^*$, $f, g \in V$.
	\end{enumerate}
	\par 
	Remark the following relation between the non-Archimedean norm $\chi_{\mathcal{F}}$ and the rays of submultiplicative norms $N_{H, \mathcal{F}}^{\perp, t}$, $ N_{V, \mathcal{F}}^t$ associated to $\mathcal{F}$ as above: for any $f \in V$, we have
	\begin{equation}\label{eq_na_nm_interpol}
		\log \chi_{\mathcal{F}}(f) 
		=
		 \lim_{t \to \infty} \frac{\log \| f \|_{V, \mathcal{F}}^{t}}{t} 
		=
		\lim_{t \to \infty} \frac{\log \| f \|_{H, \mathcal{F}}^{\perp, t}}{t}.
	\end{equation}
	Hence, the rays $N_{H, \mathcal{F}}^{\perp, t}$, $ N_{V, \mathcal{F}}^t$, $t \in [1, +\infty[$, should be regarded as interpolations between a fixed norm and a non-Archimedean norm associated to $\mathcal{F}$.
	\par 
	Let us, finally, recall some basic constructions of norms on tensor products.
	Let $V_1, V_2$ be two finite dimensional vector spaces endowed with norms $N_i = \norm{\cdot}_i$, $i = 1, 2$.
	\par 
	The \textit{projective tensor norm} $N_1 \otimes_{\pi} N_2 = \norm{\cdot}_{\otimes_{\pi}}$ on $V_1 \otimes V_2$ is defined for $f \in V_1 \otimes V_2$ as 
	\begin{equation}\label{eq_defn_proj_norm}
		\norm{f}_{ \otimes_{\pi} }
		=
		\inf
		\Big\{
			\sum \| x_i \|_1 \cdot  \| y_i \|_2
			:
			\quad
			f = \sum x_i \otimes y_i
		\Big\},
	\end{equation}
	where the infimum is taken over different ways of partitioning $f$ into a sum of decomposable terms.
	The \textit{injective tensor norm} $N_1 \otimes_{\epsilon} N_2 = \norm{\cdot}_{ \otimes_{\epsilon} }$ on $V_1 \otimes V_2$ is defined as 
	\begin{equation}\label{eq_defn_inf_norm}
		\norm{f}_{ \otimes_{\epsilon} }
		=
		\sup
		\Big\{
			\big|
				(\phi \otimes \psi)(f)
			\big|
			:
			\quad
			\phi \in V_1^*, \psi \in V_2^*, \| \phi \|_{1}^* = \| \psi \|_{2}^* = 1
		\Big\}
	\end{equation}
	where $\| \cdot \|_{i}^*$, $i = 1, 2$, are the dual norms associated with $\| \cdot \|_{i}$.
	Lemma below compares injective and projective tensor norms, see \cite[Proposition 6.1]{RyanTensProd}, \cite[Theorem 21]{SzarekComp} for a proof.
	\begin{lem}\label{lem_inj_proj_bnd_nntr}
		The following inequality between the norms on $V_1 \otimes V_2$ holds
		\begin{equation}\label{eq_bnd_proj_inf}
			 N_1 \otimes_{\epsilon} N_2
			 \leq
			 N_1 \otimes_{\pi} N_2
			 \leq
			 N_1 \otimes_{\epsilon} N_2
			 \cdot
			 \min \{ \dim V_1, \dim V_2 \}.
		\end{equation}
		If, moreover, the norms $N_1$ and $N_2$ are Hermitian, then
		\begin{equation}
			 N_1 \otimes_{\epsilon} N_2
			 \leq
			  N_1 \otimes N_2
			 \leq
			  N_1 \otimes_{\pi} N_2.
		\end{equation}
	\end{lem}
	
	\subsection{Pluripotential theory and quantization of distances}\label{sect_pp_thr}
	The main goal of this section is to recall some basic facts from pluripotential theory, emphasising metric and quantization aspects of the theory.
	\par 
	Let us fix a Kähler form $\omega$ on $X$ and consider the space $\mathcal{H}_{\omega}$ of Kähler potentials, consisting of $u \in \ccal^{\infty}(X, \real)$, such that $\omega_u := \omega + \imun \partial \dbar u$ is strictly positive.
	We denote by ${\rm{PSH}}(X, \omega)$ the set of $\omega$-psh potentials; these are upper semi-continuous functions $u \in L^1(X, \real \cup \{ -\infty \})$, such that $\omega_u$ is positive as a $(1, 1)$-current.
	When the De Rham cohomology class $[\omega]$ of $\omega$ satisfies $[\omega] \in 2 \pi H^2(X, \mathbb{Z})$, there is a Hermitian line bundle $(L, h^L_0)$, such that $\omega = 2 \pi c_1(L, h^L_0)$. 
	Hence, upon fixing $h^L_0$ (which is uniquely defined up to a multiplication by a locally constant function), the set $\mathcal{H}_{\omega}$  (resp. ${\rm{PSH}}(X, \omega)$) can be identified with the set of smooth positive (resp. psh) metrics on $L$ through the correspondence
	\begin{equation}\label{eq_pot_to_metr}
		u \mapsto h^L := e^{-u} \cdot h^L_0.
	\end{equation}
	Remark that we then have $\omega_u = 2 \pi c_1(L, h^L)$.
	This identification will be implicit later on, and all the constructions (of distances, geodesics, psh rays, etc.) for elements from $\mathcal{H}_{\omega}$ and ${\rm{PSH}}(X, \omega) \cap \mathcal{L}^{\infty}(X)$ will be implicitly extended to the corresponding sets of metrics on the line bundle $L$.
	\par 
	One can introduce on $\mathcal{H}_{\omega}$ a collection of $L^p$-type Finsler metrics.
	For $u \in \mathcal{H}_{\omega}$, let us first define the Monge-Ampère operator as $MA(u) := \frac{\omega_u^n}{V}$, where $V = \int \omega^n$.
	If $u \in \mathcal{H}_{\omega}$ and $\xi \in T_u \mathcal{H}_{\omega} \simeq \ccal^{\infty}(X, \real)$, then the $L_p$-length of $\xi$ is given by the following expression
	\begin{equation}\label{eq_finsl_dist_fir}
		\| \xi \|_{p, u}
		:=
		\sqrt[p]{
		 \int_X |\xi|^p \cdot MA(u)}.
	\end{equation}
	For $p = 2$, this was introduced by Mabuchi \cite{Mabuchi}, and for $p \in [1, +\infty[$, by Darvas \cite{DarvasFinEnerg}.
	\par 
	Using these Finsler metrics, one can introduce path length metric structures $(\mathcal{H}_{\omega}, d_p)$. 
	In \cite{DarvasFinEnerg}, Darvas studied the completion of these metric spaces, $(\mathcal{E}_{\omega}^p, d_p)$, nowadays called \textit{finite $p$-energy classes}, and proved that these completions are geodesic metric spaces and have a vector space structure.
	It is also well-known, cf. Guedj-Zeriahi \cite[Exercise 10.2]{GuedjZeriahBook}, that
	\begin{equation}\label{eq_inter_ep}
		\cap_{p = 1}^{\infty} \mathcal{E}_{\omega}^p
		=
		{\rm{PSH}}(X, \omega) \cap \mathcal{L}^{\infty}(X).
	\end{equation}
	\par 
	Darvas proved in \cite[Proposition 4.9]{DarvasFinEnerg} that a monotonic sequence of bounded psh metrics $h^L_i$ converges almost everywhere to a bounded psh metric $h^L$, if and only if for any (or for some) $p \in [1, + \infty[$, we have 
	\begin{equation}\label{eq_darvas_conv}
		\lim_{i \to \infty} d_p(h^L_i, h^L) = 0.
	\end{equation}
	\par 
	The distance on $\mathcal{E}_{\omega}^1$ can be alternatively described in terms of the \textit{Monge-Ampère energy} functional $E$.
	Recall that $E$ is explicitly given for $u, v \in \mathcal{H}_{\omega}$ by
	\begin{equation}\label{eq_energy}
		E(u) - E(v)
		=
		\frac{1}{(n + 1) V}
		\sum_{j = 0}^{n} \int_X (u - v) w_u^j \wedge w_v^{n - j}.
	\end{equation}
	By \cite[Proposition 10.14]{GuedjZeriahBook}, $E$ is monotonic, i.e. for any $u \leq v$, we have $E(u) \leq E(v)$.
	From this and Remark \ref{rem_regul_above}b), it is reasonable to extend the domain of the definition of $E$ to ${\rm{PSH}}(X, \omega)$ as
	\begin{equation}\label{eq_energy_ext}
		E(u) :=
		\inf 
		\Big \{ 
		 E(v) : v \in \mathcal{H}_{\omega}, u \leq v
		\Big \}.
	\end{equation}
	Darvas proved in \cite{DarvasFinEnerg} that $\mathcal{E}_{\omega}^1$ coincides with the set of $u \in {\rm{PSH}}(X, \omega)$, verifying $E(u) > - \infty$. 
	Moreover, for any $u, v \in \mathcal{E}_{\omega}^1$, verifying $u \leq v$, according to\cite[Corollary 4.14]{DarvasFinEnerg}, we have
	\begin{equation}\label{eq_d1_ener}
		d_1(u, v) = E(v) - E(u).
	\end{equation}
	In particular, similarly to Lemma \ref{lem_d_1_ident_metr}, $(\mathcal{E}_{\omega}^1, d_1)$ is not a uniquely geodesic space -- a fact originally observed by Darvas \cite[comment after Theorem 4.17]{DarvWeakGeod}.
	\par 
	Certain geodesic segments of $(\mathcal{E}_{\omega}^p, d_p)$ can be constructed as upper envelopes of quasi-psh functions. 
	More precisely, we identify paths $u_t \in \mathcal{E}_{\omega}^p$, $t \in [0, 1]$ with rotationally-invariant $\hat{u}$ over $X \times \mathbb{D}_{e^{-1}, 1}$ by 
	\begin{equation}\label{eq_defn_hat_u}
		\hat{u}(x, s) = u_{- \log |s|}(x).
	\end{equation}
	We say that a curve $[0,1] \ni t \to v_t \in \mathcal{E}_{\omega}^p$ is a \textit{weak subgeodesic} connecting $u_0, u_1 \in \mathcal{E}_{\omega}^p$ if $d_p(v_t, u_i) \to 0$, as $t \to 0$ for $i = 0$ and $t \to 1$ for $i = 1$, and we have
	\begin{equation}\label{eq_subgeod_req}
		\hat{u} \text{ is $\pi^* \omega$-psh on $X \times \mathbb{D}_{e^{-1}, 1}$.}
	\end{equation}	
	As shown in \cite[Theorem 2]{DarvasFinEnerg}, a distinguished $d_p$-geodesic $[0, 1] \ni t \to u_t \in \mathcal{E}_{\omega}^p$ connecting $u_0, u_1$ can be then obtained as the following envelope
	\begin{equation}\label{eq_geod_as_env}
		u_t := \sup \Big\{ 
			v_t \, : \, t \to v_t \,  \text{ is a weak subgeodesic connecting } \, v_0 \leq u_0 \text{ and } v_1 \leq u_1
		\Big\}.
	\end{equation}
	\par 
	When $u_0, u_1 \in {\rm{PSH}}(X, \omega) \cap \mathcal{L}^{\infty}(X)$, Berndtsson \cite[\S 2.2]{BernBrunnMink} in \cite[\S 2.2]{BernBrunnMink} proved that $u_t$, $t \in [0, 1]$, defined by (\ref{eq_geod_as_env}), can be described as the only path connecting $u_0$ to $u_1$, so that $\hat{u}$ is the solution of the following Monge-Ampère equation
	\begin{equation}\label{eq_ma_geod}
		(\pi^* \omega + \imun \partial \dbar \hat{u})^{n + 1} = 0,
	\end{equation}
	where the wedge power is interpreted in Bedford-Taylor sense \cite{BedfordTaylor}.
	\par 
	For smooth geodesic segments in $(\mathcal{H}_{\omega}, d_2)$, Semmes \cite{Semmes} and Donaldson \cite{DonaldSymSp} have made similar observations before.
	The uniqueness of the solution of (\ref{eq_ma_geod}) is assured by \cite[Lemma 5.25]{GuedjZeriahBook}.
	Remark, in particular, that for any $u_0, u_1 \in {\rm{PSH}}(X, \omega) \cap \mathcal{L}^{\infty}(X)$, the distinguished weak geodesic connecting them is the same if we regard $u_0, u_1$ as elements in any of $\mathcal{E}_{\omega}^p$, $p \in [1, + \infty[$.
	\par 
	\begin{thm}[{Darvas-Lu \cite[Theorem 2]{DarLuGeod} }]\label{thm_dar_lu_uniq_geod}
		For any $p \in ]1, + \infty[$, $(\mathcal{E}_{\omega}^p, d_p)$ is uniquely geodesic.
	\end{thm}
	\par 
	\begin{sloppypar}
	Let us now define the spectral measure of a finite geodesic segment.
	We fix $u_0, u_1 \in {\rm{PSH}}(X, \omega) \cap \mathcal{L}^{\infty}(X)$ and consider $u_t$, $t \in [0, 1]$ as in (\ref{eq_geod_as_env}).
	From Berndtsson \cite[\S 2.2]{BernBrunnMink}, we know that then $u_t \in \mathcal{L}^{\infty}(X)$ and the limits $\lim_{t \to 0} u_t = u_0$, $\lim_{t \to 1} u_t = u_1$ hold in the uniform sense.
	Also, remark that the condition (\ref{eq_subgeod_req}) implies that for a fixed $x \in X$, the function $u_t(x)$ is convex in $t \in [0, 1]$, see \cite[Theorem I.5.13]{DemCompl}.
	Hence, one-sided derivatives $\dot{u}_t^{-}$, $\dot{u}_t^{+}$ of $u_t$ are well-defined for $t \in ]0, 1[$ and they increase in $t$.
	We denote $\dot{u}_0 := \lim_{t \to 0} \dot{u}_t^{-} = \lim_{t \to 0} \dot{u}_t^{+}$.
	From \cite[\S 2.2]{BernBrunnMink}, we know that $\dot{u}_0$ is bounded and by Darvas \cite[Theorem 1]{DarvWeakGeod}, we, moreover, have
	\begin{equation}
		\sup |\dot{u}_0| \leq \sup |u_1 - u_0|.
	\end{equation}
	We now assume that $u_0 \in \mathcal{H}_{\omega}$ and define the \textit{spectral measure} $\mu_{u_0, u_1}$ as 
	\begin{equation}\label{eq_spec_meas_defn_seg}
		\mu_{u_0, u_1} = (\dot{u}_0)_* (MA(u_0)),
	\end{equation}
	where $MA(u_0)$ was defined in (\ref{eq_finsl_dist_fir}).
	Clearly, by (\ref{eq_pot_to_metr}), such definition coincides with the one from the introduction.
	Then according to Darvas-Lu-Rubinstein \cite[Lemma 4.5]{DarvLuRub}, for any $u_0 \in \mathcal{H}_{\omega}$, $u_1 \in {\rm{PSH}}(X, \omega) \cap \mathcal{L}^{\infty}(X)$, we have
	\begin{equation}\label{eq_d_p_berndss}
		d_p(u_0, u_t)
		=
		t \cdot \sqrt[p]{ \int_X |\dot{u}_0|^p \cdot MA(u_0) }.
	\end{equation}
	See also Berndtsson \cite{BerndtProb} and Di Nezza-Lu \cite{DiNezzLuGeodDist} for related results.
	\end{sloppypar}
	\par 
	Remark that when $u_1 \geq u_0$, from (\ref{eq_geod_as_env}), the inequality $u_t \geq u_0$ holds for any $t \in [0, 1]$.
	In particular, $\dot{u}_0 \geq 0$ and then the spectral measure is characterized by the following property 
	\begin{equation}\label{eq_char_spec_meas}
		 \text{$p$-moments of $\mu_{u_0, u_1}$ coincide with $d_p(u_0, u_1)^p$ for any $p \in \nat^*$.}
	\end{equation}
	\par 
	For a bounded metric $h^L$ on $L$ and a positive volume form $\mu$ of unit volume on $X$, we denote by ${\rm{Hilb}}_k(h^L, \mu) = \| \cdot \|_{L^2_k(h^K, \mu)}$ the $L^2$-norm on $H^0(X, L^{\otimes k})$, defined for $f \in H^0(X, L^{\otimes k})$ as
	\begin{equation}\label{eq_hilb_defn}
		 \| f \|_{L^2_k(h^K, \mu)}^2
		 =
		 \int_X |f(x)|_{h^L}^2 d \mu(x).
	\end{equation}
	The following result says that the metrics $d_p$ on the space of Hermitian norms on $H^0(X, L^{\otimes k})$ are quantisations of Darvas metrics $d_p$ on the space of bounded psh metrics on $L$.
	\begin{thm}[{Darvas-Lu-Rubinstein \cite[Theorem 1.2]{DarvLuRub}}]\label{thm_darvlurub_quant}
		For any bounded psh metrics $h^L, h^L_0, h^L_1$ on an ample line bundle $L$ and any $p \in [1, +\infty[$, we have
		\begin{equation}
		\begin{aligned}
			&
			d_p \Big( {\rm{Hilb}}(h^L_0, \mu), {\rm{Hilb}}(h^L_1, \mu) \Big)
			=
			d_p(h^L_0, h^L_1),
			\\
			\lim_{k \to \infty} & d_p \Big( FS \big( {\rm{Hilb}}_k(h^L, \mu) \big)^{\frac{1}{k}}, h^L \Big) = 0.
		\end{aligned}
		\end{equation}
	\end{thm}
	\begin{rem}
		a) When $h^L, h^L_0, h^L_1$  are smooth and positive, the first result was established by Chen-Sun \cite{ChenSunQuant} for $p = 2$ and by Berndtsson \cite{BerndtProb} for $p \in [1, +\infty[$; the second result in this more regular setting is a direct consequence of Tian's theorem \cite{TianBerg}. See also Catlin \cite{Caltin}, Zelditch \cite{ZeldBerg}, Dai-Liu-Ma \cite{DaiLiuMa} and Ma-Marinescu \cite{MaHol}, for refinements of the latter statement.
		These results go in line with the general philosophy that the geometry of the space of psh metrics on $L$ can be approximated by the geometry of the space of norms on $H^0(X, L^{\otimes k})$, as $k \to \infty$, see Donaldson \cite{DonaldSymSp} and Phong-Sturm \cite{PhongSturm}.
		\par
		b) In particular, ${\rm{Hilb}}(h^L_0, \mu) \not\sim_p {\rm{Hilb}}(h^L_1, \mu)$ for any $p \in [1, +\infty]$ and $h^L_0 \neq h^L_1$ bounded psh.
		This goes in sharp contrast with sup-norms, as we shall see below.
	\end{rem}
	Define now, following Boucksom-Eriksson \cite[\S 7.5]{BouckErik21}, the Fubini-Study envelope $Q(h^L)$ of a fixed bounded (not necessarily psh) metric $h^L$ as follows
	\begin{equation}\label{eq_defn_fs_env}
		Q(h^L)(x) := \inf \Big\{ h^L_0(x) \text{ continous psh on $L$, $h^L \leq h^L_0$} \Big\}, \text{ for any } x \in X.
	\end{equation}
	Remark that by \cite[Proposition I.4.24]{DemCompl}, $Q(h^L)_*$ is psh and regularizable from above.
	In particular, $h^L = Q(h^L)_*$ if and only if $h^L$ is psh and regularizable from above.
	\begin{thm}[{\cite[Theorem 7.26 and Corollary 7.27]{BouckErik21}}]\label{thm_bouck_erikss}
		For any bounded metric $h^L$ on $L$
		\begin{equation}
			\lim_{k \to \infty} FS \big({\rm{Ban}}^{\infty}_k(h^L) \big)^{\frac{1}{k}} = Q(h^L), \qquad {\rm{Ban}}^{\infty}(h^L) = {\rm{Ban}}^{\infty}(Q(h^L)).
		\end{equation}
	\end{thm}
	\begin{rem}\label{rem_bouck_erikss}
		From Theorems \ref{thm_darvlurub_quant}, \ref{thm_bouck_erikss}, a bounded psh metric $h^L$ is regularizable from above if and only if $\lim_{k \to \infty} d_p \big( FS \big( {\rm{Hilb}}_k(h^L, \mu) \big)^{\frac{1}{k}}, FS \big({\rm{Ban}}^{\infty}_k(h^L) \big)^{\frac{1}{k}} \big) = 0$ for any (some) $p \in [1, +\infty[$.
	\end{rem}
	\begin{prop}\label{prop_ban_hilb_eq}
		For any regularizable from above psh metric $h^L$ and $p \in [1, + \infty[$, we have 
		$
			{\rm{Hilb}}(h^L, \mu) \sim_p {\rm{Ban}}^{\infty}(h^L)
		$.
		If, moreover, $h^L$ is continuous, then one can take $p = + \infty$.
	\end{prop}	
	\begin{rem}
		By Theorems \ref{thm_darvlurub_quant}, \ref{thm_bouck_erikss} and Proposition \ref{prop_ban_hilb_eq}, regularizable from above psh metrics is the biggest subclass of bounded psh metrics for which the $p$-equivalence above holds.
	\end{rem}
	\begin{proof}
		Since the equivalence relation $\sim_{+ \infty}$ equals to $\sim$, the second part is well-known, cf. \cite[Proposition 2.10]{FinSecRing}.
		To establish the first part, take a decreasing sequence of continuous psh metrics $h^{L}_r$, $r \in \nat$,  as in Definition \ref{defn_regul_above}.
		From (\ref{eq_darvas_conv}), for any $\epsilon > 0$, there is $r_0 \in \nat$, such that
		\begin{equation}\label{eq_ban_hilb_appro0}
			d_p \big( h^L, h^{L}_{r_0} \big)  < \epsilon.
		\end{equation}	
		Remark that we trivially have
		\begin{equation}\label{eq_ban_hilb_appro2}
			{\rm{Hilb}}(h^L, \mu) \leq {\rm{Ban}}^{\infty}(h^L) \leq {\rm{Ban}}^{\infty}(h^{L}_{r_0}).
		\end{equation}
		Since $h^{L}_r$ are continuous and psh, by the second part of Proposition \ref{prop_ban_hilb_eq}, we can find $k_0 \in \nat$, such that for any $k \geq k_0$, we have
		\begin{equation}\label{eq_ban_hilb_appro3}
			{\rm{Ban}}^{\infty}_k(h^{L}_{r_0}) < \exp(\epsilon k) \cdot {\rm{Hilb}}_k(h^{L}_{r_0}, \mu)
		\end{equation}
		From Theorem \ref{thm_darvlurub_quant} and (\ref{eq_ban_hilb_appro0}), we see that
		\begin{equation}\label{eq_ban_hilb_appro4}
			d_p \Big( {\rm{Hilb}}(h^L, \mu), {\rm{Hilb}}(h^{L}_{r_0}, \mu) \Big)  < \epsilon.
		\end{equation}
		From (\ref{eq_log_rel_spec2}), (\ref{eq_ban_hilb_appro2}), (\ref{eq_ban_hilb_appro3}) and (\ref{eq_ban_hilb_appro4}), we conclude that 
		\begin{equation}
			d_p \Big( {\rm{Hilb}}(h^L, \mu), {\rm{Ban}}^{\infty}(h^L) \Big) < 2 \epsilon.
		\end{equation}
		Since $\epsilon > 0$ was chosen arbitrary, this finishes the proof.
	\end{proof}
	
	\begin{cor}\label{cor_baninf_dist}
		For two distinct regularizable from above psh metrics $h^L_i$, $i = 0, 1$, and any $p \in [1, + \infty]$, we have ${\rm{Ban}}^{\infty}(h^L_0) \not\sim_p {\rm{Ban}}^{\infty}(h^L_1)$.
	\end{cor}
	\begin{proof}
		It follows from Theorem \ref{thm_darvlurub_quant} and Proposition \ref{prop_ban_hilb_eq}.
	\end{proof}	
	
	\begin{cor}\label{cor_psh_appr_below}
		Assume that a decreasing sequence of continuous psh metrics $h^L_i$, $i \in \nat$, is uniformly bounded from below.
		Then for any $p \in [1, +\infty[$, for $h^L := \lim h^L_i$,
		$
			{\rm{Ban}}^{\infty}(h^L) \sim_p {\rm{Ban}}^{\infty}(h^L_*)
		$.
		In particular, for a bounded metric $h^L$ on $L$, we have ${\rm{Ban}}^{\infty}(Q(h^L)) \sim_p {\rm{Ban}}^{\infty}(Q(h^L)_*).$
	\end{cor}
	\begin{proof}
		It follows directly from the proof of Proposition \ref{prop_ban_hilb_eq} and the fact that $h^L_*$ coincides with $h^L$ almost everywhere, cf. \cite[Proposition I.4.24]{DemCompl}.
	\end{proof}

	\section{Study of the set of submultiplicative norms}\label{sect_class_sn}
	The main goal of this section is to establish a classification of submultiplicative norms on section rings of ample line bundles.
	More precisely, in Section \ref{sect_subm}, we prove Theorems \ref{thm_char},  \ref{thm_char2}, modulo a certain statement, which will be interpreted in Section \ref{sect_norm_proj} in a functional-analytic language.
	In Section \ref{sect_hol_ext_pp}, we discuss an application to holomorphic extension theorem.
	In Section \ref{sect_homog}, as another application of our methods, we give an explicit formula for the spectral radius seminorm associated with a submultiplicative norm and discuss the connection between the current work and the previous works in the non-Archimedean setting.
	
	\subsection{Classification of submultiplicative norms and applications}\label{sect_subm}
		The main goal of this section is to establish Theorems \ref{thm_char}, \ref{thm_char2} giving a characterization of submultiplicative norms in terms of sup-norms, and then to deduce Theorem \ref{thm_ns_conv}.
		We conserve the notation from the introduction.
		Throughout the whole section we assume that $L$ is ample.
	\par 
	For any $r \in \nat^*$, $k; k_1, \ldots, k_r \in \nat$, $k_1 + \cdots + k_r = k$, we define the multiplication map
	\begin{equation}\label{eq_mult_map}
		{\rm{Mult}}_{k_1, \cdots, k_r} : H^0(X, L^{k_1}) \otimes \cdots \otimes H^0(X, L^{k_r}) 
		\to
		H^0(X, L^{\otimes k}),
	\end{equation}	
	as follows $f_1 \otimes \cdots \otimes f_r \mapsto f_1 \cdots f_r$.
	It is standard that there is $p_0 \in \nat^*$, such that for any $k_1, \cdots, k_r \geq p_0$, the map ${\rm{Mult}}_{k_1, \cdots, k_r}$ is surjective, cf. \cite[Proposition 3.1]{FinSecRing}.
	\par 
	Assume now that $k, l \in \nat^*$ are big enough so that ${\rm{Mult}}_{k, l}$ is surjective.
	As we shall later apply the following result only for sufficiently high tensor powers of ample line bundles, we could always reduce to this case.
	A central idea of our approach to Theorems \ref{thm_char}, \ref{thm_char2} is to interpret the submultiplicativity condition in terms of projective tensor norms, see (\ref{eq_defn_proj_norm}).
	In fact, using notations (\ref{eq_defn_quot_norm}), (\ref{eq_defn_proj_norm}), the submultiplicativity condition can be reformulated in terms of inequalities between the norms on $H^0(X, L^{\otimes (k + l)})$ as follows
	\begin{equation}\label{eq_reform_subm_cond}
		N_{k + l} \leq [N_k \otimes_{\pi} N_l].
	\end{equation}
	\par 
 	Assume now for simplicity that $L$ is very ample and all the multiplication maps are surjective.
 	Let $N_1$ be a norm on $H^0(X, L)$.
 	By the surjectivity of the multiplication maps, we endow $H^0(X, L^{\otimes k})$ with the norms $N_k^{\pi} = [N_1 \otimes_{\pi} \cdots \otimes_{\pi} N_1]$ and $N_k^{\epsilon} = [N_1 \otimes_{\epsilon} \cdots \otimes_{\epsilon} N_1]$, where the tensor powers are repeated $k$ times.
 	We denote by $N^{\pi} = \sum N_k^{\pi}$ and $N^{\epsilon} = \sum N_k^{\epsilon}$ the induced graded norms on $R(X, L)$.
 	According to (\ref{eq_reform_subm_cond}), the norm $N^{\pi}$ is the biggest submultiplicative norm on $R(X, L)$, coinciding with $N_1$ on $H^0(X, L)$.
 	Next result, established in Sections \ref{sect_norm_proj} and \ref{sec_norm_poly_big}, lies in the core of our approach to the proofs of Theorems \ref{thm_char} and \ref{thm_char2}.
 	\begin{thm}\label{thm_induc}
 		The norms $N^{\pi}$, $N^{\epsilon}$ and ${\rm{Ban}}^{\infty}(FS(N_1))$ are equivalent.
 	\end{thm} 
 	\begin{rem}
 		In \cite[Theorem 4.18]{FinSecRing}, we established a similar statement, where we assumed that $N_1$ is Hermitian and projective/injective tensor norms were replaced by the Hermitian tensor norm.
 		Since according to Lemma \ref{lem_inj_proj_bnd_nntr}, the Hermitian tensor norm is pinched between the injective and projective tensor norms, Theorem \ref{thm_induc} refines \cite[Theorem 4.18]{FinSecRing}.
 		The Hermitian assumption in \cite{FinSecRing} simplified substantially the proof, as it allowed us to do explicit calculations on the projective space, see \cite[the first part of the proof of Theorem 4.18]{FinSecRing}.
 		Circumventing these calculations is exactly the content of Section \ref{sec_norm_poly_big} of this article.
 	\end{rem}
 	To establish Theorem \ref{thm_char}, recall the following basic lemma.
		\begin{lem}\label{lem_fs_dini}
			The sequence of Fubini-Study metrics $FS(N_k)$, $k \in \nat^*$, is submultiplicative for any submultiplicative graded norm $N = \sum N_k$.
			In particular, by Fekete's lemma, the sequence of metrics $FS(N_k)^{\frac{1}{k}}$ on $L$ converges, as $k \to \infty$, to a (possibly only bounded from above and even null) upper semi-continuous metric, which we denote by $FS(N)$.
			We, moreover, have 
			\begin{equation}\label{eq_fek_inf}
				FS(N) = \inf FS(N_k)^{\frac{1}{k}}.
			\end{equation}
			If $N$ is bounded, then $FS(N)_*$ is a regularizable from above psh metric.
			If $FS(N)$ is lower semi-continuous and everywhere non-null, the convergence is uniform and $FS(N)$ is psh.
		\end{lem}
		\begin{proof}
			The first part follows easily from Lemma \ref{lem_fs_inf_d}.
			The second part follows from Lemma \ref{lem_fs_inf_d} and some classical results, cf. \cite[Proposition I.4.24]{DemCompl}.
			The third part is a consequence of the well-known subadditive analogue of Dini's theorem and a statement asserting that a pointwise limit of subadditive sequence of continuous functions is upper semi-continuous, cf. \cite[Appendix A]{FinSecRing}.
		\end{proof}
	\begin{proof}[Proof of Theorem \ref{thm_char}.]
		Let us fix $\epsilon > 0$.
		By our assumption on the continuity of $FS(N)$ and Lemma \ref{lem_fs_dini}, there is $k_0 \in \nat$, such that for any $k \geq k_0$, we have
		\begin{equation}\label{eq_fs_unif}
			FS(N_k)^{\frac{1}{k}}
			\leq
			\exp(\epsilon / 3) \cdot FS(N).
		\end{equation}
		\par 
		Recall that in \cite[Theorem 1.5]{FinSecRing}, we proved that for any continuous psh metric $h^L$ and a smooth volume form $\mu$, the graded norm ${\rm{Hilb}}(h^L, \mu)$ is multiplicatively generated in the sense of \cite[Definition 1.3]{FinSecRing}.
		This means, in particular, that there is $k_1 \in \nat$, such that for any $k, l \geq k_1$, we have
		\begin{multline}
			\exp(- \epsilon (k + l) / 6) \cdot {\rm{Hilb}}_{k + l}(h^L, \mu) \leq [{\rm{Hilb}}_k(h^L, \mu) \otimes {\rm{Hilb}}_l(h^L, \mu)] 
			\\
			\leq \exp(\epsilon (k + l) / 6) \cdot {\rm{Hilb}}_{k + l}(h^L, \mu),
		\end{multline}
		where ${\rm{Hilb}}_k(h^L, \mu) \otimes {\rm{Hilb}}_l(h^L, \mu)$ is the Hermitian norm on $H^0(X, L^{\otimes k}) \otimes H^0(X, L^{\otimes l})$ induced by ${\rm{Hilb}}_k(h^L, \mu)$ and ${\rm{Hilb}}_l(h^L, \mu)$.
		Remark, however, that by Proposition \ref{prop_ban_hilb_eq}, the graded norms ${\rm{Hilb}}(h^L, \mu)$ and ${\rm{Ban}}^{\infty}(h^L)$ are equivalent.
		Applying this for $h^L := FS(N)$ with the use of Lemma \ref{lem_inj_proj_bnd_nntr}, we see that there is $k_2 \in \nat$, such that for any $k, l \geq k_2$, we have
		\begin{multline}\label{eq_ban_inf_proj}
			\exp(- \epsilon (k + l) / 3) \cdot {\rm{Ban}}^{\infty}_{k + l}(FS(N)) \leq [{\rm{Ban}}^{\infty}_{k}(FS(N)) \otimes_{\pi} {\rm{Ban}}^{\infty}_{l}(FS(N))] 
			\\
			\leq \exp(\epsilon (k + l) / 3) \cdot {\rm{Ban}}^{\infty}_{k + l}(FS(N)).
		\end{multline}
		We fix from now on $k' \geq \max\{ k_0, k_1, k_2 \}$.
		\par 
		Directly from Lemma \ref{lem_fs_inf_d}, we see that for any $k \in \nat^*$, we have
		\begin{equation}\label{eq_norm_with_fs_compar}
			N_k \geq {\rm{Ban}}^{\infty}_{k}(FS(N_k)).
		\end{equation}		 
		In conjunction with (\ref{eq_fek_inf}), we see that for any $k \in \nat^*$, we have
		\begin{equation}\label{eq_pf_fina_thm_10}
			N_k \geq  {\rm{Ban}}^{\infty}_{k}(FS(N)).
		\end{equation}
		\par 
		Now, through iteration of the submultiplicativity condition, (\ref{eq_reform_subm_cond}), for any $l \in \nat^*$, we have
		\begin{equation}\label{eq_nkl_bnd_fs}
			N_{k' l} \leq [N_{k'} \otimes_{\pi} \cdots \otimes_{\pi} N_{k'}],
		\end{equation}
		where the tensor product is repeated $l$ times.
		By the application of Theorem \ref{thm_induc}, (\ref{eq_fs_unif}) and (\ref{eq_nkl_bnd_fs}), we see that there is $l_0 \in \nat^*$, such that for any $l \geq l_0$, we have
		\begin{equation}\label{eq_nkl_bnd_fsa}
			N_{k' l} \leq \exp(2 \epsilon k'l / 3) \cdot {\rm{Ban}}^{\infty}_{k'l}(FS(N)).
		\end{equation}
		\par 
		Remark that since the spaces $H^0(X, L^p)$, $p = k', \ldots, 2k' - 1$, are finite dimensional, the norms $N_p$ and ${\rm{Ban}}^{\infty}_{p}(FS(N))$ are comparable up to a uniform constant. 
		From this and (\ref{eq_ban_inf_proj}), we deduce that there is $l_1 \in \nat$, such that for any $0 \leq r \leq k' - 1$, $l \geq l_1$, we have
		\begin{equation}\label{eq_pf_fina_thm_8}
			[{\rm{Ban}}^{\infty}_{k' l}(FS(N)) \otimes_{\pi} N_{k' + r}] \leq \exp(\epsilon k'l / 4) \cdot {\rm{Ban}}^{\infty}_{k'(l + 1) + r}(FS(N)).
		\end{equation}		 
		A combination of (\ref{eq_reform_subm_cond}), (\ref{eq_nkl_bnd_fsa}) and (\ref{eq_pf_fina_thm_8}) yields for $k \geq 2 k' \max \{ l_0, l_1\}$ the following estimate
		\begin{equation}\label{eq_pf_fina_thm_9}
			N_{k} \leq \exp(\epsilon k) \cdot {\rm{Ban}}^{\infty}_{k}(FS(N)).
		\end{equation}
		The result now follows directly from (\ref{eq_pf_fina_thm_10}) and (\ref{eq_pf_fina_thm_9}).
	\end{proof}
	\begin{rem}
		Similarly to \cite[Definition 1.3]{FinSecRing}, one can lighten the submultiplicativity assumption by requiring that there is $p_0 \in \nat$ and $f : \nat_{\geq p_0} \to \real$, verifying $f(k) = o(k)$, as $k \to \infty$, such that for any $r \in \nat^*$, $k; k_1, \ldots, k_r \geq p_0$, $k_1 + \cdots + k_r = k$, $f_i \in H^0(X, L^{k_i})$, $i = 1, \cdots, r$, we have
		\begin{equation}\label{eq_mult_gen}
			\| f_1 \cdots f_r \|_{k} \leq 
			\| f_1 \|_{k_1} \cdots \| f_r \|_{k_r}
			 \cdot
			 \exp \Big(
			 f(k_1) + \cdots + f(k_r) + f(k)
			 \Big).
		\end{equation}
		The proof in this case remains the same with only one modification: instead of the usual Fekete's lemma for the proof of the convergence of $FS(N_k)^{\frac{1}{k}}$, one needs to rely on \cite[Appendix A]{FinSecRing}.	
	\end{rem}
	\begin{proof}[Proof of Theorem \ref{thm_char2}]
		By our boundness assumption, the fact that $FS(N_{2^k})^{\frac{1}{2^k}}$, $k \in \nat$, decrease and Lemma \ref{lem_fs_dini}, we conclude from (\ref{eq_darvas_conv}) that for any $\epsilon > 0$, there is $r \in \nat^*$, such that
		\begin{equation}\label{eq_darvas_conv2}
			d_p \Big( FS(N)_*, FS(N_r)^{\frac{1}{r}} \Big) \leq \epsilon / 2.
		\end{equation}
		From Theorem \ref{thm_darvlurub_quant}, Proposition \ref{prop_ban_hilb_eq} and (\ref{eq_darvas_conv2}), we conclude that 
		\begin{equation}\label{eq_nkl_bnd_fsa2244}
			d_p \Big( {\rm{Ban}}^{\infty}(FS(N)_*), {\rm{Ban}}^{\infty}(FS(N_r)^{\frac{1}{r}}) \Big) \leq \epsilon / 2.
		\end{equation}
		From the proof of Theorem \ref{thm_char}, there is $k_0 \in \nat$, such that for any $k \geq k_0$, we have
		\begin{equation}\label{eq_nkl_bnd_fsa22}
			N_k \leq \exp(\epsilon k / 2) \cdot {\rm{Ban}}^{\infty}_{k}(FS(N_r)^{\frac{1}{r}}).
		\end{equation}
		From (\ref{eq_log_rel_spec2}), (\ref{eq_pf_fina_thm_10}), (\ref{eq_nkl_bnd_fsa2244}) and (\ref{eq_nkl_bnd_fsa22}), we conclude that  $d_p({\rm{Ban}}^{\infty}(FS(N)_*), N) \leq \epsilon$. 
		This finishes the proof by Corollary \ref{cor_psh_appr_below}, as $\epsilon > 0$ was chosen arbitrary.
	\end{proof}
	\par 
	We define for two bounded metrics $h^L_0$, $h^L_1$ on $L$, the distance $d_{+\infty} ( h^L_0, h^L_1)$ as the minimal constant $C \geq 0$ verifying $h^L_0 \leq e^C \cdot h^L_1$ and $h^L_1 \leq e^C \cdot h^L_0$.
	The following corollary establishes the relation between distances of submultiplicative graded norms and their Fubini-Study potentials.
	\begin{cor}\label{cor_d_p_norm_fs_rel}
		For any bounded submultiplicative graded norms $N, N'$ on a section ring $R(X, L)$, and any $p \in [1, +\infty[$, we have
		\begin{equation}
			d_p(N, N') 
			=
			d_p \big( FS(N)_*, FS(N')_* \big). 
		\end{equation}
		If, moreover, $FS(N)$, $FS(N')$ are continuous, then one can take $p = +\infty$ above.
	\end{cor}	
	\begin{proof}
		For $p = [1, +\infty[$, the statement follows directly from Theorems \ref{thm_char2}, \ref{thm_darvlurub_quant}, Proposition \ref{prop_ban_hilb_eq} and Corollary \ref{cor_psh_appr_below}.
		For $p = +\infty$, the statement is a consequence of Theorem \ref{thm_char}, Proposition \ref{prop_ban_hilb_eq} and \cite[Theorem 1.7]{FinSecRing}.
	\end{proof}
	\par 
	One can verify, cf. \cite[Lemma 4.12]{FinSecRing}, that for any graded norms $N, N'$, for which the sequences $FS(N_k)^{\frac{1}{k}}$, $FS(N'_k)^{\frac{1}{k}}$, $k \in \nat$, converge uniformly to some metrics $FS(N)$, $FS(N')$ on $L$, we have $d_{+\infty} ( FS(N), FS(N')) \leq  d_{+ \infty}(N, N')$.
	It is tempting to think that a similar conclusion holds for $d_p$-distances, $p \in [1, +\infty[$, or even the submultiplicativity assumption in Corollary \ref{cor_d_p_norm_fs_rel} is superfluous.
	It is not the case.
	The following example shows that the $d_p$-distances between norms and their Fubini-Study metrics are essentially unrelated. 
	\begin{prop}\label{prop_ex_dp_nnwork}
		There is a bounded graded Hermitian norm $H = \sum H_k$ on $R(\mathbb{P}^1, \mathscr{O}(1))$, such that $FS(H_k)^{\frac{1}{k}}$ converge uniformly, as $k \to \infty$, to a (continuous psh) metric $FS(H)$ on $L$, and there is a continuous psh metric $h^L \neq FS(H)$, for which $H \sim_p {\rm{Ban}}^{\infty}(h^L)$ for any $p \in [1, +\infty[$.
	\end{prop}
	\begin{sloppypar}
	\begin{rem}
		In particular, by Corollary \ref{cor_psh_appr_below}, for any $p \in [1, +\infty[$, we have $d_p(H, {\rm{Ban}}^{\infty}(FS(H))) \neq 0$ and $d_p(H, {\rm{Ban}}^{\infty}(h^L)) = 0$, while $d_p(FS(H), h^L) \neq 0$.
	\end{rem}
	\end{sloppypar}
	\begin{proof}
		Our proof is a slight modification of \cite[Proposition 4.16]{FinSecRing}.
		Let us identify $\mathbb{P}^1$ to $\mathbb{P}(V^*)$, where $V$ is a vector space generated by two elements: $x$ and $y$.
		Let us consider a metric $H$ on $V$, which makes $x$ and $y$ an orthonormal basis, and denote by $h^{FS}$ the induced Fubini-Study metric on $\mathscr{O}(1)$.
		For any $k \in \nat^*$, $a, b \in \nat$, $a + b = k$, under the isomorphism ${\rm{Sym}}^k(V) \to H^0(\mathbb{P}(V^*), \mathscr{O}(k))$, an easy calculation shows that we have
		\begin{equation}\label{eq_l2_norm_calc}
			\big\| x^a \cdot y^b \big\|_{{\rm{Hilb}}_k(h^{FS})}^2
			=
			\frac{a! b!}{(k + 1)!}.
		\end{equation}
		\par
		Let us consider a Hermitian norm $H_k$ on $H^0(\mathbb{P}(V^*), \mathscr{O}(k))$, for which the basis $x^a \cdot y^b$ is orthogonal and in the above notations, we have
		\begin{equation}
			\big\| x^a \cdot y^b \big\|_{H_k}^2
			=
			\begin{cases}
				\frac{1}{2^k + k + 1}, &\text{if } b = 0 \\
				\big\| x^a \cdot y^b \big\|_{{\rm{Hilb}}_k(h^{FS})}^2 ,  &\text{otherwise}.
			\end{cases}
		\end{equation}
		We will now verify that $H_k$ satisfies the assumptions of the proposition.
		\par 
		First of all, from Proposition \ref{prop_ban_hilb_eq}, it is trivial to verify that for any $p \in [1, +\infty[$, we have $H \sim_p {\rm{Ban}}^{\infty}(h^L)$.
		Remark also that $H$ is bounded by Proposition \ref{prop_ban_hilb_eq}, since we have $H \geq {\rm{Hilb}}(h^{FS} / 3)$.
		From Lemma \ref{lem_fs_inf_d}, for any $a, b \in \comp$, not simultaneously equal to zero, we have
	 	\begin{equation}\label{eq_expl_exmpl11}
	 		\frac{FS({\rm{Hilb}}_{k}(h^{FS}))^2}{FS(H_{k})^2} \Big([a x^* + b y^*] \Big)
	 		=
	 		\frac{2^k |a|^k + (k + 1) (|a| + |b|)^{k}}{(k + 1) (|a| + |b|)^{k}}.
	 	\end{equation}
	 	In particular, we conclude that 
	 	\begin{equation}
	 		\lim_{k \to \infty} \bigg( \frac{FS(H_{k})}{FS({\rm{Hilb}}_{k}(h^{FS}))} \bigg)^{\frac{2}{k}} \Big([a x^* + b y^*] \Big)
	 		=
	 		\frac{\max\{2 |a|, |a| + |b|\}}{|a| + |b|},
	 	\end{equation}
	 	and the convergence is uniform. 
	 	This finishes the proof by Tian's theorem, cf. Theorem \ref{thm_bouck_erikss}.
	\end{proof}
	\par
	We will now show that one cannot take $p = + \infty$ in Theorem \ref{thm_char2}.
	Recall that to any complex normed commutative ring $(A, \| \cdot \|_A)$, one can associate the seminorm $\| \cdot \|_A^{{\rm{hom}}}$, sometimes called the \textit{homogenization} or \textit{spectral radius seminorm}, defined as follows
	\begin{equation}\label{eq_homog_defn}
		\| f \|_A^{{\rm{hom}}}
		:=
		\lim_{k \to \infty} \| f^k \|_A^{\frac{1}{k}}.
	\end{equation}
	The existence of the limit above is assured by Fekete's lemma.
	
	\begin{prop}\label{prop_example}
		There is a bounded submultiplicative graded norm $N = \sum N_k$ on a section ring $R(X, L)$, such that $N \not\sim {\rm{Ban}}^{\infty}(h^L)$ for any bounded metric $h^L$ on $L$.
	\end{prop}
	\begin{proof}
		We fix a bounded metric $h^L$ on $L$, an effective divisor $D \subset X$ and consider the ray of norms $N_k^{t} = \| \cdot \|_k^{t}$, $t \in [0, +\infty[$, constructed by the procedure (\ref{eq_ray_norm_defn2}) from the norm ${\rm{Ban}}^{\infty}(h^L)$ and the filtration, for which jumping numbers are given by $ik$, $i = 0, 1, 2$, and such that $\mathcal{F}^{ik} H^0(X, L^{\otimes k}) = H^0(X, L^{\otimes k} \otimes \mathcal{J}_D^i)$, where $\mathcal{J}_D$ is a sheaf of holomorphic germs vanishing along $D$.
		\par 
		Let us verify that the graded norm $N := \sum N_k^{1}$ provides an example for Proposition \ref{prop_example}.
		An easy verification shows that it is submultiplicative.
		It is also trivially bounded from below by ${\rm{Ban}}^{\infty}(e^{-2}h^L)$.
		Let us show that $N \not\sim {\rm{Ban}}^{\infty}(h^L)$ for any metric $h^L$ on $L$.
		\par 
		Indeed, let us consider a sequence of elements $f_k \in H^0(X, L^{\otimes k} \otimes \mathcal{J}_D) \setminus H^0(X, L^{\otimes k} \otimes \mathcal{J}_D^2)$. 
		The existence of such $f_k$ for $k$ large enough is assured by ampleness of $L$ and effectivity of $D$.
		We denote $g_k := f_k - P_k(f_k)$, where $P_k(f_k)$ is a projection (with respect to the norm ${\rm{Ban}}^{\infty}_k(h^L)$) of $f_k$ to $H^0(X, L^{\otimes k} \otimes \mathcal{J}_D^2)$.
		We then see that $
			\| g_k \|_k^{1}
			=
			e^{- k} \cdot 
			\| g_k \|_{L^{\infty}_{k}(X, h^L)}
		$.
		However, in the notations of Section \ref{sect_homog}, we obviously have $\| g_k \|_k^{1, \rm{hom}}
			\leq
			\sqrt{
				\| g_k^2 \|^{1}_{2k}
			}$. 
		Since $g_k^2 \in H^0(X, L^{2k} \otimes \mathcal{J}_D^2)$, we deduce that
		$
			\| g_k \|_k^{1, \rm{hom}}
			\leq
			e^{- 2k} \cdot 
			\| g_k \|_{L^{\infty}_{k}(X, h^L)}
		$.
		Hence, $N \not\sim N^{{\rm{hom}}}$, which implies that $N \not\sim {\rm{Ban}}^{\infty}(h^L)$ for any bounded metric $h^L$ on $L$.
	\end{proof}
	\begin{rem}\label{rem_ex_non_cont_ray}
		Our example is given by a ray of submultiplicative norms constructed using a non-Archimedean submultiplicative norm from Boucksom-Jonsson \cite[Example 2.25]{BouckJohn21} through a general procedure outlined in (\ref{eq_ray_norm_defn0}).
		As it follows from Theorems \ref{thm_char} and \ref{thm_ray_geod_bergm}, the reason why we have $N \not\sim {\rm{Ban}}^{\infty}(FS(N))$ is that the Fubini-Study geodesic ray associated with our filtration turns out to be non-continuous, see Section \ref{sect_pf_filt} for details.
	\end{rem}
	
	\subsection{Semiclassical holomorphic extension theorem and pluripolar sets}\label{sect_hol_ext_pp}
	The main goal of this section is to deduce from Theorem \ref{thm_char2} a characterization of submanifolds for which a weak version of semiclassical Ohsawa-Takegoshi extension theorem holds.
	\par 
	Recall that Ohsawa-Takegoshi in \cite{OhsTak1} gave a sufficient condition under which a holomorphic section of a vector bundle on a submanifold extends to a holomorphic section over an ambient manifold with a reasonable bound on the $L^2$-norm of the extension in terms of the $L^2$-norm of the section. 
	Later in \cite{FinOTAs}, the author proved a more precise statement in the semiclassical limit, i.e. when the vector bundle is given by a sufficiently high tensor power of a fixed positive line bundle.
	In particular, in \cite[Theorem 1.1]{FinOTAs}, we established an asymptotically optimal semiclassical version of Ohsawa-Takegoshi extension theorem and in \cite[Theorem 1.10]{FinOTAs} we proved its version for sup-norms. See also \cite[Theorems 1.1 and 1.3]{FinOTRed} for a more general statement about jet extensions and Zhang \cite{ZhangPosLinBun}, Bost \cite{BostDwork} for related previous works.
	\par 
	The regularity of the line bundle and its strict positivity were crucial in latter developments. 
	Since many constructions in complex geometry (as those arising from envelopes) yield non-regular metrics with weak positivity, it is natural to ask to which extent our results remain valid in these circumstances. 
	As an application of Theorem \ref{thm_char2}, we give in Theorem \ref{thm_ot_weak} a characterization of submanifolds for which a weak analogue of the semiclassical extension theorem holds.
	\par 
	Let us first set up the notations.
	Let $Y$ be a closed submanifold of a compact complex manifold $X$ and $L$ be an ample line bundle over $X$.
	It is classical that there is $k_0 \in \nat$, such that for any $k \geq k_0$, the map	
	\begin{equation}\label{eq_res_y_mor}
	 	\res_Y : H^0(X, L^{\otimes k}) \to H^0(Y, L|_Y^{\otimes k}),
	 \end{equation}
	is surjective. Hence, a norm on $H^0(X, L^{\otimes k})$ induces a norm on $H^0(Y, L|_Y^{\otimes k})$.
	In this language, Ohsawa-Takegoshi extension theorem basically compares the two norms on $H^0(Y, L|_Y^{\otimes k})$: one induced from the metric on $Y$, another one is the quotient norm induced from the metric on $X$.
	Let us now recall a semiclassical version of holomorphic extension theorem for continuous metrics.
 	\begin{thm}\label{thm_ot_semi}
 		For any continuous psh metric $h^L$ on an ample line bundle $L$ over a compact complex manifold $X$, under surjection (\ref{eq_res_y_mor}), the following equivalence of norms on $R(Y, L)$ holds
		$
			[{\rm{Ban}}^{\infty}_{X}(h^L)] \sim {\rm{Ban}}^{\infty}_{Y}(h^L)
		$.
	\end{thm}
	\begin{proof}
		This result was proved by Bost \cite[Theorem A.1]{BostDwork} with stronger assumption of strict positivity on the curvature of $(L, h^L)$, refining previous result of Zhang \cite{ZhangPosLinBun}.
		See also Randriambololona \cite{RandriamCrelle} for a statement which requires laxer assumptions on the manifolds $X$ and $Y$.
		The proof of exactly this version of the theorem can be found in \cite[Corollary 2.12]{FinSecRing}.
	\end{proof}
	We will now show that for non-continuous metrics $h^L$, things become far more complicated, and the analogue of Theorem \ref{thm_ot_semi} only holds generically.
	Recall that a subset $E \subset X$ is called \textit{pluripolar} if it is a subset of a \textit{complete pluripolar} set, where the latter is defined as $\{ x \in X : u(x) = - \infty \}$ for a certain $u \in {\rm{PSH}}(X, \omega)$ and a Kähler form $\omega$.
	These definitions do not depend on the choice of the form $\omega$, cf. \cite[Proposition 2.3]{GuedZeriGeomAnal}.
	\par 
	The main result of this section goes as follows.
 	\begin{thm}\label{thm_ot_gen}
 		A bounded psh metric $h^L$ on $L$ is regularizable from above if and only if for any $p \in [1, +\infty[$, under (\ref{eq_res_y_mor}), for generic submanifolds $Y$ of $X$, we have
		\begin{equation}\label{eq_ot_gen}
			[{\rm{Ban}}^{\infty}_{X}(h^L)] \sim_p {\rm{Ban}}^{\infty}_{Y}(h^L).
		\end{equation}
		Generic here means that it holds for $Y$ not contained in a certain pluripolar subset of $X$.
	\end{thm}
	\par 
	In order to prove Theorem \ref{thm_ot_gen}, recall that Bedford-Taylor in \cite[Theorem 7.1]{BedfordTaylor} proved in local setting that for a uniformly bounded sequence $\phi_i$, $i \in \nat$ of psh functions, we have $\sup \phi_i = (\sup \phi_i)^*$ away from a pluripolar subset.
	Since on a Kähler manifold $X$, any locally pluripolar subset is pluripolar by a theorem of Josefson, cf. \cite[Theorem 7.2]{GuedZeriGeomAnal}, the same conclusion holds for functions from ${\rm{PSH}}(X, \omega)$.
	\par 
	Now, from Theorem \ref{thm_bouck_erikss}, we see that to study sup-norms, it is enough to consider regularizable from above psh metrics $h^L$.
	For such $h^L$, we define the \textit{contact subset} $E(h^L) \subset X$ as follows
	\begin{equation}\label{eq_e_sub_defn}
		E(h^L) := \big\{ x \in X : h^L_x \neq Q(h^L)_x \big\},
	\end{equation}	 
	where $Q(h^L)$ is the Fubini-Study envelope of $h^L$, defined in (\ref{eq_defn_fs_env}).
	From the above results of Bedford-Taylor and Josefson, the set $E(h^L)$ is pluripolar.
	The following result gives a criteria for a weaker version of semiclassical Ohsawa-Takegoshi extension theorem to hold.
	It classifies submanifolds for which the asymptotic contribution of holomorphic sections over the submanifold, which cannot be effectively extended to the ambient manifold, is “negligible".
	\begin{thm}\label{thm_ot_weak}
		For any regularizable from above psh metrics $h^L$ on an ample line bundle over a compact complex manifold $X$, the following conditions are equivalent.
		\par 
		1) A submanifold $Y \subset X$ intersects $E(h^L)$ over a pluripolar subset (of $Y$).
		\par 
		2) For any $p \in [1, +\infty[$, under (\ref{eq_res_y_mor}), the equivalence of norms (\ref{eq_ot_gen}) holds.
		\par 
		Moreover, (\ref{eq_ot_gen}) holds for $p = + \infty$ if $Y \cap E(h^L) = \emptyset$.
	\end{thm}
	\begin{sloppypar}
	\begin{proof}
		The norm $N := [{\rm{Ban}}^{\infty}_{X}(h^L)]$ on $R(Y, L)$ is submultiplicative as a quotient of a submultiplicative norm.
		From Lemma \ref{lem_fs_dini}, we have $FS(N) = FS({\rm{Ban}}^{\infty}_{X}(h^L))|_Y$.
		From Theorem \ref{thm_bouck_erikss}, this yields $FS(N)  = Q(h^L)|_Y$.
		Hence, by Theorem \ref{thm_char2} and Corollaries \ref{cor_baninf_dist}, \ref{cor_psh_appr_below}, we conclude that  (\ref{eq_ot_gen}) holds if and only if $(Q(h^L)|_Y)_* = h^L|_Y$.
		Remark, however, that by the already mentioned result \cite[Theorem 7.1]{BedfordTaylor} of Bedford-Taylor, this happens if and only if $Q(h^L)|_Y = h^L|_Y$ away from a pluripolar subset (of $Y$).
		This concludes the proof of the first part of Theorem \ref{thm_ot_weak} by the definition of the subset $E(h^L)$.
		\par 
		Now, points of discontinuity of $h^L$ are contained in $E(h^L)$.
		This is due to the fact that $h^L$ is lower semi-continuous, $Q(h^L)$ is upper semi-continuous and $Q(h^L) \geq h^L$.
		In particular, $h^L|_Y$ is continuous if $Y \cap E(h^L) = \emptyset$.
		By the above, we also have $FS(N)  = h^L|_Y$ under the assumption $Y \cap E(h^L) = \emptyset$.
		The second part of Theorem \ref{thm_ot_weak} now follows from Theorem \ref{thm_char}.
	\end{proof}
	\end{sloppypar}
	\begin{proof}[Proof of Theorem \ref{thm_ot_gen}]
		We first assume that $h^L$ is a regularizable from above psh metric.
		 Let $E^*(h^L)$ be the \textit{pluripolar hull} of $E(h^L)$, i.e. the intersection of all complete pluripolar subsets in $X$ containing $E(h^L)$.
		The subset $E^*(h^L)$ is clearly pluripolar, and when $Y$ is not contained in $E^*(h^L)$, it intersects $E^*(h^L)$ over a pluripolar subset (of $Y$).
		Hence, one direction of Theorem \ref{thm_ot_gen} follows from Theorem \ref{thm_ot_weak}.
		\par 
		Let us establish the opposite direction.
		We fix a bounded psh metric $h^L$ on $L$.
		From the proof of Theorem \ref{thm_ot_weak}, we see that (\ref{eq_ot_gen}) holds for generic points $Y := \{ x \}$, $x \in X$, if and only if the identity $h^L_x = Q(h^L)_x$ holds away from a pluripolar set.
		The latter is clearly equivalent to the fact that $h^L$ is regularizable from above.
	\end{proof}

	\subsection{Projective geometry and norms on symmetric algebras}\label{sect_norm_proj}
	In this section, we reduce the proof of Theorem \ref{thm_induc} to a functional-analytic statement about the norms on the symmetric algebra of complex vector spaces. 
	We also show that the latter statement can be seen as a special case of Theorem \ref{thm_induc}, applied for the projective space.
	\par 
	\begin{sloppypar}
	We fix a finite dimensional complex vector space $V$ with a norm $N_V := \| \cdot \|_V$.
	Recall that for any $k \in \nat^*$, we have the \textit{polarisation} map ${\rm{Pol}} : \sym^k(V) \to V^{\otimes k}$ and the \textit{symmetrization} map $\sym : V^{\otimes k} \to \sym^k(V)$.
	Consider two norms $\sym^k_{\epsilon}(N_V) := \| \cdot \|^{\sym, \epsilon}_{N_V, k}$ and $\sym^k_{\pi}(N_V) := \| \cdot \|^{\sym, \pi}_{N_V, k}$ on symmetric tensors $\sym^k(V)$, induced by the polarisation map, and the norms $N_V \otimes_{\epsilon} \cdots \otimes_{\epsilon} N_V$, $N_V \otimes_{\pi} \cdots \otimes_{\pi} N_V$ on $V^{\otimes k}$.
	Define the norm $\sym^k_{{\rm{ev}}}(N_V) := \| \cdot \|^{{\rm{ev}}}_{N_V, k}$ on $\sym^k(V)$ as
	\begin{equation}\label{eq_sup_norm_polll}
		\| P \|^{{\rm{ev}}}_{N_V, k}
		:=
		\sup_{\substack{v \in V^* \\ \| v \|_V^* \leq 1}} |P(v)|, 
		\qquad
		P \in \sym^k(V).
	\end{equation}
	We construct from these norms the graded norms $\sym_{{\rm{ev}}}(N_V)$, $\sym_{\epsilon}(N_V)$ and $\sym_{\pi}(N_V)$ on the symmetric algebra $\sym(V)$.
	Similarly to (\ref{defn_equiv_rel}), we define the equivalence relation on the set of graded norms over $\sym(V)$.
	The following result will be established in Sections \ref{sect_bohn_hill} and \ref{sect_proj_t_ot}.
	\end{sloppypar}
	\begin{thm}\label{thm_sym_equiv}
		The norms $\sym_{\pi}(N_V)$, $\sym_{\epsilon}(N_V)$ and $\sym_{{\rm{ev}}}(N_V)$ are equivalent.
	\end{thm}
	\begin{rem}\label{rem_sym_equiv}
		a) Restriction to symmetric tensors is absolutely necessary for this statement. In fact, as it follows from the work of Pisier \cite[Théorème 3.1]{PisierDecart}, see also more recent result of Aubrun-M{\"u}ller-Hermes \cite[Theorem 1.1]{AuburnMull}, in the full tensor algebra $T(V) := \sum_{k = 1} V^{\otimes i}$, the gap between injective and projective tensor norms on the graded pieces is exponential for any normed vector space $(V, N_V)$ of dimension bigger than $1$. 
		\par 
		b) Surprisingly, the corresponding statement for real vector spaces is false.
		In fact, if we consider a polynomial $P(x, y) = xy(x^2 - y^2)$ and view it as a polynomial on $(\mathbb{R}^2, \textit{l}_1)$, then an easy calculation shows that for any $k \in \nat^*$, we have $\| P^k \|^{{\rm{ev}}}_{\textit{l}_1, 4k} = \sup_{-1 \leq x, y \leq 1} |P^k(x, y)| = \big( \frac{2 \sqrt{3}}{9} \big)^k$, cf. \cite[proof of Theorem 4.2]{RealBohnHil}. 
		But from the proof of Theorem \ref{thm_sym_equiv}, we know that $\sym_{\pi}(\textit{l}_1)$ corresponds to the sum of the absolute values of the coefficients.
		Hence, we have $\| P^k \|^{\sym, \pi}_{\textit{l}_1, 4k} = 2^k$.
	\end{rem}
	\par 
	We now explain that Theorem \ref{thm_sym_equiv} is in fact a special case of Theorem \ref{thm_induc}.
	For this, we give geometric interpretations for some of the above norms. 
	First of all, directly from (\ref{eq_defn_inf_norm}), we have
	\begin{equation}\label{eq_pol_epsnm}
		\| P \|^{\sym, \epsilon}_{k}
		:=
		\sup_{\substack{v_1, \cdots, v_k \in V^* \\ \| v_i \|_V^* \leq 1}} \Big| {\rm{Pol}}(P) \big(v_1, \cdots, v_k \big) \Big|, 
		\qquad
		P \in \sym^k(V).
	\end{equation}
	Hence, from Lemma \ref{lem_inj_proj_bnd_nntr}, (\ref{eq_sup_norm_polll}) and (\ref{eq_pol_epsnm}), the following chain of inequalities holds
	\begin{equation}\label{eq_compar_3_nmrs}
		\sym_{{\rm{ev}}}(N_V) \leq \sym_{\epsilon}(N_V) \leq \sym_{\pi}(N_V).
	\end{equation}
	In particular, we see that for the proof of Theorem \ref{thm_sym_equiv}, it is enough to establish the equivalence of the norms $\sym_{{\rm{ev}}}(N_V)$ and $\sym_{\pi}(N_V)$.
	\begin{sloppypar}
	\begin{rem}	
		By (\ref{eq_pol_epsnm}), the equivalence of $\sym_{{\rm{ev}}}(N_V)$ and $\sym_{\epsilon}(N_V)$ from Theorem \ref{thm_sym_equiv} means exactly that the \textit{polarisation constant}, cf. \cite[(4)]{PolarConst} for a definition, for finite dimensional complex normed vector spaces is equal to $1$.
		This fact was recently established by Dimant-Galicer-Rodr{\'{\i}}guez \cite[Theorem 1.1]{PolarConst} through different methods.
	\end{rem}
	\end{sloppypar}
	\par 
	Let us now give an interpretation of the norm $\sym_{{\rm{ev}}}(N_V)$ through projective spaces.
	We view the symmetric algebra $\sym(V)$ as the section ring $R(\mathbb{P}(V^*), \mathscr{O}(1))$ through the identification
	\begin{equation}\label{eq_sym_isom}
		\sym^k(V) = H^0(\mathbb{P}(V^*), \mathscr{O}(k)).
	\end{equation}
	Under this isomorphisms, we have the following identification of norms
	\begin{equation}\label{eq_sum_ev_norm}
		\sym_{{\rm{ev}}}(N_V) = {\rm{Ban}}^{\infty}_{\mathbb{P}(V^*)}(FS(N_V)).
	\end{equation}
	\par 
	We now consider the norms $\sym^k_{\epsilon, 0}(N_V)$ and $\sym^k_{\pi, 0}(N_V)$ on $\sym^k(V)$, given by the quotients of $N_V \otimes_{\epsilon} \cdots \otimes_{\epsilon} N_V$, $N_V \otimes_{\pi} \cdots \otimes_{\pi} N_V$ on $V^{\otimes k}$ through the symmetrization map, ${\rm{Sym}}$. 
	\begin{lem}\label{lem_sym_quot_inj}
		The norms $\sym^k_{\epsilon, 0}(N_V)$ (resp. $\sym^k_{\pi, 0}(N_V)$) and $\sym^k_{\epsilon}(N_V)$ (resp. $\sym^k_{\pi}(N_V)$) over $\sym^k(V)$ coincide.
	\end{lem} 
	\begin{proof}
		It follows directly from the fact that permutations of coordinates are isometries for both norms $N_V \otimes_{\epsilon} \cdots \otimes_{\epsilon} N_V$, $N_V \otimes_{\pi} \cdots \otimes_{\pi} N_V$. 
	\end{proof}
	Remark that symmetrization and multiplication maps (\ref{eq_mult_map}) can be put under the isomorphisms (\ref{eq_sym_isom}) into the following commutative diagram 
	\begin{equation}\label{eq_comm_diag222}
	\begin{CD}
		H^0(\mathbb{P}(V^*), \mathscr{O}(1))^{\otimes k}  @> {\rm{Mult}}_{1, \cdots, 1} >> H^0(\mathbb{P}(V^*), \mathscr{O}(k)) 
		\\
		@| {}  @| {} 
		\\
		V^{\otimes k}  @> \sym >> \sym^k(V).
	\end{CD}
	\end{equation}
	\par 
	Lemma \ref{lem_sym_quot_inj}, (\ref{eq_sum_ev_norm}) and (\ref{eq_comm_diag222}) imply that Theorem \ref{thm_sym_equiv} is a specialisation of Theorem \ref{thm_induc} to $X = \mathbb{P}(V^*)$, $L = \mathscr{O}(1)$ and $N_1 := N_V$.
	Remark, however, that our proof proceeds in another direction: we first establish Theorem \ref{thm_sym_equiv} and then prove Theorem \ref{thm_induc}.
	\begin{proof}[Proof of Theorem \ref{thm_induc} assuming Theorem \ref{thm_sym_equiv}.]
		Let us first prove the following inequalities
		\begin{equation}\label{eq_norms_order_ep_ev}
			{\rm{Ban}}^{\infty}(FS(N_1)) \leq N^{\epsilon} \leq N^{\pi}.
		\end{equation}
		The first inequality is a direct consequence of (\ref{eq_norm_with_fs_compar}) and \cite[Lemma 4.3]{FinSecRing}.
		The second inequality follows directly from Lemma \ref{lem_inj_proj_bnd_nntr}.
		\begin{sloppypar}
		From (\ref{eq_norms_order_ep_ev}), it is enough to establish that the norm $N^{\pi}$ can be bounded from above by ${\rm{Ban}}^{\infty}(FS(N_1))$, considered up to a subexponential factor.
		The proof of this result is essentially a word-to-word repetition of the proof of the second part of \cite[Theorem 4.18]{FinSecRing}.
		We only need to replace the use of the first part of the proof of \cite[Theorem 4.18]{FinSecRing} by Theorem \ref{thm_sym_equiv}.
		For the convenience of the reader, we reproduce the argument below.
		\end{sloppypar}
		\begin{sloppypar}
		Let us consider the Kodaira embedding ${\rm{Kod}}_1$ from (\ref{eq_kod}).
		We denote by $\res_{{\rm{Kod}}} : R(\mathbb{P}(H^0(X, L)^*), \mathscr{O}(1)) \to R(X, L)$ the associated restriction operator, and by $\res_{{\rm{Kod}}, k}$, $k \in \nat^*$, the restriction operators on the associated graded pieces.
		The multiplication operator ${\rm{Mult}}_{1, \cdots, 1}$ from (\ref{eq_mult_map}) factorizes under the identification (\ref{eq_sym_isom}) through symmetrization and restriction as
		\begin{equation}\label{eq_kod_map_comm_d}
		\begin{tikzcd}
			H^0(X, L)^{\otimes k} \arrow[swap, rrdd, "{\rm{Mult}}_{1, \cdots, 1}"] \arrow[r, "\sym"] & {\rm{Sym}}^k(H^0(X, L)) \arrow[rd, equal] &  \\
			&  & H^0(\mathbb{P}(H^0(X, L)^*), \mathscr{O}(k)) \arrow[d, "\res_{{\rm{Kod}}, k}"]  \\
 	 		&  & H^0(X, L^{\otimes k}).
    	\end{tikzcd}
		\end{equation}
		\end{sloppypar}
		\par 
		Now, from (\ref{eq_kod_map_comm_d}), it is sufficient to show that by a subsequent quotient of the projective tensor norm induced by $N_1$ through the symmetrization map $\sym$ and the map $\res_{{\rm{Kod}}}$, we get the norm ${\rm{Ban}}^{\infty}_{X}(FS(N_1))$ on $H^0(X, L^{\otimes k})$.
		From Theorem \ref{thm_sym_equiv}, this quotient norm on ${\rm{Sym}}(H^0(X, L))$ is equivalent to $\sym_{{\rm{ev}}}(N_V)$, which by (\ref{eq_sum_ev_norm}) coincides with ${\rm{Ban}}^{\infty}_{\mathbb{P}(H^0(X, L)^*)}(FS(N_1))$ under the identification (\ref{eq_sym_isom}).
		But by Theorem \ref{thm_ot_semi}, the quotient of the norm ${\rm{Ban}}^{\infty}_{\mathbb{P}(H^0(X, L)^*)}(FS(N_1))$ under the map $\res_{{\rm{Kod}}}$ is equivalent to ${\rm{Ban}}^{\infty}_{X}(FS(N_1))$. 
		This finishes the proof.
	\end{proof}

	\subsection{Homogenization of submultiplicative norms on section rings}\label{sect_homog}
	The main goal of this section is to give an explicit formula for the spectral radius seminorm, defined in (\ref{eq_homog_defn}), associated with a submultiplicative graded norm on a section ring.
	We also discuss more precisely the connection between the current work and the previous works on the non-Archimedean analogues of Theorems \ref{thm_char}, \ref{thm_char2}.
	\par 
	Let us first recall why the spectral radius seminorm is a seminorm. 
	We assume that $A$ has a unit (otherwise one can formally add it).
	Gelfand’s spectral radius formula says, cf. \cite[Theorems I.4.4, I.4.6]{GelfCommNormRing}, that $\| \cdot \|_A^{{\rm{hom}}}$ can be alternatively described as follows
	\begin{equation}\label{eq_spec_hom}
		\| f \|_A^{{\rm{hom}}} 
		=
		\max \big\{ 
			|\lambda| : \lambda \in {\rm{spec}}(f)
		\big\},
	\end{equation}
	where ${\rm{spec}}(f)$ is the spectrum of $f$, given by $\lambda \in \comp$ such that $\lambda - f$ is not invertible in the completion of  $(A, \| \cdot \|_A)$.
	The formula (\ref{eq_spec_hom}) implies that $\| \cdot \|_A^{{\rm{hom}}}$ is a seminorm.
	\begin{thm}\label{thm_hom}
		Assume that a graded norm $N = \sum N_k$ over the section ring $R(X, L)$ of an ample line bundle $L$ is bounded and submultiplicative. Then $
			N^{{\rm{hom}}} = {\rm{Ban}}^{\infty}(FS(N))$.
	\end{thm}
	\begin{rem}
		In non-Archimedean setting, where a submultiplicative norm is replaced by a submultiplicative filtration, a result analogous to Theorem \ref{thm_hom}, was established by Rees \cite[Corollary on p.168]{ReesValI}, cf. also \cite[\S 4.1]{ReesBook}, for Noetherian filtrations, and by Boucksom-Jonsson \cite[Theorem 2.16]{BouckJohn21} for bounded submultiplicative filtrations.
	\end{rem}
	\begin{proof}
		From (\ref{eq_fek_inf}), (\ref{eq_norm_with_fs_compar}) and (\ref{eq_nkl_bnd_fsa22}), we conclude that for any $l \in \nat^*$, we have
		\begin{equation}
			{\rm{Ban}}^{\infty}(FS(N)) \leq N^{{\rm{hom}}} \leq {\rm{Ban}}^{\infty}(FS(N_l)^{\frac{1}{l}})
		\end{equation}
		It is, hence, enough to establish that $
			{\rm{Ban}}^{\infty}(FS(N)) = \inf_{l \in \nat} {\rm{Ban}}^{\infty}(FS(N_l)^{\frac{1}{l}})
		$.
		\par 
		Clearly, by considering a subsequence $l = 2^k$, $k \in \nat$, from (\ref{eq_fek_inf}), it is enough to prove that for a decreasing sequence of upper semi-continuous functions $\phi_i$ on a compact manifold $X$, the following identity holds
		$
			\lim_{i \to \infty} \sup_{x \in X} \phi_i(x)
			=
			\sup_{x \in X} \lim_{i \to \infty} \phi_i(x)
		$.
		It is trivial that $\lim_{i \to \infty} \sup_{x \in X} \phi_i(x) \geq \sup_{x \in X} \lim_{i \to \infty} \phi_i(x)$.
		To argue in another direction, we mimic the proof of Dini's theorem.
		Let us denote $M = \sup_{x \in X} \lim_{i \to \infty} \phi_i(x)$. 
		Then for any $\epsilon > 0$, the sets $U_i = \{ x \in X : \phi_i(x) < M + \epsilon \}$ provide a cover of $X$.
		By our assumption on upper semi-continuity of $\phi_i$, this cover is open.
		By compactness of $X$, there is a finite subcover.
		But since $\phi_i$ decrease, the sets $U_i$ are nested, and hence there is $i_0 \in \nat$, such that $U_{i_0} = X$. 
		Hence, $\sup_{x \in X} \phi_{i_0}(x) < M + \epsilon$.
		Since $\epsilon > 0$ was chosen arbitrary and $\phi_i$ decrease, we deduce the inverse direction  $\lim_{i \to \infty} \sup_{x \in X} \phi_i(x) \leq M$, which finishes the proof.
	\end{proof}
	\begin{rem}
		We learned from Sébastien Boucksom that one can alternatively prove Theorem \ref{thm_hom} relying only on the Gelfand’s spectral radius formula and the interpretation of the spectrum of a section ring as an affine cone, very much in spirit of Fang \cite[Proposition 3.15]{FangNonArch}.
	\end{rem}
	In particular, directly from Theorems \ref{thm_char}, \ref{thm_char2} and \ref{thm_hom}, we deduce the following result.
	\begin{cor}\label{cor_str_val}
		For any bounded submultiplicative graded norm $N$ over the section ring $R(X, L)$ of an ample line bundle $L$ and any $p \in [1, +\infty[$, we have
		$
			N \sim_p N^{{\rm{hom}}}
		$.
		If, moreover, $FS(N)$ is continuous, then one can even take $p = +\infty$ above.
	\end{cor}
	\begin{rem}
		In non-Archimedean setting, a result analogous to Corollary \ref{cor_str_val}, was established by Rees \cite[Theorem 3.4]{ReesValII}, cf. also \cite[\S 5.3]{ReesBook}, \cite[Theorem 2.3]{BouckJohn21}, for Noetherian filtrations on rings such that their localizations over maximal ideals are analytically unramified (in particular, the result applies to finitely generated submultiplicative filtrations on section rings, see \cite[\S 9]{SwanHuneke}).
		\par 
		Unlike in this article, in \cite{ReesValII}, \cite{BouckJohn21}, the analogues of Theorems \ref{thm_char},  \ref{thm_char2} follow from the analogues of Theorem \ref{thm_hom} and Corollary \ref{cor_str_val}.
		The techniques of the proofs of the analogue of Corollary \ref{cor_str_val} from \cite{ReesValII}, \cite{BouckJohn21} rely on the non-Archimedean analysis, which do not seem to adapt in our setting.
		It is interesting if one can prove Corollary \ref{cor_str_val} without having established Theorems \ref{thm_char}, \ref{thm_char2} beforehand, providing an alternative approach to the main results of this article.
	\end{rem}

	\section{Norms on spaces of polynomials}\label{sec_norm_poly_big}
	The main goal of this self-contained section is to establish Theorem \ref{thm_sym_equiv}.
	In Section \ref{sect_bohn_hill}, we establish Theorem \ref{thm_sym_equiv} in the special case  $V = \comp^r$, $r \in \nat^*$, and $N_V := \| \cdot \|_V := \textit{l}_1$, and in Section \ref{sect_proj_t_ot}, we prove Theorem \ref{thm_sym_equiv} in its full generality by relying on some tools from complex geometry.
	
	\subsection{Bohnenblust-Hille inequality as ratio of injective and projective norms}\label{sect_bohn_hill}
	The main goal of this section is to establish Theorem \ref{thm_sym_equiv} in the special case  $V = \comp^r$, $r \in \nat^*$, and $N_V := \| \cdot \|_V := \textit{l}_1$.
	To establish this, we rely on a recent result about the optimal estimate in Bohnenblust-Hille inequality, which we now recall.
	Consider a vector space $V_{r, k}$ of homogeneous complex polynomials of degree $k$ in $r$ variables.
	We represent an element $P \in V_{r, k}$ as
	\begin{equation}\label{eq_poly}
		P(x_1, \cdots, x_r)
		=
		\sum_{|\alpha| = k} a_{\alpha} x^{\alpha}.
	\end{equation}
	Since $\dim V_{r, k} = \binom{r + k}{r} < +\infty$, any two norms on $V_{r, k}$ are equivalent.
	In particular, for any $\beta \geq 1$, there is a constant $B_{r, k}^{\beta} > 0$, such that for any $P \in V_{r, k}$ as in (\ref{eq_poly}), we have
	\begin{equation}\label{eq_boh_hil}
		\Big( 
			\sum_{|\alpha| = k} |a_{\alpha}|^{\beta}
		\Big)^{\frac{1}{\beta}}
		\leq
		B_{r, k}^{\beta}
		\cdot
		\| P \|,
	\end{equation}
	where the sup-norm $\| P \|$ is defined as follows 
	\begin{equation}
		\| P \| := \sup_{\substack{x_i \in \comp \\ |x_i| \leq 1}} \big| P(x_1, \cdots, x_r) \big|.
	\end{equation}
	We assume that the constants $B_{r, k}^{\beta}$ for $r, k \in \nat^*$, $\beta \geq 1$, are the minimal constants verifying the inequality (\ref{eq_boh_hil}).
	The main result of this section goes as follows.
	\begin{prop}\label{prop_subexp}
		For any fixed $r \in \nat^*$, the sequence $B_{r, k}^{1}$, $k \in \nat$, grows subexponentially in $k$.
	\end{prop}
	Recall that Bohnenblust–Hille in \cite{BohnHill} showed that for $\beta := \frac{2k}{k + 1}$, the constant 
	\begin{equation}
		B_k := \sup_{r \in \nat} B_{r, k}^{\beta}
	\end{equation}	
	is finite.
	In other words, for this choice of $\beta$, the bound like (\ref{eq_boh_hil}) can be made uniform in the number of variables.
	We need the following recent result about the asymptotics of $B_k$.
	\begin{thm}[{Bayart-Pellegrino-Seoane-Sep{\'u}lveda \cite[Corollary 5.3]{BohrRad} }]\label{thm_boh_hil_subx}
		The constants $B_k$ grow subexponentially in $k$. 
	\end{thm}
	\begin{proof}[Proof of Proposition \ref{prop_subexp}.]
		By the generalized mean inequality and (\ref{eq_boh_hil}), we have
		\begin{equation}
			\sum_{|\alpha| = k} |a_{\alpha}|
			\leq
			B_k
			\cdot
			\binom{r + k}{r}^{1 - \frac{k + 1}{2k}}
			\cdot
			\| P \|,
		\end{equation}
		in the notations (\ref{eq_poly}).
		In particular, since the binomial coefficients $\binom{r + k}{r}$ are polynomials in $k$ for fixed $r$ (and, hence, subexponential in $k$), we deduce Proposition \ref{prop_subexp} from Theorem \ref{thm_boh_hil_subx}.
	\end{proof}
	\begin{rem}
		We learned from Sébastien Boucksom that one can bypass the use of Theorem \ref{thm_boh_hil_subx} in the proof of Proposition \ref{prop_subexp} by a maximum principle, implying that for any polynomial $P$, we have $\sup_{x \in \mathbb{D}^r} | P(x_1, \cdots, x_r) |
		=
		\sup_{x \in (\partial \mathbb{D})^r} | P(x_1, \cdots, x_r) |$,
		and Parseval's theorem.
	\end{rem}
	\begin{sloppypar}
	\begin{proof}[Proof of Theorem \ref{thm_sym_equiv} in the special case when $V = \comp^r$ and $N_V := \| \cdot \|_V := \textit{l}_1$.]
		From (\ref{eq_compar_3_nmrs}), it is sufficient to show that $\sym_{\pi}(N_V)$, considered up to a subexponential constant, is bounded from above by $\sym_{{\rm{ev}}}(N_V)$.
		\par 
		Let us denote by $x_1, \ldots, x_r$ the coordinate vectors in $\comp^r$.
		We use the notation (\ref{eq_poly}) for $P \in \sym^k(V)$, $k \in \nat^*$.
		Since the dual of the $\textit{l}_1$-norm is given by the $\textit{l}_{\infty}$-norm on $\comp^r$, (\ref{eq_sup_norm_polll}) gives us 
		\begin{equation}\label{eq_nmrsp_1}
			\| P \|^{{\rm{ev}}}_{\textit{l}_1, k}
			=
			\| P \|.
		\end{equation}
		On another hand, since projective tensor norms behave multiplicatively on $\textit{l}_1$-spaces, i.e. $(\comp^n, \textit{l}_1) \otimes_{\pi} (\comp^n, \textit{l}_1) = (\comp^{nm}, \textit{l}_1)$, cf. \cite[Exercise 2.8]{RyanTensProd}, the norm $\sym_{\pi}(\textit{l}_1)$ corresponds to the sum of absolute values of the coefficients occurring in the representation (\ref{eq_poly}), i.e. we have 
		\begin{equation}\label{eq_nmrsp_2}
			\| P \|^{\sym, \pi}_{\textit{l}_1, k}
			=
			\sum_{|\alpha| = k} |a_{\alpha}|.
		\end{equation}
		We conclude by Proposition \ref{prop_subexp} and (\ref{eq_nmrsp_1}), (\ref{eq_nmrsp_2}) that $\sym_{\pi}(N_V)$, considered up to a subexponential constant, can be bounded from above by ${\rm{Ban}}^{\infty}(FS(N_V))$.
	\end{proof}
	\end{sloppypar}
		 
	\subsection{Projective tensor norms and holomorphic extension theorem}\label{sect_proj_t_ot}
	
	The main goal of this section is to prove Theorem \ref{thm_sym_equiv} in its full generality.
	Surprisingly, our main technical tool in the proof of this purely functional-analytic statement comes from complex geometry.
	We also use the following classical result.
	\begin{lem}[{cf. \cite[Lemma 2.2]{PolarConst}}]\label{lem_proj_fun_an}
		For any finite dimensional complex normed vector space $(V, \| \cdot \|_V)$, and any $\epsilon > 0$, there is $l \in \nat^*$ and a surjective map $\pi : \comp^l \to V$, such that $\| \cdot \|_V$ is related to the quotient norm associated with the $\textit{l}_1$-norm on $\comp^l$ as follows
		\begin{equation}\label{eq_proj_fun_an}
			\exp(-\epsilon) \cdot [\textit{l}_1] \leq \| \cdot \|_V \leq [\textit{l}_1].
		\end{equation}
	\end{lem}
	\begin{proof}[Proof of Theorem \ref{thm_sym_equiv}.]
		Since Theorem \ref{thm_sym_equiv} holds for $\textit{l}_1$-norms by the result from Section \ref{sect_bohn_hill}, we deduce by Lemma \ref{lem_proj_fun_an} that it is enough to show that the validity of Theorem \ref{thm_sym_equiv} is stable under taking quotients, i.e. if Theorem \ref{thm_sym_equiv} holds for a normed vector space $(U, N_U)$, then it holds for any normed quotient $(V, N_V)$, $\pi : U \to V$.
		As we shall see below, this is a consequence of the semiclassical version of Ohsawa-Takegoshi extension theorem.
		We consider the embedding 
		\begin{equation}
			{\rm{Im}}_{\pi} : \mathbb{P}(V^*) \to \mathbb{P}(U^*).
		\end{equation}
		Clearly, under this embedding, the associated restriction operator, which we denote by $\res_{\pi, k}$, and the projection map to the symmetric tensors induced by $\pi$, which we denote by $\sym^k \pi$, can be put with the identifications (\ref{eq_sym_isom}) into the following commutative diagram 
		\begin{equation}\label{eq_comm_diag}
		\begin{CD}
			H^0(\mathbb{P}(U^*), \mathscr{O}(k))  @> \res_{\pi, k} >> H^0(\mathbb{P}(V^*), \mathscr{O}(k)) 
			\\
			@| {}  @| {} 
			\\
			\sym^k(U)  @> \sym^k \pi >> \sym^k(V).
		\end{CD}
		\end{equation}
		Since $(V, N_V)$ is a quotient of $(U, N_U)$, we also have 
		\begin{equation}\label{eq_fs_bnd}
			FS(N_V) = FS(N_U)|_{\mathbb{P}(V^*)}.
		\end{equation}
		\par 
		From Theorem \ref{thm_ot_semi}, (\ref{eq_sum_ev_norm}), (\ref{eq_comm_diag}) and (\ref{eq_fs_bnd}), we conclude that for any $\epsilon > 0$, there is $k_0 \in \nat^*$, such that for any $k \geq k_0$, $f \in \sym^k(V)$, there is $g \in \sym^k(U)$, such that $\sym^k \pi(g) = f$, and
		\begin{equation}\label{eq_fin_1}
			\| f \|^{{\rm{ev}}}_{N_V, k}
			\geq
			\exp(- \epsilon k)
			\cdot
			\| g \|^{{\rm{ev}}}_{N_U, k}
		\end{equation}
		\par 
		Now, since Theorem \ref{thm_sym_equiv} holds for $(U, N_U)$, we deduce that there is $k_1 \in \nat^*$, such that for any $k \geq k_1$, $g \in \sym^k(U)$, we have
		\begin{equation}\label{eq_fin_2}
			\| g \|^{{\rm{ev}}}_{N_U, k}
			\geq
			\exp(- \epsilon k)
			\cdot
			\| g \|^{\sym, \pi}_{N_U, k}.
		\end{equation}
		Since $(V, N_V)$ is a quotient of $(U, N_U)$, for any $x \in U$, we have 
		\begin{equation}
			\| x \|_U \geq \| \pi(x) \|_V.
		\end{equation} 
		From this, the definition of the projective tensor norm and Lemma \ref{lem_sym_quot_inj}, we deduce that for any $k \in \nat^*$, $f \in \sym^k(V)$ and $g \in \sym^k(U)$, verifying $\sym^k \pi(g) = f$, we have
		\begin{equation}\label{eq_fin_3}
			\| g \|^{\sym, \pi}_{N_U, k}
			\geq
			\| f \|^{\sym, \pi}_{N_V, k}.
		\end{equation}
		\par 
		From (\ref{eq_fin_1}), (\ref{eq_fin_2}) and (\ref{eq_fin_3}), we see that for any $k \geq \max \{ k_0, k_1 \}$, $f \in \sym^k(V)$, we have
		\begin{equation}\label{eq_fin_4}
			\| f \|^{{\rm{ev}}}_{N_V, k}
			\geq
			\exp(- 2 \epsilon k)
			\cdot
			\| f \|^{\sym, \pi}_{N_V, k}.
		\end{equation}
		As $\epsilon > 0$ is arbitrary, from (\ref{eq_compar_3_nmrs}) and (\ref{eq_fin_4}), we conclude that $\sym_{{\rm{ev}}}(N_V)$ and $\sym_{\pi}(N_V)$ are asymptotically equivalent.
		As described after (\ref{eq_compar_3_nmrs}), this finishes the proof.
	\end{proof}
	
\section{Limiting behavior of jumping measures and geodesic rays}\label{sect_gr_jm}
	The main goal of this section is to study the limiting behavior of jumping measures of submultiplicative filtrations and to establish equivalence of several definitions of geodesic rays.
	More precisely, in Section \ref{sect_pf_filt}, we state the main result of this section making a connection between the asymptotic properties of filtration and the associated geodesic ray. 
	Then in Sections \ref{sec_spec_meas}, \ref{sect_filt}, which are more or less independent from the rest of the article, we give an alternative description of the geodesic ray featuring in Section \ref{sect_pf_filt}.

	\subsection{From submultiplicative filtrations to rays of submultiplicative norms}\label{sect_pf_filt}
	The main goal of this section is to provide an application of Theorem \ref{thm_char2} to the asymptotic study of submultiplicative filtrations.
	For this, on an arbitrary section ring, we associate with any submultiplicative filtration $\mathcal{F}$ a ray of submultiplicative norms.
	When this ray emanates from ${\rm{Ban}}^{\infty}(h^L_0)$ for a certain regularizable from above psh metric $h^L_0$ on $L$ and bounded $\mathcal{F}$, we obtain a ray of metrics on the line bundle through the Fubini-Study construction.
	We prove that this ray of metrics is geodesic and compare it with the Bergman geodesic ray of Phong-Sturm \cite{PhongSturmTestGeodK} and Ross-Witt Nystr{\"o}m \cite{RossNystAnalTConf}.
	\par 
	Now, let us first associate with any given submultiplicative filtration a ray of submultiplicative graded norms.
	More precisely, we fix a submultiplicative graded norm $N = \sum N_k$, $N_k := \| \cdot \|_k$, over the section ring $R(X, L)$.
	For a fixed submultiplicative filtration $\mathcal{F}$ on $R(X, L)$, we define by the procedure (\ref{eq_ray_norm_defn0}) for any $t \in [0, +\infty[$, $k \in \nat^*$, the ray of norms $N^t_{\mathcal{F},  k} := \| \cdot \|^t_{\mathcal{F}, k}$, over $H^0(X, L^{\otimes k})$, emanating from $N_k$.
	An easy verification shows that by submultiplicativity of $N$ and $\mathcal{F}$, for any $t \in [0, +\infty[$, the graded norm $N^t_{\mathcal{F}} = \sum N^t_{\mathcal{F},  k}$ is submultiplicative.
	\par
	We will now specify this to $N = {\rm{Ban}}^{\infty}(h^L_0)$ for a certain regularizable from above psh metric $h^L_0$ on $L$ and a bounded submultiplicative filtration $\mathcal{F}$.
	From the boundness of $\mathcal{F}$ (and of $h^L_0$), we deduce that the norm $N^t_{\mathcal{F}}$ is bounded for any $t \in [0, +\infty[$.
	Hence, by Theorem \ref{thm_char2} and Corollary \ref{cor_psh_appr_below}, we conclude that for any $p \in [1, +\infty[$, $t \in [0, + \infty[$, the following equivalence holds
	\begin{equation}\label{eq_n1_equiv}
		N^t_{\mathcal{F}}
		\sim_p
		{\rm{Ban}}^{\infty}(FS(N^t_{\mathcal{F}})_*).
	\end{equation}
	Recall that in (\ref{eq_ma_geod}), we defined the notion of geodesic ray.
	\begin{thm}\label{thm_ray_geod_bergm}
		For any regularizable from above psh metric $h^L_0$ on $L$ and any bounded submultiplicative filtration $\mathcal{F}$ on $R(X, L)$, $FS(N^t_{\mathcal{F}})_*$, $t \in [0, +\infty[$ is a geodesic ray emanating from $h^L_0$.
	\end{thm}
	\begin{proof}
		From Lemma \ref{lem_log_spec_ray} and Minkowski inequality, for any $t \geq s \geq 0$, $p \in [1, +\infty[$, we have
		\begin{equation}\label{eq_ray_nm_11}
			d_p(N_{\mathcal{F}}^t, N_{\mathcal{F}}^s)
			=
			(t - s) 
			\cdot
			d_p(N_{\mathcal{F}}^1, N_{\mathcal{F}}^0)
		\end{equation}
		Remark, however, that by Corollary \ref{cor_d_p_norm_fs_rel}, we have
		\begin{equation}\label{eq_ray_nm_101}
			d_p(N_{\mathcal{F}}^t, N_{\mathcal{F}}^s)
			=
			d_p \big( FS(N^t_{\mathcal{F}})_*, FS(N^s_{\mathcal{F}})_* \big).
		\end{equation}
		By (\ref{eq_ray_nm_11}) and (\ref{eq_ray_nm_101}), we conclude that the curve $FS(N^t_{\mathcal{F}})_*$, $t \in [0, +\infty[$, is a metric geodesic in any of the metric spaces $(\mathcal{E}_{\omega}^p, d_p)$, $p \in [1, +\infty[$.
		By Theorem \ref{thm_dar_lu_uniq_geod}, we yield that $FS(N^t_{\mathcal{F}})_*$, $t \in [0, +\infty[$, is a geodesic ray.
		As we assumed that $h^L_0$ is regularizable from above, $FS(N^t_{\mathcal{F}})_*$ emanates from $h^L_0$ by Theorem \ref{thm_bouck_erikss}.
	\end{proof}
	We call the above ray the \textit{Fubini-Study geodesic ray} and denote it by $h^{L, FS}_{\mathcal{F}, t} := FS(N^t_{\mathcal{F}})_*$, $t \in [0, +\infty[$.
	Let us now compare the Fubini-Study geodesic ray with the Bergman geodesic ray, defined by Phong-Sturm \cite{PhongSturmTestGeodK} and Ross-Witt Nystr{\"o}m \cite{RossNystAnalTConf}.
	To recall the definition of the latter one, we fix a continuous psh metric $h^L_0$ on $L$, a volume form $\mu$ of unit volume on $X$, and consider the $L^2$-norm ${\rm{Hilb}}(h^L, \mu)$ on $R(X, L)$, defined as in (\ref{eq_hilb_defn}).
	Now, for a fixed bounded submultiplicative filtration $\mathcal{F}$ on $R(X, L)$, and any $k \in \nat^*$, we define the ray of Hermitian norms $H^t_{\mathcal{F}, k} := \| \cdot \|^{t, H}_{\mathcal{F}, k}$ on $H^0(X, L^{\otimes k})$, emanating from ${\rm{Hilb}}(h^L, \mu)$, as in (\ref{eq_bas_st}).
	Let $H^t_{\mathcal{F}} = \sum H^t_{\mathcal{F}, k}$ be the induced graded norm on $R(X, L)$.
	We define the \textit{Bergman geodesic ray} $h^{L, \mathcal{B}}_{\mathcal{F}, t}$, $t \in [0, +\infty[$ as
	\begin{equation}\label{defn_ray_geod_bergm}
		h^{L, \mathcal{B}}_{\mathcal{F}, t} := \big( \lim_{k \to \infty} \inf_{l \geq k} FS(H^t_{\mathcal{F}, l})^{\frac{1}{l}} \big)_*.
	\end{equation}
	From Lemma \ref{lem_fs_inf_d}, we see that this definition is equivalent to the one of Ross-Witt Nystr{\"o}m \cite[Definition 9.1]{RossNystAnalTConf} up to a change of variables $t \mapsto 2 t$. 
	The following result ties the two constructions.
	\begin{prop}\label{prop_two_norms_comp}
		For any $t \in [0, +\infty[$, $H^t_{\mathcal{F}} \sim N^t_{\mathcal{F}}$.
		In particular, Fubini-Study geodesic ray emanating from a fixed continuous psh metric coincides with the respective Bergman geodesic ray.
	\end{prop}
	\begin{proof}
		It follows directly from Lemma \ref{lem_two_norms_comp0} and Proposition \ref{prop_ban_hilb_eq}.
	\end{proof}
	\begin{rem}
		Phong-Sturm in \cite[Theorem 1]{PhongSturmTestGeodK} established that Bergman geodesic ray emanating from a smooth positive initial point is indeed a geodesic ray for any filtration arising from a test configuration, see Section \ref{sect_filt} for a recap of the relation between the two.
		An alternative proof was given by Ross-Witt Nystr{\"o}m in \cite[\S 9]{RossNystAnalTConf}.
		From Proposition \ref{prop_two_norms_comp}, Theorem \ref{thm_ray_geod_bergm} gives a new proof of this result for more general initial points of the ray.
	\end{rem}
	\par 
	Let us now finally state the main result of this section.
	We define the sequence of \textit{jumping measures} $\mu_{\mathcal{F}, k}$, $k \in \nat^*$, on $\real$ of $\mathcal{F}$ as follows
	\begin{equation}\label{eq_jump_meas_d}
		\mu_{\mathcal{F}, k} := \frac{1}{\dim H^0(X, L^{\otimes k})} \sum_{j = 1}^{\dim H^0(X, L^{\otimes k})} \delta_{k^{-1} e_{\mathcal{F}}(j, k)}, 
	\end{equation}
	where $\delta_x$ is the Dirac mass at $x \in \real$ and $e_{\mathcal{F}}(j, k)$ are the \textit{jumping numbers}, defined as follows
	\begin{equation}\label{eq_defn_jump_numb}
		e_{\mathcal{F}}(j, k) := \sup \Big\{ t \in \real : \dim \mathcal{F}^t H^0(X, L^{\otimes k}) \geq j \Big\}.
	\end{equation}
	\par 
	\begin{sloppypar}
	Now, for any geodesic ray $h^L_t$, $t \in [0, +\infty[$, emanating from a smooth positive metric, one can define its \textit{spectral measure} by $(-\dot{h}^L_t)_* (c_1(L, h^L_0)^n / \int_X c_1(L)^n)$, where the derivative $\dot{h}^L_t := (h^L_t)^{-1} \frac{\partial h^L_t}{\partial t}|_{t = 0}$ can be defined by convexity, see (\ref{eq_spec_meas_defn_seg}), despite the possible absence of regularity.
	For this measure, in particular, the $p$-absolute moments are related with the slopes of the $p$-Finsler distances between points on the geodesic ray, see (\ref{eq_d_p_berndss}).
	For the Fubini-Study geodesic ray constructed from a filtration $\mathcal{F}$, we denote the spectral measure by $\mu_{\mathcal{F}}$.
	As it follows from the next theorem (and implicit in the notation), $\mu_{\mathcal{F}}$ is independent from the initial point $h^L_0$ of the ray.
	\end{sloppypar}
	\begin{thm}\label{thm_filt}
		For any bounded submultiplicative filtration $\mathcal{F}$ on a section ring $R(X, L)$ of an ample line bundle $L$, the jumping measures $\mu_{\mathcal{F}, k}$, $k \in \nat^*$, converge weakly, as $k \to \infty$, to $\mu_{\mathcal{F}}$.
	\end{thm}
	\begin{rem}
		a) The existence of the weak limit was proved by Boucksom-Chen \cite{BouckChen} in a more general setting of big line bundles, refining an earlier work of Chen \cite{ChenHNolyg}.
		\par 
		b) When the filtration is induced by an ample test configuration, Theorem \ref{thm_filt} was established for product test configurations by Witt Nystr{\"o}m \cite[Theorems 1.1 and 1.4]{NystOkounTest} and for general test configurations by Hisamoto in \cite[Theorem 1.1]{HisamSpecMeas}, proving a conjecture \cite[after Theorem 1.4]{NystOkounTest}.
		\par 
		c) Our method differs from \cite{ChenHNolyg}, \cite{BouckChen}, \cite{NystOkounTest} and \cite{HisamSpecMeas}; it is not algebraic in nature and instead of Okounkov bodies, we rely on Fubini-Study geodesic rays and Theorem \ref{thm_char2}.
	\end{rem}
	\begin{proof}
		Let us first remark that it is enough to establish Theorem \ref{thm_filt} in the special case when the filtration $\mathcal{F}$ satisfies the additional assumption 
		\begin{equation}\label{eq_filt_bnd_zero}
			\mathcal{F}^0 R(X, L) = \{0\}.
		\end{equation}
		To see this, remark that since $\mathcal{F}$ is bounded, there is $C > 0$, verifying $\mathcal{F}^{C k} H^0(X, L^{\otimes k}) = \{0\}$.
		Consider now another filtration $\mathcal{F}_0$ on $R(X, L)$, defined for any $k \in \nat$, $\lambda \in \real$, as follows $\mathcal{F}_0^{\lambda} H^0(X, L^{\otimes k}) = \mathcal{F}^{\lambda + Ck} H^0(X, L^{\otimes k})$.
		Clearly, $\mathcal{F}_0$ is submultiplicative and bounded whenever $\mathcal{F}$ is submultiplicative and bounded.
		An easy verification shows that establishing Theorem \ref{thm_filt} for $\mathcal{F}_0$ and $\mathcal{F}$ is equivalent. We, hence, assume from now on that $\mathcal{F}$ satisfies (\ref{eq_filt_bnd_zero}).
		\par 
		By Lemma \ref{lem_log_spec_ray}, (\ref{eq_filt_bnd_zero}) and Minkowski inequality, the jumping measures $\mu_{\mathcal{F}, k}$ of $\mathcal{F}$ are related to the ray of submultiplicative norms $N_{\mathcal{F}, k}^t$, $t \in [0, +\infty[$, constructed in (\ref{eq_ray_norm_defn0}), by
		\begin{equation}\label{eq_dp_jump_norm2}
			\Big| 
			d_p(N_{\mathcal{F}, k}^1, N_{\mathcal{F}, k}^0)
			-
			k
			\cdot
			\sqrt[p]{
			\int (-x)^p  \mu_{\mathcal{F}, k}(x)}
			\Big|
			\leq
			\log \dim H^0(X, L^{\otimes k}).
		\end{equation}
		Since $N_{\mathcal{F}}^t \geq N_{\mathcal{F}}^0$, we have $h^{L, FS}_{\mathcal{F}, 1} \geq h^{L, FS}_{\mathcal{F}, 0}$ and Theorem \ref{thm_filt} follows from (\ref{eq_char_spec_meas}), (\ref{eq_ray_nm_101}) and (\ref{eq_dp_jump_norm2}).
	\end{proof}

	\subsection{Maximal geodesic rays from submultiplicative filtrations}\label{sec_spec_meas}
	The main goal of this section is to give an alternative description of the geodesic ray from (\ref{defn_ray_geod_bergm}).
	This and the next section are essentially independent from the rest of the article.
	\par 
	In order to state our result, let us recall the definition of maximal geodesic rays.
	This definition requires fixing initial point of the geodesic ray and its singularities at $+\infty$ (in the form of Lelong numbers).
	Maximal geodesic ray is then the supremum over all geodesic rays, verifying these “boundary conditions".
	The precise description of this requires some basic notions from non-Archimedean geometry.
	\par 
	Denote by $X^{an}$ the Berkovich analytification of the projective manifold $X$ with respect to the trivial absolute value on the ground field $\comp$. 
	We view $X^{an}$ as a topological space, whose points can be understood as semivaluations on $X$, i.e. valuations $v : \comp(Y)^* \to \real$ on the function field $\comp(Y)$ of subvarieties $Y$ of $X$, trivial on $\comp$. 
	In particular, $X^{an}$ contains the set $X^{div}$ of divisorial valuations on $\comp(X)$, i.e. valuations $v :  \comp(X)^* \to \real$ of the form $v = c \cdot {\rm{ord}}_E$, where $c \in \mathbb{Q}_{> 0}$, $E$ is a prime divisor on some normal variety $Y$ mapping birationally to $X$ and ${\rm{ord}}_E$ corresponds to the valuation calculating the order of vanishing along $E$.
	Remark, in particular, that ${\rm{ord}}_{\Sigma}$ is well-defined for any submanifold $\Sigma \subset X$ through the divisorial valuation of the exceptional divisor in the blow-up of $X$ along $\Sigma$.
	The space $X^{an}$ can be seen as a compactification of $X^{div}$, endowed with the topology of pointwise convergence, see \cite[\S 6.1]{BerBouckJonYTD}.
	\par 
	Now, the projection $\pi: X \times \mathbb{D} \to X$ induces a map $(X \times \mathbb{D})^{div} \to X^{div}$; this has a canonical section $\sigma : X^{div} \to (X \times \mathbb{D})^{div}$, the Gauss extension, defined by 
	\begin{equation}\label{eq_sigm_Gauss}
		\sigma(v)\Big( \sum f_i \tau^i \Big)
		=
		\min_i \{ v(f_i) + i \}.
	\end{equation}
	The rationale behind this is that we have $\sigma({\rm{ord}}_{\Sigma}) = {\rm{ord}}_{\Sigma \times \{0\}}$ for any submanifold $\Sigma \subset X$.
	\par 
	Recall that a \textit{psh ray} is a map $U : ]0, + \infty[ \to {\rm{PSH}}(X, \omega)$, such that in the notations (\ref{eq_defn_hat_u}), (\ref{eq_subgeod_req}), the function $\hat{U}$ is $\pi^* \omega$-psh on $X \times \mathbb{D}^*$.
	Of particular importance are psh rays $U : ]0, +\infty[ \to \mathcal{E}_{\omega}^1$ called \textit{geodesic rays}; the restriction of such $U$ to each $[a, b] \in ]0, +\infty[$ coincides (up to affine reparametrization) with the distinguished psh geodesic joining $U_a$ to $U_b$, see (\ref{eq_geod_as_env}).
	We see from (\ref{eq_ma_geod}) that for bounded psh rays, this definition coincides with (\ref{eq_ma_geod}) modulo the identification (\ref{eq_pot_to_metr}).
	Given now a psh ray $U : ]0, + \infty[ \to {\rm{PSH}}(X, \omega)$ of \textit{linear growth} (i.e. such that there is $a > 0$, for which $U(t) - at$ is bounded from above, as $t \to + \infty$), define the $S^1$-invariant $\pi^* \omega$-psh function $V$ on $X \times \mathbb{D}^*$, in the notations of (\ref{eq_defn_hat_u}) by
	\begin{equation}
		V (x, \tau) := \hat{U} + a \log |\tau|.
	\end{equation}
	Then $V$ is bounded above near $X \times {0}$, hence, it uniquely extends to a $\pi^* \omega$-psh function on $X \times \mathbb{D}$, cf. \cite[Theorem I.5.23]{DemCompl}. 
 	For each divisorial valuation $w$ on $X \times \mathbb{D}$, we then can make sense of $w(V) \geq 0$ as a generic Lelong number on a suitable blowup, see \cite[\S B.6]{BerBouckJonYTD}.
 	We set $w(U) := w(V) - aw(\tau)$. 
 	This is independent of the choice of the constant $a$ by the additivity of Lelong numbers.
 	We construct the function $U_{NA} : X^{div} \to \real$, decoding the singularities of $U$ at $+ \infty$, by
 	\begin{equation}
 		U_{NA}(v) = - \sigma(v)(U).
 	\end{equation}
 	\par 
	Following \cite[Definition 6.5]{BerBouckJonYTD}, we say that a psh geodesic ray $U : [0, + \infty[ \to \mathcal{E}_{\omega}^1$ is \textit{maximal} if for any psh ray $V : ]0, + \infty[ \to \mathcal{E}_{\omega}^1$ of linear growth with $\lim_{t \to 0} V_t \leq U_0$ and $V_{NA} \leq U_{NA}$, we have $V \leq U$.
 	A \textit{maximal geodesic ray} is thus uniquely determined by $U_0$ and $U_{NA}$.
 	Remark the analogy between this and (\ref{eq_geod_as_env}).
 	However, not every geodesic ray is maximal, see \cite[Example 6.10]{BerBouckJonYTD}.
	\par 
	Now, as in (\ref{eq_filtr_norm}), we denote by  $\chi_{\mathcal{F}, k} : H^0(X, L^{\otimes k}) \to [0, +\infty[$, $k \in \nat^*$ the non-Archimedean norm on $H^0(X, L^{\otimes k})$ associated with the restriction to $H^0(X, L^{\otimes k})$ of a graded submultiplicative filtration $\mathcal{F}$ on $R(X, L)$.
	The associated graded norm $\chi_{\mathcal{F}} = \max \chi_{\mathcal{F}, k}$ on $R(X, L)$ is submultiplicative, i.e. for any $f \in H^0(X, L^{\otimes k})$, $g \in H^0(X, L^{\otimes l})$, $k, l \in \nat^*$, we have
	\begin{equation}\label{eq_na_submult_n}
		\chi_{\mathcal{F}, k + l}(f \cdot g) \leq \chi_{\mathcal{F}, k}(f) \cdot \chi_{\mathcal{F}, l}(g).
	\end{equation}
	\par 
	A graded (non-Archimedean) norm $\chi = \max \chi_k$ on $R(X, L)$ is called \textit{bounded} if there is $C > 0$, such that for any $k \in \nat^*$,  the following inequality is satisfied $\chi_k \geq \exp(- C k)$.
	Remark that since $R(X, L)$ is finitely generated, the existence of $C > 0$ such that for any $k \in \nat^*$, $\chi_k \leq \exp(C k)$ is automatic for submultiplicative norms by (\ref{eq_na_submult_n}).
	Clearly, if $\mathcal{F}$ is a bounded filtration on $R(X, L)$, $\chi_{\mathcal{F}}$ is bounded, and bounded submultiplicative filtrations on $R(X, L)$ are in one-to-one correspondence with bounded submultiplicative non-Archimedean norms on $R(X, L)$.
	\par 
	In this perspective, the construction of Boucksom-Jonsson \cite{BouckJohn21} of a non-Archimedean potential on $X^{an}$ from $\mathcal{F}$ realizes in the non-Archimedean context the complex-geometric philosophy we recalled in Section \ref{sec_fs_bdem} and Lemma \ref{lem_fs_dini} that associates to any bounded submultiplicative norm $N$ on $R(X, L)$ the Fubini-Study metric $FS(N)$ on $L$ through the associated Kodaira embedding.
	\par 
	To describe it precisely, recall that a semivaluation $v \in X^{an}$ can be naturally evaluated on a section $s \in H^0(X, M)$ of any line bundle $M$ on $X$, by defining $v(s) \in [0, +\infty]$ as the value of $v$ on the local function corresponding to $s$ in any local trivialization of $M$ at the center of the valuation, see \cite[p. 15]{BouckJohn21} for details. 
	For any $s \in H^0(X, M)$, we then can define $|s| : X^{an} \to [0, 1]$ by setting
	\begin{equation}\label{eq_nm_s_defn}
		|s|(v) := \exp(-v(s)).
	\end{equation}
	\par 
	Now, Boucksom-Jonsson in \cite[(4.3)]{BouckJohn21} associated with any non-Archimedean norm $\chi_k$ on $H^0(X, L^{\otimes k})$ the Fubini-Study potential, $FS(\chi_k) : X^{an} \to \real$, defined as follows
	\begin{equation}\label{eq_fs_fin_ord}
		FS(\chi_k) := \sup_{s \in H^0(X, L^{\otimes k}) \setminus \{0\} } \big\{ \log |s| - \log \chi_{k}(s) \big\}.
	\end{equation}
	\par 
	Whenever $\chi_k$ is associated with a graded filtration $\mathcal{F}$ on $R(X, L)$, we denote the associated Fubini-Study potentials by $FS(\mathcal{F})_{k}$.
	For bounded $\mathcal{F}$, the resulting sequence of potentials $FS(\mathcal{F})_{k}$, $k \in \nat^*$, is uniformly bounded from above.
	If the filtration $\mathcal{F}$ is, moreover, submultiplicative, then the sequence of potentials $FS(\mathcal{F})_{k}$ is superadditive and the Fubini-Study potential of a bounded submultiplicative filtration is now defined by Fekete's lemma as 
	\begin{equation}\label{eq_na_fs_defn}
		FS(\mathcal{F}) = \sup_{k \in \nat^*} \Big\{ \frac{1}{k} FS(\mathcal{F})_{k} \Big\}.
	\end{equation}
	\par 
	\begin{sloppypar}
	We will now recall some basic results from non-Archimedean pluripotential theory.
	Following \cite[\S 1.8]{BouckJohn21}, we say that a function $f$ on $X^{an}$ is a \textit{Fubini-Study function} if there is $m \in \nat^*$, base point free $s_1, \ldots, s_r \in H^0(X, L^m)$, and some $\lambda_1, \ldots, \lambda_r \in \real$, such that 
	\begin{equation}
		f = \frac{1}{m} \max_{j = 1, \ldots, r} \{ 
			\log |s_j| + \lambda_j
		\},
	\end{equation}
	where $|s_j|$ were defined in (\ref{eq_nm_s_defn}).
	Following \cite[\S 6.2, 6.3]{BerBouckJonYTD}, we say that a function $\phi : X^{an} \to [-\infty, +\infty[$ is in $PSH^{NA}(X)$ if it can be obtained as the pointwise limit of a decreasing net of Fubini-Study functions, excluding $\phi = -\infty$.
	See the analogy with Remark \ref{rem_regul_above}b).
	\par 
	Analogously to the complex situation, see  (\ref{eq_energy}), using the non-Archimedean mixed Monge-Ampère operator, one can define the energy functional $E$ on the space $CPSH^{NA}(X) := PSH^{NA}(X) \cap \ccal^0(X^{an})$, see \cite[(4.3)]{BouckJohn21}.
	This energy functional satisfies similar monotonicity properties as its complex analogue.
	Using this, it is then possible to extend $E$ to $PSH^{NA}(X)$ through the same procedure as in (\ref{eq_energy_ext}).
	We denote by $\mathcal{E}^{1, NA}$ the subset of $PSH^{NA}(X)$ with finite (i.e. not equal to $- \infty$) energy.	
	\end{sloppypar}
	\begin{thm}[{\cite[Theorems 6.2, 6.4 and 6.6]{BerBouckJonYTD} and \cite[Lemma 4.3]{BouckJohn21}}]\label{thm_lin_grwth_e1}
 		For any psh ray of linear growth $U : ]0, + \infty[ \to \mathcal{E}_{\omega}^1$, the function $U_{NA} : X^{div} \to \real$ extends uniquely to $U_{NA} \in PSH^{NA}(X)$, and we, moreover, have $U_{NA} \in \mathcal{E}^{1, NA}$. 
 		Similarly, for any bounded submultiplicative filtration $\mathcal{F}$ on $R(X, L)$, we have $FS(\mathcal{F}) \in \mathcal{E}^{1, NA}$. 
 		For any $u \in \mathcal{E}_{\omega}^1$ and $\phi \in \mathcal{E}^{1, NA}$, there exists a unique maximal geodesic ray $U : [0, + \infty[ \to \mathcal{E}_{\omega}^1$ emanating from $u$ such that $U_{NA} = \phi$.
 	\end{thm}
	From Theorems \ref{thm_lin_grwth_e1}, we see, in particular, that for any bounded psh metric $h^L_0$ and any bounded submultiplicative filtration $\mathcal{F}$, there is the \textit{maximal geodesic ray} $h^{L, \max}_{\mathcal{F}, t}$, emanating from $h^L_0$, corresponding on the potential level, see (\ref{eq_pot_to_metr}), to the ray $U$, verifying $U_{NA} = FS(\mathcal{F})$.
	\par 
	We can now state the main result of this section.
	\begin{thm}\label{thm_ray_coinc}
		For any bounded submultiplicative filtration $\mathcal{F}$ on a section ring $R(X, L)$ of an ample line bundle $L$, the Fubini-Study geodesic ray, $h^{L, \mathcal{B}}_{\mathcal{F}, t}$, $t \geq 0$, emanating from a fixed regularizable from above psh metric on $L$ coincides with the respective maximal geodesic ray, $h^{L, \max}_{\mathcal{F}, t}$.
	\end{thm}
	\begin{rem}\label{rem_ray_coinc}
		a)
		The fact that $h^{L, \mathcal{B}}_{\mathcal{F}, t}$ is maximal was previously established by Darvas-Xia \cite{DarvXiaTest} by relying on Ross-Witt Nystr{\"o}m \cite{RossNystAnalTConf}.
		\par 
		b) 
		As we shall explain below, for finitely generated filtrations, Theorem \ref{thm_ray_coinc} was established by Phong-Sturm \cite{PhongSturmDirMA}, Berman-Boucksom-Jonsson \cite{BerBouckJonYTD} and Boucksom-Jonsson \cite{BouckJohn21}.
	\end{rem}
	\par 
	To prove Theorem \ref{thm_ray_coinc}, we rely on the works of Phong-Sturm \cite{PhongSturmDirMA}, Berman-Boucksom-Jonsson \cite{BerBouckJonYTD} and Boucksom-Jonsson \cite{BouckJohn21}, which establish Theorem \ref{thm_ray_coinc} for $\mathcal{F}$ induced by ample test configurations, and on the fact, remarked by Sz{\'e}kelyhidi \cite{SzekeTestConf}, that any filtration can be approximated by filtrations induced by ample test configurations.
	It remains to establish that both Fubini-Study geodesic rays and maximal geodesic rays behave reasonably under these approximations.
	For Fubini-Study geodesic rays, the corresponding statement is Theorem \ref{thm_can_trunc_bergm}, and it follows from Theorem \ref{thm_char2}.
	For maximal geodesic rays, the corresponding statement is Theorem \ref{thm_max_can_appr}, and it follows from several results from complex and non-Archimedean pluripotential theory developed by Berman-Boucksom-Jonsson \cite{BerBouckJonYTD} and Boucksom-Jonsson \cite{BouckJohn21}.
	We then compare Theorem \ref{thm_ray_coinc} with the maximality result following from the works of Ross-Witt Nystr{\"o}m \cite{RossNystom}, \cite{RossNystAnalTConf}.
	\par
	Define now, following Sz{\'e}kelyhidi \cite{SzekeTestConf}, for any graded filtration $\mathcal{F}$ on $R(X, L)$ the sequence of \textit{canonical approximations} $\mathcal{F}_{(k)}$ of $\mathcal{F}$, as the filtrations induced by $\mathcal{F} H^0(X, L^{\otimes k})$ on $R(X, L^{\otimes k})$, for $k \in \nat^*$ big enough so that $H^0(X, L^{\otimes k})$ generates the algebra $R(X, L^{\otimes k})$.
	Taking into account the identification of submultiplicative filtrations and submultiplicative non-Archimedean norms, see (\ref{eq_filtr_norm}), remark that  $\mathcal{F}_{(k)}$ are the formal analogues of the norm $N^{\pi}$ from Theorem \ref{thm_induc}. 
	\par 
	The proof of Theorem \ref{thm_ray_coinc} decomposes into several statements.
	The first step is to establish that Fubini-Study geodesic rays behave reasonably under the above approximations.
	\begin{thm}\label{thm_can_trunc_bergm}
		For any bounded submultiplicative filtration $\mathcal{F}$ on a section ring $R(X, L)$ of an ample line bundle $L$ and any regularizable from above psh metric $h^L_0$, Fubini-Study geodesic rays emanating from $h^L_0$ behave continuously in the topology of $\mathcal{E}_{\omega}^1$ with respect to canonical approximations $\mathcal{F}_{(k)}$ of $\mathcal{F}$.
		In other words, in $\mathcal{E}_{\omega}^1$, we have
		\begin{equation}
			\lim_{k \to \infty} \big( h^{L^{\otimes k}, FS}_{\mathcal{F}_{(k)}, t} \big)^{\frac{1}{k}} = h^{L, FS}_{\mathcal{F}, t},
		\end{equation}
		where the limit is taken over multiplicative sequence $k \in \nat^*$, as for example $k = 2^l$, $l \in \nat$.
	\end{thm}
	The proof of Theorem \ref{thm_can_trunc_bergm} is based on Theorem \ref{thm_char2} and the study of the volume functional for canonical approximations.
	Recall that the volume of a bounded submultiplicative filtration $\mathcal{F}$ is defined as follows
	\begin{equation}\label{eq_vol_filtr_defn}
		{\rm{vol}}(\mathcal{F}) := \lim_{k \to \infty} \int x \mu_{\mathcal{F}, k},
	\end{equation}
	where $\mu_{\mathcal{F}, k}$, $k \in \nat$ are jumping measures (\ref{eq_jump_meas_d}) of the filtration.
	The existence of the aforementioned limit is a consequence of \cite{ChenHNolyg}, \cite{BouckChen} or of Theorem \ref{thm_filt}.
	\begin{thm}[{Boucksom-Jonsson \cite[Theorem 3.18 and (3.14)]{BouckJohn21}}]\label{thm_vol_cont_dep}
		Volumes depend continuously under canonical approximations, i.e. for any bounded submultiplicative filtration $\mathcal{F}$, we have
		\begin{equation}\label{eq_vol_filtr_defn101}
			\lim_{k \to \infty} {\rm{vol}}(\mathcal{F}_{(k)}) = {\rm{vol}}(\mathcal{F}),
		\end{equation}
		where the limit is taken over multiplicative sequence $k \in \nat^*$, as for example $k = 2^l$, $l \in \nat$.
	\end{thm}
	Let us give an alternative proof of Theorem \ref{thm_vol_cont_dep}, relying solely on Theorem \ref{thm_char2}.
	As we shall see, Theorem \ref{thm_can_trunc_bergm} would follow rather easily from this new proof of Theorem \ref{thm_vol_cont_dep}.
	\begin{proof}
		First of all, from considerations, similar to the ones from the proof of Theorem \ref{thm_filt}, we can assume that the filtration $\mathcal{F}$ satisfies the additional assumption (\ref{eq_filt_bnd_zero}).
		\par 
		From Corollary \ref{cor_d_p_norm_fs_rel}, we see that for any $k \in \nat^*$, $t \in [0, +\infty[$, we have
		\begin{equation}\label{eq_cont_bergm_1}
			d_1 \big( h^{L, FS}_{\mathcal{F}, t}, (h^{L^{\otimes k}, FS}_{\mathcal{F}_{(k)}, t})^{\frac{1}{k}} \big)
			=
			\frac{1}{k}
			d_1 \big( N^t_{\mathcal{F}}|_{R(X, L^{\otimes k})}, N^t_{\mathcal{F}_{(k)}} \big).
		\end{equation}
		Remark, however, that since $\mathcal{F}$ satisfies (\ref{eq_filt_bnd_zero}), and since the weight of any element of $R(X, L^{\otimes k})$ with respect to $\mathcal{F}$ is at least as big as the weight with respect to $\mathcal{F}_{(k)}$, we have
		\begin{equation}\label{eq_ord_n_norms}
			N^t_{\mathcal{F}_{(k)}} \geq N^t_{\mathcal{F}}|_{R(X, L^{\otimes k})} \geq N|_{R(X, L^{\otimes k})}.
		\end{equation}
		In particular, by Lemma \ref{lem_d_1_ident_metr_gr}, we conclude that 
		\begin{equation}\label{eq_cont_bergm_2}
			d_1 \big( N^t_{\mathcal{F}}|_{R(X, L^{\otimes k})}, N^t_{\mathcal{F}_{(k)}} \big)
			=
			d_1 \big( N|_{R(X, L^{\otimes k})}, N^t_{\mathcal{F}_{(k)}} \big)
			-
			d_1 \big( N|_{R(X, L^{\otimes k})}, N^t_{\mathcal{F}}|_{R(X, L^{\otimes k})} \big).
		\end{equation}
		However, by Lemma \ref{lem_log_spec_ray} and the fact that $\mathcal{F}$ satisfies (\ref{eq_filt_bnd_zero}), we have
		\begin{equation}\label{eq_cont_bergm_3}
		\begin{aligned}
			&
			d_1 \big( N|_{R(X, L^{\otimes k})}, N^t_{\mathcal{F}_{(k)}} \big) = - t \cdot k \cdot {\rm{vol}}(\mathcal{F}_{(k)}),
			\\
			&
			d_1 \big( N|_{R(X, L^{\otimes k})}, N^t_{\mathcal{F}}|_{R(X, L^{\otimes k})} \big) = - t \cdot k \cdot {\rm{vol}}(\mathcal{F}).
		\end{aligned}
		\end{equation}
		\par 
		By Lemma \ref{lem_fs_dini}, the trivial fact that $N^t_{\mathcal{F}_{(k)}, 1} = N^t_{\mathcal{F}, k}$ and (\ref{eq_ord_n_norms}), we deduce
		\begin{equation}
			FS(N^t_{\mathcal{F}})^k \leq FS(N^t_{\mathcal{F}_{(k)}}) \leq FS(N^t_{\mathcal{F}, k}).
		\end{equation}
		From this, (\ref{eq_log_rel_spec2}) and (\ref{eq_cont_bergm_1}), we conclude that 
		\begin{equation}\label{eq_vol_cont_00}
			\frac{1}{k}
			d_1 \big( N^t_{\mathcal{F}}|_{R(X, L^{\otimes k})}, N^t_{\mathcal{F}_{(k)}} \big)
			\leq
			d_1 \big( FS(N^t_{\mathcal{F}})_*, FS(N^t_{\mathcal{F}, k})_*^{\frac{1}{k}} \big).
		\end{equation}
		From Lemma \ref{lem_fs_dini}, (\ref{eq_darvas_conv}), (\ref{eq_vol_cont_00}) and the fact that the sequence of metrics $FS(N^t_{\mathcal{F}, k})^{\frac{1}{k}}$ is decreasing over multiplicative sequence $k \in \nat^*$, we deduce the following convergence
		\begin{equation}\label{eq_vol_cont_11}
			\lim_{k \to \infty} \frac{1}{k}
			d_1 \big( N^t_{\mathcal{F}}|_{R(X, L^{\otimes k})}, N^t_{\mathcal{F}_{(k)}} \big)
			=
			0,
		\end{equation}		 
		where $k$ runs over a multiplicative sequence.
		We deduce (\ref{eq_vol_filtr_defn101}) from (\ref{eq_cont_bergm_2}), (\ref{eq_cont_bergm_3}) and (\ref{eq_vol_cont_11}).
	\end{proof}
	\begin{proof}[Proof of Theorem \ref{thm_can_trunc_bergm}.]
		It follows directly from Theorem \ref{thm_vol_cont_dep}, (\ref{eq_darvas_conv}), (\ref{eq_cont_bergm_1}), (\ref{eq_cont_bergm_2}) and (\ref{eq_cont_bergm_3}).
	\end{proof}
	The second step of the proof of Theorem \ref{thm_ray_coinc} consists in the following result.
	\begin{thm}\label{thm_max_can_appr}
		An analogue of Theorem \ref{thm_can_trunc_bergm} holds for maximal geodesic rays emanating from bounded psh metrics.
		In other words, under the assumptions of Theorem \ref{thm_can_trunc_bergm}, in $\mathcal{E}_{\omega}^1$, we have
		\begin{equation}
			\lim_{k \to \infty} \big( h^{L^{\otimes k}, \max}_{\mathcal{F}_{(k)}, t} \big)^{\frac{1}{k}} = h^{L, \max}_{\mathcal{F}, t},
		\end{equation}
		where the limit is taken over multiplicative sequence $k \in \nat^*$, as for example $k = 2^l$, $l \in \nat$.
	\end{thm}
	\begin{proof}
		First, remark that maximal geodesic rays are monotonic with respect to the data. 
 	In other words, for $u_0, u_1 \in \mathcal{E}_{\omega}^1$ and $\phi_0, \phi_1 \in \mathcal{E}^{1, NA}$, verifying $u_0 \leq u_1$ and $\phi_0 \leq \phi_1$, we have
	 	\begin{equation}\label{eq_max_psh_ray_mon}
	 		U^0_t \leq U^1_t,
	 	\end{equation}
	 	for any $t \in [0, + \infty[$, where $U^i : [0, + \infty[ \to \mathcal{E}_{\omega}^1$ is the maximal geodesic ray emanating from $u_i$ such that $U^i_{NA} = \phi_i$ for $i = 0, 1$.
	 	\par  
 		Boucksom-Jonsson in \cite[Theorem 5.4 and Lemma 6.17iii) and \S 3.6]{BouckJohn21} established that for any bounded submultiplicative filtration $\mathcal{F}$, for any $k \in \nat^*$, we have
		\begin{equation}\label{eq_thm_e_fs_filt}
			\frac{1}{k}
			FS(\mathcal{F}_{(k)}) \leq FS(\mathcal{F}),
			\qquad
			\lim_{k \to \infty} \frac{1}{k} E( FS(\mathcal{F}_{(k)})) = E(FS(\mathcal{F})).
		\end{equation}
		Clearly, by (\ref{eq_max_psh_ray_mon}) and monotonicity from (\ref{eq_thm_e_fs_filt}), we obtain that for any $t \in [0, +\infty[$, we have
		\begin{equation}
			\frac{1}{k} U^{\mathcal{F}_{(k)}, u}_{t} 
			\leq 
			U^{\mathcal{F}, u}_t.
		\end{equation}
		Since over multiplicative sequence $k \in \nat^*$, the sequence $\frac{1}{k} FS(\mathcal{F}_{(k)})$ increases to $FS(\mathcal{F})$, by (\ref{eq_darvas_conv}) and (\ref{eq_d1_ener}), it is enough to establish that 
		\begin{equation}\label{eq_max_psh_ray_mon101}
			\lim_{k \to \infty} \frac{1}{k} E(U^{\mathcal{F}_{(k)}, u}_{t} ) = E(U^{\mathcal{F}, u}_t).
		\end{equation}
		\par Recall, however, that Berman-Boucksom-Jonsson in \cite[Corollary 6.7]{BerBouckJonYTD} established that a psh geodesic ray $U : [0, + \infty[ \to \mathcal{E}_{\omega}^1$ is maximal if and only if we have 
		\begin{equation}\label{eq_max_psh_ray_mon1022}
			E(U_t) = E(U_0) + t E(U_{NA}),
		\end{equation}
		for any $t \in [0, +\infty[$.
		Hence, we see that (\ref{eq_max_psh_ray_mon101}) is a consequence of (\ref{eq_thm_e_fs_filt}) and (\ref{eq_max_psh_ray_mon1022}).
	\end{proof}
	\par 
	As we recall in Section \ref{sect_filt}, with any ample test configuration $\mathcal{T}$ of $(X, L)$, one can associate a bounded submultiplicative filtration $\mathcal{F}_{\mathcal{T}}$ on $R(X, L)$. 
	Through a combination of the results of Phong-Sturm \cite{PhongSturmRegul}, Berman-Boucksom-Jonsson \cite{BerBouckJonYTD} and Boucksom-Jonsson \cite{BouckJohn21}, we obtain in Section \ref{sect_filt} the following result.
	\begin{thm}\label{thm_ray_coinc_test}
		For any filtration $\mathcal{F}_{\mathcal{T}}$ arising from an ample test configuration $\mathcal{T}$ of $(X, L)$, the conclusion of Theorem \ref{thm_ray_coinc} holds for rays emanating from smooth positive metrics.
	\end{thm}
	We can now finally draw the main consequence of this section.
	\begin{proof}[Proof of Theorem \ref{thm_ray_coinc}]
		First of all, it is enough to establish the statement for filtrations with integer weights.
		Indeed, instead of $\mathcal{F}$, one can consider the round-down $\lfloor \mathcal{F} \rfloor$, defined in such a way that its weight function $w_{\lfloor \mathcal{F} \rfloor}$ (see (\ref{eq_filtr_norm}) for a definition) is related to $w_{\mathcal{F}}$ as follows
		\begin{equation}
			w_{\lfloor \mathcal{F} \rfloor} = \lfloor w_{\mathcal{F}} \rfloor.
		\end{equation}
		Remark that $\lfloor \mathcal{F} \rfloor$ is submultiplicative and bounded whenever $\mathcal{F}$ is.
		Directly from the definitions, we then obtain $h^{L, FS}_{\mathcal{F}, t} = h^{L, FS}_{\lfloor \mathcal{F} \rfloor, t}$.
		According to \cite[Example 1.7]{BouckJohn21}, we have $FS(\lfloor \mathcal{F} \rfloor) = FS(\mathcal{F})$, resulting in the identity $h^{L, \max}_{\mathcal{F}, t} = h^{L, \max}_{\lfloor \mathcal{F} \rfloor, t}$.
		This means that the statements of Theorem \ref{thm_ray_coinc} for $\lfloor \mathcal{F} \rfloor$ and $\mathcal{F}$ are equivalent.
		We, hence, assume that $\mathcal{F}$ has integer weights.
		\par 
		As we recall in Section \ref{sect_filt}, for a filtration $\mathcal{F}$ with integer weights, for any $k \in \nat^*$, the filtration $\mathcal{F}_{(k)}$ is associated with an ample test configuration.
		Now, Theorem \ref{thm_ray_coinc} for rays emanating from smooth positive metrics is a trivial consequence of Theorems \ref{thm_can_trunc_bergm}, \ref{thm_max_can_appr}, \ref{thm_ray_coinc_test} and (\ref{eq_darvas_conv}).
		\par 
		Let us now establish Theorem \ref{thm_ray_coinc} for rays emanating from regularizable from above psh metrics $h^L_0$.
		Consider a sequence of smooth positive metrics $h^L_i$, $i \in \nat^*$, decreasing almost everywhere to $h^L_0$.
		 Such a sequence exists by the definition of regularizable from above psh metrics and Remark \ref{rem_regul_above}b).
		 We already know that Fubini-Study geodesics rays associated to $h^L_i$, denoted here by $h^{L, FS}_{\mathcal{F}, t, i}$, are identical to the respective maximal geodesic rays, denoted here by $h^{L, \max}_{\mathcal{F}, t, i}$. 
		 Since both  $h^{L, FS}_{\mathcal{F}, t, i}$ and $h^{L, \max}_{\mathcal{F}, t, i}$ are decreasing in $i \in \nat$ by (\ref{eq_max_psh_ray_mon}) and obvious reasons, by (\ref{eq_darvas_conv}) it is enough to show that both $h^{L, FS}_{\mathcal{F}, t, i}$ and $h^{L, \max}_{\mathcal{F}, t, i}$ behave continuously in topology of $\mathcal{E}_{\omega}^1$, as $i \to \infty$.
		 For this, similarly to the proof of Theorem \ref{thm_filt}, we may assume that the filtration $\mathcal{F}$ satisfies the additional assumption (\ref{eq_filt_bnd_zero}).
		 Then from (\ref{eq_ord_n_norms}) and (\ref{eq_max_psh_ray_mon}), the following string of inequalities is satisfied
		\begin{equation}\label{eq_ineq_fs_approx_above}
		\begin{aligned}
			&
			h^{L, FS}_{\mathcal{F}, t, i} \geq h^{L, FS}_{\mathcal{F}, 0, i} \geq h^{L, FS}_{\mathcal{F}, 0},
			&&
			h^{L, \max}_{\mathcal{F}, t, i} \geq h^{L, \max}_{\mathcal{F}, 0, i} \geq h^{L, \max}_{\mathcal{F}, 0},
			\\
			&
			h^{L, FS}_{\mathcal{F}, t, i} \geq h^{L, FS}_{\mathcal{F}, t} \geq h^{L, FS}_{\mathcal{F}, 0},
			&&
			h^{L, \max}_{\mathcal{F}, t, i} \geq h^{L, \max}_{\mathcal{F}, t} \geq h^{L, \max}_{\mathcal{F}, 0}.
		\end{aligned}
		\end{equation}
		From (\ref{eq_d1_ener}) and (\ref{eq_ineq_fs_approx_above}), we conclude that 
		\begin{equation}\label{eq_equal_b_fs_0}
		\begin{aligned}
			&
			d_1(h^{L, FS}_{\mathcal{F}, t, i}, h^{L, FS}_{\mathcal{F}, t}) 
			= 
			d_1(h^{L, FS}_{\mathcal{F}, 0, i}, h^{L, FS}_{\mathcal{F}, 0})
			-
			d_1(h^{L, FS}_{\mathcal{F}, t}, h^{L, FS}_{\mathcal{F}, 0})
			+
			d_1(h^{L, FS}_{\mathcal{F}, t, i}, h^{L, FS}_{\mathcal{F}, 0, i}),
			\\
			&
			d_1(h^{L, \max}_{\mathcal{F}, t, i}, h^{L, \max}_{\mathcal{F}, t}) 
			= 
			d_1(h^{L, \max}_{\mathcal{F}, 0, i}, h^{L, \max}_{\mathcal{F}, 0})
			-
			d_1(h^{L, \max}_{\mathcal{F}, t}, h^{L, \max}_{\mathcal{F}, 0})
			+
			d_1(h^{L, \max}_{\mathcal{F}, t, i}, h^{L, \max}_{\mathcal{F}, 0, i}).
		\end{aligned}
		\end{equation}
		\par 
		Let us deal with maximal geodesic rays first. 
		From (\ref{eq_d1_ener}), (\ref{eq_max_psh_ray_mon1022}) and (\ref{eq_ineq_fs_approx_above}), we deduce
		\begin{equation}\label{eq_equal_b_fs_00}
			d_1(h^{L, \max}_{\mathcal{F}, t, i}, h^{L, \max}_{\mathcal{F}, 0, i})
			=
			d_1(h^{L, \max}_{\mathcal{F}, t}, h^{L, \max}_{\mathcal{F}, 0}).
		\end{equation}
		By Theorem \ref{thm_lin_grwth_e1}, we have $h^{L, \max}_{\mathcal{F}, 0, i} = h^L_i$ and $h^{L, \max}_{\mathcal{F}, 0} = h^L$. 
		Hence, by (\ref{eq_darvas_conv}), we deduce that $d_1(h^{L, \max}_{\mathcal{F}, 0, i}, h^{L, \max}_{\mathcal{F}, 0}) \to 0$.
		From this, (\ref{eq_equal_b_fs_0}) and (\ref{eq_equal_b_fs_00}), we yield that maximal geodesic rays behave continuously in $\mathcal{E}_{\omega}^1$, i.e.
		\begin{equation}\label{eq_equal_b_fs_2}
			\lim_{i \to \infty} d_1(h^{L, \max}_{\mathcal{F}, t, i}, h^{L, \max}_{\mathcal{F}, t})
			=
			0.
		\end{equation}
		\par Let us establish the corresponding statement for Fubini-Study geodesic rays.
		By Lemma \ref{lem_log_spec_ray} and (\ref{eq_ray_nm_101}), we conclude that
		\begin{equation}\label{eq_equal_b_fs_1}
			d_1(h^{L, FS}_{\mathcal{F}, t, i}, h^{L, FS}_{\mathcal{F}, 0, i})
			=
			d_1(h^{L, FS}_{\mathcal{F}, t}, h^{L, FS}_{\mathcal{F}, 0}).
		\end{equation}
		Since $h^L_0$ is regularizable from above, by Theorem \ref{thm_ray_geod_bergm}, we have $h^{L, FS}_{\mathcal{F}, 0, i} = h^L_i$ and $h^{L, FS}_{\mathcal{F}, 0} = h^L$, hence, by (\ref{eq_darvas_conv}), we deduce that $d_1(h^{L, FS}_{\mathcal{F}, 0, i}, h^{L, FS}_{\mathcal{F}, 0}) \to 0$.
		From this, (\ref{eq_equal_b_fs_0}), (\ref{eq_equal_b_fs_2}) and (\ref{eq_equal_b_fs_1}), we deduce that Fubini-Study geodesic rays behave continuously in topology of $\mathcal{E}_{\omega}^1$.
	\end{proof}

	\subsection{Filtrations, test configurations and geodesic rays}\label{sect_filt}
	The goal of this section is to recall the relation between test configurations and filtrations and to establish Theorem \ref{thm_ray_coinc_test}.
	Recall first that a test configuration $\mathcal{T} = (\mathcal{X}, \mathcal{L})$ for $(X, L)$ consists of
	\begin{enumerate}
		\item A scheme $\mathcal{X}$ with a $\comp^*$-action $\rho$,
		\item A $\comp^*$-equivariant line bundle $\mathcal{L}$ over $\mathcal{X}$,
		\item A flat $\comp^*$-equivariant projection $\pi : \mathcal{X} \to \comp$, where $\comp^*$ acts on $\comp$ by multiplication if we denote by $X_t := \pi^{-1}(t)$, then $\mathcal{L}|_{X_1}$ is isomorphic to $L^r$ for some $r > 0$.
	\end{enumerate}
	For simplicity, we assume from now on that $r = 1$ in the above definition.
	We say that the test configuration is (semi)ample if $\mathcal{L}$ is relatively (semi)ample.
	We say that it is normal if $\mathcal{X}$ is normal.
	Remark that the $\comp^*$-action induces the canonical isomorphisms
	\begin{equation}\label{eq_can_ident_test}
		\mathcal{X} \setminus X_0 \simeq X \times \comp^*, \qquad \mathcal{L}|_{\mathcal{X} \setminus X_0} \simeq \pi^* L.
	\end{equation}
	\par 
	Let us construct a submultiplicative filtration  $\mathcal{F}_{\mathcal{T}}$ on $R(X, L)$ associated with a test configuration $\mathcal{T}$ as follows.
	Pick an element $s \in H^0(X, L^{\otimes k})$, $k \in \nat^*$, and consider the section $\tilde{s} \in H^0(\mathcal{X} \setminus X_0, \mathcal{L})$, obtained by the application of the $\comp^*$-action to $s$.
	By the flatness of $\pi$, the section $\tilde{s}$ extends to a meromorphic section over $\mathcal{X}$, cf. Witt Nystr{\"o}m \cite[Lemma 6.1]{NystOkounTest}.
	In other words, there is $k \in \integ$, such that for a coordinate $z$ on $\comp$, we have $\tilde{s} \cdot z^{k} \in H^0(\mathcal{X}, \mathcal{L})$.
	\par 
	We define the filtration $\mathcal{F}_{\mathcal{T}}$ as follows
	\begin{equation}
		\mathcal{F}_{\mathcal{T}}^{\lambda} H^0(X, L^{\otimes k})
		:=
		\Big\{
			s \in H^0(X, L^{\otimes k}) : \tilde{s} \cdot z^{- \lceil \lambda \rceil} \in H^0(\mathcal{X}, \mathcal{L})
		\Big\}.
	\end{equation}
	As it was observed in \cite[(9)]{NystOkounTest}, the associated graded algebra of this filtration can be identified with the section ring of the central fiber of the fibration, $R(X_0, \mathcal{L}|_{X_0})$, endowed with the bigrading coming from the associated $\comp^*$-action and the natural grading of the section ring.
	From this observation and the fact that the bigraded ring $R(X_0, \mathcal{L}|_{X_0})$ is finitely generated for relatively ample $\mathcal{L}$, we conclude that the bigraded algebra associated with a filtration $\mathcal{F}_{\mathcal{T}}$ on $R(X, L)$ of an ample test configuration $\mathcal{T}$ is finitely generated as well.
	\par 
	In fact, Rees construction implies that any filtration with integer weights, for which the associated bigraded algebra is finitely generated, arises from an ample test configuration, see Sz{\'e}kelyhidi \cite[\S 3.1]{SzekeTestConf} or Boucksom-Jonsson \cite[\S A.2]{BouckJohn21}.
	In particular, for any filtration $\mathcal{F}$ with integer weights on $R(X, L)$, the filtrations $\mathcal{F}_{(k)}$, $k \in \nat^*$, are associated with ample test configurations.
	\par 
	Remark that finite generatedness of the bigraded algebra associated with the filtration $\mathcal{F}_{\mathcal{T}}$ on $R(X, L)$ implies that the filtration $\mathcal{F}_{\mathcal{T}}$ is bounded, see also Phong-Sturm \cite[Lemma 4]{PhongSturmTestGeodK}.
	\par 
	Now, consider a geodesic ray $h^L_t$, $t \in [0, +\infty[$, emanating from a bounded psh metric $h^L_0$.
	As in (\ref{eq_defn_hat_u}), we construct a metric $\hat{h}^L$ on $\pi^* L$ over $X \times \mathbb{D}^*$.
	Consider a test configuration $\mathcal{T} := (\mathcal{X}, \mathcal{L})$, $\pi : \mathcal{X} \to \comp$ and take its restriction to a unit disc $\pi_{\mathbb{D}} : \mathcal{X}_{\mathbb{D}} \to \mathbb{D}$, $\mathcal{L}_{\mathbb{D}} := \mathcal{L}|_{\mathcal{X}_{\mathbb{D}}}$.
	\begin{thm}[{Berman-Boucksom-Jonsson \cite[Lemmas 4.4, 5.3 and Corollary 6.7]{BerBouckJonYTD} and Boucksom-Jonsson \cite[Lemma A.12]{BouckJohn21}}]\label{thm_bbj_compat}
		Assume that $\mathcal{T}$ is ample and normal.
		Then, taking into account identification (\ref{eq_can_ident_test}), the metric $\hat{h}^L$ extends as a bounded psh metric to $\mathcal{L}_{\mathbb{D}}$ if and only if $h^L_t$ is a maximal geodesic ray with respect to the non-Archimedean potential $FS(\mathcal{F}_{\mathcal{T}})$.
	\end{thm}
	\par 
	Let us now give, following Phong-Sturm \cite{PhongSturmDirMA}, a construction of geodesic rays associated with an ample test configuration $\mathcal{T} = (\mathcal{X}, \mathcal{L})$ through the solution of the Dirichlet problem for a Monge-Ampère equation.
	Consider the test configuration $\tilde{\mathcal{T}} = (\tilde{\mathcal{X}}, \tilde{\mathcal{L}})$ given by the normalization of a fixed ample test configuration $\mathcal{T}$.
	Consider a $\comp^*$-equivariant resolution $p : \mathcal{X}' \to \tilde{\mathcal{X}}$ of $\tilde{\mathcal{X}} \to \comp$ and denote $\mathcal{L}' := p^* \tilde{\mathcal{L}}$.
	Consider the restriction $\pi: \mathcal{X}'_{\mathbb{D}} \to \mathbb{D}$ of $\pi: \mathcal{X}' \to \comp$ to the unit disc $\mathbb{D}$ and denote $\mathcal{L}'_{\mathbb{D}} := \mathcal{L}|_{\mathcal{X}'_{\mathbb{D}}}$.
	Phong-Sturm in \cite[Theorem 3]{PhongSturmDirMA} established that for any fixed smooth positive metric $h^L_0$ on $L$, there is a rotation invariant bounded psh metric $h^{\mathcal{L}'}$ over $\mathcal{L}'_{\mathbb{D}}$, verifying the Monge-Ampère equation
	\begin{equation}\label{eq_ma_geod_dir}
		c_1(\mathcal{L}', h^{\mathcal{L}'})^{n + 1} = 0,
	\end{equation}
	and such that its restriction over $\partial \tilde{\mathcal{X}}_{\mathbb{D}}$ coincides with the rotation-invariant metric obtained from the fixed metric $h^L_0$ on $L$.
	Under the identification (\ref{eq_can_ident_test}), we then construct a geodesic ray $h^{L, MA}_t$, $t \in [0, + \infty[$, such that $\hat{h}^{L, MA} = h^{\mathcal{L}'}$ in the notations (\ref{eq_defn_hat_u}).
	\par
	Recall that Phong-Sturm in \cite[Theorem 5]{PhongSturmRegul} established that there is a unique bounded psh solution to (\ref{eq_ma_geod_dir}).
	Since any two $\comp^*$-equivariant resolutions can be dominated by a third one, and a pull-back of a solution (\ref{eq_ma_geod_dir}) from one resolution will be a solution on another resolution, by the same uniqueness theorem, the geodesic ray $h^{L, MA}_t$, $t \in [0, + \infty[$, is independent of the choice of the $\comp^*$-equivariant resolution.
	We call $h^{L, MA}_t$ the \textit{Monge-Ampère geodesic ray}.
	\begin{thm}[{\cite[Theorem 3]{PhongSturmRegul}}]\label{thm_ph_st_compat}
		For any ample test configuration $\mathcal{T}$ of $(X, L)$, the Monge-Ampère geodesic ray emanating from a fixed smooth positive metric on $L$ coincides with the respective Bergman geodesic ray associated to the filtration $\mathcal{F}_{\mathcal{T}}$ arising from $\mathcal{T}$.
	\end{thm}
	\begin{proof}[Proof of Theorem \ref{thm_ray_coinc_test}]
		From Theorem \ref{thm_bbj_compat}, the maximal geodesic ray extends to the normalization of the test configuration as a bounded psh metric over the pull-back of $\mathcal{L}$. 
		Therefore after pulling it back to an equivariant resolution, we get a solution of (\ref{eq_ma_geod_dir}).
		Hence, by the unicity result of Phong-Sturm \cite[Theorem 5]{PhongSturmRegul}, the maximal geodesic ray coincides with the Monge-Ampère geodesic ray.
		The result now follows from Theorem \ref{thm_ph_st_compat} and Proposition \ref{prop_two_norms_comp}.
	\end{proof}

\bibliography{bibliography}

		\bibliographystyle{abbrv}

\Addresses

\end{document}